\documentclass{amsart}

\usepackage{graphicx}
\usepackage{latexsym}
\usepackage{amsfonts}
\usepackage{amscd}
\usepackage{amssymb}
\usepackage{amsmath}
\usepackage{amsthm}
\usepackage{bm}
\usepackage[all]{xy}
\usepackage[lmargin=1.25in,rmargin=1.25in,tmargin=1.25in,bmargin=1in,dvips]{geometry}
\usepackage{pinlabel}
\usepackage{paralist}
\usepackage{mathtools}
\usepackage{enumitem}
\usepackage[normalem]{ulem}
\usepackage[unicode, linktocpage, bookmarksnumbered=true]{hyperref}
\usepackage{stmaryrd}

\usepackage{caption,subcaption}

\usepackage{tikz}
\usetikzlibrary{arrows,positioning}

\newtheorem{theorem}{Theorem}[section]
\newtheorem{lemma}[theorem]{Lemma}

\newtheorem{corollary}[theorem]{Corollary}
\newtheorem{proposition}[theorem]{Proposition}

\theoremstyle{definition}
\newtheorem{definition}[theorem]{Definition}

\newtheorem{remark}[theorem]{Remark}

\newtheorem{example}[theorem]{Example}

\def\R{\mathbb{R}}

\def\Z{\mathbb{Z}}

\def\F{\mathbb{F}}

\def\CC{\mathcal{C}}
\def\FF{\mathcal{F}}

\numberwithin{equation}{section}

\def\Khred{\widetilde{\operatorname{Kh}}}

\def\Kh{\operatorname{Kh}}
\def\CKh{\operatorname{CKh}}

\def\HF {\widehat{\operatorname{HF}}}

\def\KhBS{\operatorname{Kh_{\mathit{BS}}}}

\newcommand{\abs}[1] {\left\lvert #1 \right\rvert}

\def\Th{^{\text{th}}}
\def\minus{\smallsetminus}
\def\co{\colon\thinspace}

\newcommand{\Brackets}[1] { \left\llbracket #1 \right\rrbracket }

 \DeclareMathOperator{\id}{id}

 \DeclareMathOperator{\nbd}{nbd}

\DeclareMathOperator{\Mod}{Mod}

 \DeclareMathOperator{\Kom}{Kom} \DeclareMathOperator{\CKom}{CKom} \DeclareMathOperator{\CKob}{CKob} \DeclareMathOperator{\Mat}{Mat} \DeclareMathOperator{\Cob}{Cob}
\DeclareMathOperator{\Kob}{Kob} \DeclareMathOperator{\Mor}{Mor} \DeclareMathOperator{\Obj}{Obj}

\DeclareMathOperator{\intt}{int}
\DeclareMathOperator{\ext}{ext}
\DeclareMathOperator{\sep}{sep}
\DeclareMathOperator{\splitt}{split}

\def\KJ{\mathit{KJ}}

\newcommand{\tikzbox}[1]{\vcenter{\hbox{#1}}}

\renewcommand{\MR}[1]{}

\numberwithin{equation}{section}

\begin{document}
\begin{abstract}
In this paper, we study the (in)sensitivity of the Khovanov functor to four-dimensional linking of surfaces. We prove that if $L$ and $L'$ are split links, and $C$ is a cobordism between $L$ and $L'$ that is the union of disjoint (but possibly linked) cobordisms between the components of $L$ and the components of $L'$, then the map on Khovanov homology induced by $C$ is completely determined by the maps induced by the individual components of $C$ and does not detect the linking between the components. As a corollary, we prove that a strongly homotopy-ribbon concordance (i.e., a concordance whose complement can be built with only 1- and 2-handles) induces an injection on Khovanov homology, which generalizes a result of the second author and Zemke. Additionally, we show that a non-split link cannot be ribbon concordant to a split link.
\end{abstract}

\title{Khovanov homology and cobordisms between split links}
\author{Onkar Singh Gujral}
\address{Department of Mathematics, Massachusetts Institute of Technology, Cambridge, MA 02139, USA}
\email{onkar@mit.edu}

\author{Adam Simon Levine}
\address{Department of Mathematics, Duke University, Durham, NC 27705, USA}
\email{alevine@math.duke.edu}

\maketitle

\tikzstyle{basepoint}=[circle,fill=black,minimum height=4pt,inner sep=0pt, outer sep=0pt, style={transform shape=false}]
\tikzstyle{crossing}=[circle,fill=white,minimum height=5pt,inner sep=0pt, outer sep=0pt, style={transform shape=false}]

\tikzstyle{blackdot}=[circle,fill=black,minimum height=4pt,inner sep=0pt, outer sep=0pt, style={transform shape=false}]
\tikzstyle{greendot}=[circle,fill=green!50!black,minimum height=4pt,inner sep=0pt, outer sep=0pt, style={transform shape=false}]
\tikzstyle{bluedot}=[circle,fill=blue,minimum height=4pt,inner sep=0pt, outer sep=0pt, style={transform shape=false}]
\tikzstyle{blackempty}=[circle,draw,thick,black,fill=white,minimum height=4pt,inner sep=0pt, outer sep=0pt, style={transform shape=false}]
\tikzstyle{greenempty}=[circle,draw,thick,green!50!black,fill=white,minimum height=4pt,inner sep=0pt, outer sep=0pt, style={transform shape=false}]

\newcommand{\crossing}[1][1]{
\begin{tikzpicture}[scale=#1]
      \draw[thick] (0,1) -- (1,0);
      \node[crossing] at (.5,.5) {};
      \draw[thick] (0,0) -- (1,1);
\end{tikzpicture}}
\newcommand{\vres}[1][1]{
    \begin{tikzpicture}[scale=#1]
      \draw[thick] (0,0) .. controls (.4,.4) and (.4,.6) .. (0,1);
      \draw[thick] (1,0) .. controls (.6,.4) and (.6,.6) .. (1,1);
    \end{tikzpicture}}
\newcommand{\hres}[1][1]{
    \begin{tikzpicture}[scale=#1]
      \draw[thick] (0,0) .. controls (.4,.4) and (.6,.4) .. (1,0);
      \draw[thick] (0,1) .. controls (.4,.6) and (.6,.6) .. (1,1);
    \end{tikzpicture}}
\newcommand{\vsaddle}[1][1]{
    \begin{tikzpicture}[scale=#1]
      \draw[thick] (0,0) .. controls (.4,.4) and (.4,.6) .. (0,1);
      \draw[thick] (1,0) .. controls (.6,.4) and (.6,.6) .. (1,1);
      \draw[] (0.3 ,0.5) -- (0.7,0.5);
    \end{tikzpicture}}
\newcommand{\hsaddle}[1][1]{
    \begin{tikzpicture}[scale=#1]
      \draw[thick] (0,0) .. controls (.4,.4) and (.6,.4) .. (1,0);
      \draw[thick] (0,1) .. controls (.4,.6) and (.6,.6) .. (1,1);
      \draw[] (0.5 ,0.3) -- (0.5,0.7);
    \end{tikzpicture}}

\newcommand{\sphere}[1][1]{
\begin{tikzpicture}[scale=#1]
    \draw (0,0) circle (1);
    \draw (1,0) arc (0:-180:1 and 0.3);
    \draw[dashed] (1,0) arc (0:180:1 and 0.3);
\end{tikzpicture}
}

\newcommand{\dottedsphere}[1][1]{
\begin{tikzpicture}[scale=#1]
    \draw (0,0) circle (1);
    \draw (1,0) arc (0:-180:1 and 0.3);
    \draw[dashed] (1,0) arc (0:180:1 and 0.3);
    \node[] at (0,0.65) {$\bullet$};
\end{tikzpicture}
}

\newcommand{\unknot}[1][1]{
\begin{tikzpicture}[scale=#1]
  \draw[thick] (0,0) circle (0.5);
\end{tikzpicture}
}

\newcommand{\birth}[1][1]{
\begin{tikzpicture}[scale=#1]
\begin{scope}[rotate=45];
\draw[thick] (0,0) circle (0.5);
\draw (.5,0) -- (.75,0);
\draw (-.5,0) -- (-.75,0);
\draw (0,.5) -- (0,.75);
\draw (0,-.5) -- (0,-.75);
\end{scope}
\end{tikzpicture}
}

\newcommand{\death}[1][1]{
\begin{tikzpicture}[scale=#1]
\begin{scope}[rotate=45];
\draw[thick] (0,0) circle (0.5);
\draw (.5,0) -- (.25,0);
\draw (-.5,0) -- (-.25,0);
\draw (0,.5) -- (0,.25);
\draw (0,-.5) -- (0,-.25);
\end{scope}
\end{tikzpicture}
}

\newcommand{\torus}[1][1]{
\begin{tikzpicture}[scale=#1]
    \draw (0,0) ellipse (1.5 and 1);
    \draw (-.9,0) .. controls (-.5, -.5) and (.5,-.5) .. (.9,0);
    \draw (-.75,-.15) .. controls (-.5, .5) and (.5,.5) .. (.75,-.15);
\end{tikzpicture}
}

\newsavebox{\vsaddlebox}
\sbox{\vsaddlebox}{\vsaddle[0.25]}%

\newsavebox{\hsaddlebox}
\sbox{\hsaddlebox}{\hsaddle[0.25]}%

\newcommand{\poscr}[1][1]{
\begin{tikzpicture}[scale=#1]
      \draw [->] (1,0) -- (0,1);
      \node[crossing] at (.5,.5) {};
      \draw [->](0,0) -- (1,1);
\end{tikzpicture}}

\newcommand{\negcr}[1][1]{
\begin{tikzpicture}[scale=#1]
      \draw [->](0,0) -- (1,1);
      \node[crossing] at (.5,.5) {};
      \draw [->] (1,0) -- (0,1);
\end{tikzpicture}}

\newcommand{\vshadedposcr}[1][1]{
\begin{tikzpicture}[scale=#1]
      \filldraw[black!20!white] (0,1) -- (.5,.5) -- (1,1);
      \filldraw[black!20!white] (0,0) -- (.5,.5) -- (1,0);
      \draw [->, thick] (1,0) -- (0,1);
      \node[crossing] at (.5,.5) {};
      \draw [->,thick](0,0) -- (1,1);
\end{tikzpicture}}

\newcommand{\hshadedposcr}[1][1]{
\begin{tikzpicture}[scale=#1]
      \filldraw[black!20!white] (0,1) -- (.5,.5) -- (0,0);
      \filldraw[black!20!white] (1,1) -- (.5,.5) -- (1,0);
      \draw [->, thick] (1,0) -- (0,1);
      \node[crossing] at (.5,.5) {};
      \draw [->,thick](0,0) -- (1,1);
\end{tikzpicture}}

\newcommand{\vshadednegcr}[1][1]{
\begin{tikzpicture}[scale=#1]
      \filldraw[black!20!white] (0,1) -- (.5,.5) -- (1,1);
      \filldraw[black!20!white] (0,0) -- (.5,.5) -- (1,0);
      \draw [->,thick](0,0) -- (1,1);
      \node[crossing] at (.5,.5) {};
      \draw [->,thick] (1,0) -- (0,1);
\end{tikzpicture}}

\newcommand{\hshadednegcr}[1][1]{
\begin{tikzpicture}[scale=#1]
      \filldraw[black!20!white] (0,1) -- (.5,.5) -- (0,0);
      \filldraw[black!20!white] (1,1) -- (.5,.5) -- (1,0);
      \draw [->,thick](0,0) -- (1,1);
      \node[crossing] at (.5,.5) {};
      \draw [->,thick] (1,0) -- (0,1);
\end{tikzpicture}}

\section{Introduction} \label{sec:introduction}

Khovanov homology is a link homology theory, discovered by Mikhail Khovanov \cite{Khov}, that is functorial under link cobordisms. It associates to any link $L \subset \R^3$ a bigraded group called $\Kh(L)$ (and, more generally, for any commutative ring $R$, a bigraded $R$-module $\Kh(L;R)$). Further, a smoothly embedded cobordism $C$ in $\R^3 \times [0,1]$ between links $L_0$ and $L_1$ (i.e., a smoothly embedded, oriented surface in $\R^3 \times [0,1]$ with oriented boundary $-L_0 \times \{0\} \cup L_1 \times \{1\}$) induces a map $\Kh(C):\Kh(L_0;R)\to \Kh(L_1;R)$ that is an invariant of $C$ under smooth isotopies rel boundary. In this paper, we shall prove that the Khovanov functor is insensitive to the $4$-dimensional linking of cobordisms between split links, in a sense that we now make precise.

\begin{definition} \label{def: splitting}
A \emph{splitting} of a link $L$ is a decomposition $L=L^1 \cup \dots \cup L^k$, where $L^1, \dots,L^k$ are links, such that $L^1, \dots,L^k$ are contained in disjoint 3-balls $B^1, \dots, B^k \subset \R^3$. $L^1, \dots,L^k$ are called the \emph{parts} of the splitting. A \emph{total splitting} is a splitting in which each of the parts is a knot.

Given links $L_0$ and $L_1$ with partitions $L=L_0^1\cup  \dots  \cup L_0^k$ and $L_1=L_1^1\cup  \dots  \cup L_1^k$ (which need not be splittings), a cobordism $C \subset \R^3 \times [0,1]$ is called \emph{partition-preserving} if it is a disjoint union of surfaces $C^1, \dots, C^k$ (the \emph{parts} of $C$), where $C^i$ is a cobordism from $L_0^i$ to $L_1^i$. The cobordisms $C^i$ are not required to be connected. We say the cobordism is \emph{split} if $L_0$ and $L_1$ are split links (using the same family of $3$-balls $B^1, \dots, B^k$) and each $C^i$ is contained in $B^i \times [0,1]$.

Two partition-preserving cobordisms $C = C^1 \cup \dots \cup C^k$ and $D = D^1 \cup \dots \cup D^k$ from $L_0$ to $L_1$ are called \emph{partition-homotopic} if there is a smooth homotopy from $C$ to $D$ (rel boundary) that is an isotopy when restricted to each of the components $C^i$. This is equivalent to the statement that for each $i$, $C^i$ is isotopic rel boundary to $D^i$.
\end{definition}

The following is the main theorem of the paper.

\begin{theorem} \label{thm: main}
Let $L_0$ and $L_1$ be links with splittings into $k$ parts, and let $C$ and $D$ be partition-preserving cobordisms from $L_0$ to $L_1$ that are partition-homotopic. Then the maps $\Kh(C)$ and $\Kh(D)$ are equal up to an overall sign.
\end{theorem}

The Khovanov homology of a split link satisfies a K\"unneth theorem under disjoint unions \cite[Proposition 33]{Khov}. For simplicity, if $R$ is a field, then for any split link $L = L^1 \cup \dots \cup L^k$, we have
\begin{equation} \label{eq: Kh-tensor}
\Kh(L;R) \cong \Kh(L^1;R) \otimes_R \dots \otimes_R \Kh(L^k;R).
\end{equation}
Similarly, for a split cobordism $C = C^1 \cup \dots \cup C^k$ between split links $L_0 = L_0^1 \cup \dots \cup L_0^k$ and $L_1 = L_1^1 \cup \dots \cup L_1^k$, the induced maps respect the tensor product structure, in the sense that
\begin{equation} \label{eq: Kh-cob-tensor}
\Kh(C) = \pm \Kh(C^1) \otimes \dots \otimes \Kh(C^k).
\end{equation}
Our main technical result is that \eqref{eq: Kh-cob-tensor} holds not only for a split cobordism but also for an arbitrary partition-preserving cobordism between split links. If two such cobordisms are partition-homotopic, equation \eqref{eq: Kh-cob-tensor} coupled with isotopy invariance of Khovanov maps then implies Theorem \ref{thm: main}. (Over an arbitrary ring, there are also Tor terms in \eqref{eq: Kh-tensor}, but the same argument goes through at the chain complex level.)

One consequence of Theorem \ref{thm: main} concerns cobordisms with closed components:
\begin{corollary} \label{cor: closed-component}
Let $C \co L_0 \to L_1$ be any cobordism in $\R^3 \times [0,1]$. Let $S$ be a closed, connected, oriented surface in the complement of $C$; it may be nontrivially linked with $C$, and may be dotted. Then:
\begin{itemize}
\item If $S$ is an undotted sphere, then $\Kh(C \cup S) = 0$.

\item If $S$ is a dotted sphere, then $\Kh(C \cup S) = \pm \Kh(C)$.

\item If $S$ is an undotted torus, then $\Kh(C \cup S) = \pm 2 \Kh(C)$.

\item If $g(S)>1$, or if $S$ is dotted and $g(S)>0$, then $\Kh(C \cup S)=0$.
\end{itemize}
\end{corollary}
Each of these statements was previously known in the case where $S$ is unlinked from $C$ (i.e., contained in a $4$-ball disjoint from $C$), by results of Rasmussen \cite{RasmussenClosed} and Tanaka \cite{tanaka}. We deduce the general statement by applying Theorem \ref{thm: main} to see that the map is unchanged when we pull $S$ off of $C$. Thus we should interpret these statements as completing a $4$-dimensional lift of Bar-Natan's famous sphere, torus and dotted sphere relations (see Figures \ref{fig: local-nodots} and \ref{fig: local-relations} below).

\subsection{Applications to ribbon concordance}

One application of Theorem \ref{thm: main}, which was the original motivation behind this project, concerns \emph{strongly homotopy-ribbon concordance}. A \emph{concordance} from $L_0$ to $L_1$ is a cobordism $C \subset \R^3 \times [0,1]$ consisting of disjoint annuli $C^1, \dots, C^k$, each connecting a component of $L_0$ to a component of $L_1$.
Following \cite{MillerZemkeHomotopyRibbon}, a concordance $C$ is called:
\begin{itemize}
\item \emph{ribbon} if projection onto the $[0,1]$ factor, restricted to $C$, is a Morse function with only index $0$ and $1$ critical points;

\item \emph{strongly homotopy-ribbon} if the complement of $C$ in $S^3 \times [0,1]$ can be built from $(S^3 \minus \nbd(L_0)) \times [0,1]$ by adding $4$-dimensional $1$- and $2$-handles only;

\item \emph{homotopy-ribbon} if the induced map $\pi_1(S^3 \minus L_0) \to \pi_1((S^3 \times [0,1]) \minus C)$ is injective and the induced map $\pi_1(S^3 \minus L_1) \to \pi_1((S^3 \times [0,1]) \minus C)$ is surjective.
\end{itemize}
Here, we have implicitly identified $S^3$ with $\R^3 \cup \{\infty\}$. Gordon \cite{GordonRibbon} showed that the implications
\[
\text{ribbon} \Rightarrow \text{strongly homotopy-ribbon} \Rightarrow \text{homotopy-ribbon}
\]
both hold. (Showing injectivity on $\pi_1$ is the most difficult piece of the argument; it relies on major theorems of Gerstenhaber--Rothaus \cite{GerstenhaberRothaus} and Thurston \cite{Thurston3ManifoldKleinian}.) It is unknown whether the reverse implications hold, however. Furthermore, Gordon conjectured (and proved under some added hypotheses) that if there are ribbon concordances from $K_0$ to $K_1$ and from $K_1$ to $K_0$ (where $K_0$ and $K_1$ are knots), then $K_0$ and $K_1$ must in fact be isotopic. Philosophically, if there is a ribbon concordance from $L_0$ to $L_1$, one expects $L_0$ to be ``no more complicated'' than $L_1$, as measured by a variety of invariants.

In 2019, Zemke \cite{ZemkeRibbon} proved that any ribbon concordance induces an injection on knot Floer homology. Shortly thereafter, Zemke and the second author \cite{LevineZemke} proved an analogous result for Khovanov homology, using a key topological lemma proved in \cite{ZemkeRibbon} in conjunction with some of the formal properties for Khovanov homology established by Bar-Natan \cite{BarNatanTangles}. A few months later, Miller and Zemke \cite{MillerZemkeHomotopyRibbon} extended Zemke's original result for knot Floer homology to the \emph{a priori} weaker hypothesis of a strongly homotopy-ribbon concordance. Corollary \ref{cor: closed-component} allows us to obtain the corresponding result in the Khovanov setting:

\begin{theorem} \label{thm: SHR}
Let $S\co L\to L'$ be a strongly homotopy-ribbon concordance. Then the map induced on Khovanov homology $\Kh(S)$ is injective.
\end{theorem}

The definition of strongly homotopy-ribbon is purely in terms of the topology of the complements of the links and the concordance, with no reference to link diagrams or movies. In contrast, Khovanov homology is defined entirely in terms of the combinatorics of link diagrams, and its connection to the topology of link complements is still largely mysterious. Theorem \ref{thm: SHR} is thus an unusual instance in which a non-diagrammatic property has implications for Khovanov homology. Furthermore, Theorem \ref{thm: SHR} implies that the various corollaries presented in \cite{LevineZemke} also apply for a strongly homotopy-ribbon concordance.

We also prove an additional, purely topological result concerning ribbon concordance of links:

\begin{theorem} \label{thm: ribbon-split}
If $L$ is strongly homotopy-ribbon concordant to $L'$, and $L'$ is split, then $L$ is split. More precisely, if there is a $2$-sphere in $S^3 \minus L'$ separating two components of $L'$, then there is a $2$-sphere in $S^3 \minus L$ separating the corresponding components of $L$.
\end{theorem}

This follows from the injectivity results mentioned above (\cite{LevineZemke} for ribbon concordance, Theorem \ref{thm: SHR} for strongly homotopy-ribbon concordance) along with recent work by Lipshitz and Sarkar \cite{LipSar} showing that Khovanov homology detects split links. We also give a second proof using Heegaard Floer homology, making use of a similar injectivity result due to Daemi, Lidman, Vela-Vick, and Wong \cite{DaemiLidmanVelaVickWong}. This result provides another example of nondecreasing simplicity under ribbon concordance, as mentioned above.

\subsection{Batson--Seed homology}

The proof of Theorem \ref{thm: main} makes use of a variant of Khovanov homology due to Batson and Seed \cite{BatsonSeed}. Let $R$ be a commutative ring with unit. A \emph{weighted link} is a link $L$ along with a choice of element of $R$ associated to each component of $L$; call this data $w$. Given a weighted link $(L,w)$, Batson and Seed defined a perturbation of the differential on the Khovanov complex, leading to a homology theory which we call $\KhBS(L,w)$. The Batson--Seed complex is filtered, with associated graded complex equal to the original Khovanov complex (with coefficients in $R$); thus, there is a spectral sequence from $\Kh(L;R)$ to $\KhBS(L,w)$. If $L$ is a totally split link (or, more generally, a partially split link in which the weights of all components of each part agree), then the spectral sequence collapses, and $\KhBS(L,w) \cong \Kh(L;R)$. The key property of the Batson--Seed complex is that up to isomorphism it is unchanged (up to an overall grading shift) upon reversing crossings between strands the difference of whose weights is invertible in $R$. Thus, if all such differences are invertible, the invariant $\KhBS(L,w)$ depends only on the Khovanov homology of the individual components of $L$.

The bulk of this paper is devoted to developing a theory of tangle invariants, cobordism maps, and functoriality for Batson--Seed homology. A \emph{weighted cobordism} between weighted links is a link cobordism along with a choice of weights on each component, agreeing with the weights on the boundary. We prove:

\begin{theorem} \label{thm: BS-func}
Let $(L_0, w_0)$ and $(L_1, w_1)$ be weighted links.
\begin{enumerate}
\item \label{item: BS-func-maps}
Any weighted cobordism $(C,w)$ from $(L_0, w_0)$ to $(L_1,w_1)$ induces a filtered chain map on Batson--Seed complexes, giving an induced homomorphism
\[
\KhBS(C,w) \co \KhBS(L_0,w_0) \to \KhBS(L_1,w_1).
\]
The induced map on associated graded objects agrees with $\Kh(C) \co \Kh(L_0;R) \to \Kh(L_1;R)$.

\item \label{item: BS-func-isotopy}
If $C$ and $D$ are isotopic rel boundary and are equipped with the same weighting, then the filtered chain maps associated to $C$ and $D$ are filtered chain homotopic up to an overall sign. Thus, $\KhBS(C,w)$ and $\KhBS(D,w)$ agree up to sign, as do the induced morphisms of spectral sequences.

\item \label{item: BS-func-homotopy}
If $C$ and $D$ are partition-homotopic and are equipped with the same weighting, and all differences of weights between distinct components are invertible in $R$, then the chain maps associated to $C$ and $D$ are chain homotopic up to multiplication by a unit in $R$, although not necessarily in a filtered sense. Thus, the induced maps $\KhBS(L_0,w_0) \to \KhBS(L_1,w_1)$ agree up to a unit in $R$.
\end{enumerate}
\end{theorem}

En route to proving this theorem, we also develop an analogue of the Batson--Seed complex for tangles, taking values in a category of formal curved chain complexes of diagrams, along the lines of Bar-Natan's ``tangles and cobordisms'' story for Khovanov homology \cite{BarNatanTangles}. (Batson and Seed alluded to the possibility of a such a construction in their original paper.) In the case of a cobordism between split links, the proof of Theorem \ref{thm: main} then follows from studying how the crossing change isomorphisms interact with the cobordism maps, using the fact that the Khovanov and Batson--Seed invariants agree for split links.

\begin{remark}
We briefly discuss the ``up to sign'' provisos in the theorems above. Originally, the cobordism maps on Khovanov homology were only shown to be isotopy-invariant up to an overall sign \cite{KhovanovCobordism, JacobssonCobordisms, BarNatanTangles}. Subsequent work by Caprau \cite{CaprauCobordisms} and Clark--Morrison--Walker \cite{ClarkMorrisonWalkerFunctoriality} modified the construction of the cobordism maps so as to eliminate the sign ambiguity. Both of these approaches require expanding the coefficient ring to include $i$ (the primitive fourth root of unity) and keeping track of more topological data than just oriented cobordisms. (A recent preprint by Sano \cite{SanoFunctoriality} provides an alternate approach that requires adjusting the signs of the maps.) Because the injectivity statement of Theorem \ref{thm: SHR} does not require pinning down the sign, we opted to stick with the simpler framework from \cite{BarNatanTangles}, at the cost of maintaining the sign indeterminacy.
\end{remark}

\subsection*{Organization}

In Section \ref{sec: algebra}, we describe the algebraic setup that will be used throughout the paper. In Section \ref{sec: BScomplex}, we use this structure to define a version of the Batson--Seed complex for tangles. Sections \ref{sec: reidemeister} and \ref{sec: cobordism-maps} establish the invariance and functoriality of this theory, and Section \ref{sec: crossing-change} describes the analogue of Batson--Seed's crossing change isomorphisms. In Section \ref{sec: main-theorem}, we then assemble these ingredients to prove Theorems \ref{thm: main} and \ref{thm: BS-func}, along with some generalizations to tangles. Finally, we prove the applications to concordance (Theorems \ref{thm: SHR} and \ref{thm: ribbon-split}) in Section \ref{sec: SHR}.

\subsection*{Acknowledgments}

The authors are deeply grateful to Nathan Dowlin, Mikhail Khovanov, Gage Martin, Maggie Miller, Sucharit Sarkar, Radmila Sazdanovic, and Ian Zemke for helpful conversations, as well as to the referee for a very careful reading of our paper and insightful suggestions.

The first author was supported in summer 2019 by Duke University's Program for Research for Undergraduates (PRUV). He thanks David Kraines for organising the PRUV program. He also thanks Lenhard Ng for serving on his senior thesis committee and for his valuable comments on an early draft of his senior thesis.

The second author was partially supported by NSF Topology grant DMS-1806437.

\section{Preliminaries} \label{sec: algebra}

\subsection{The Batson--Seed complex of a link} \label{ssec: BSlink}

We begin with a brief structural summary of the original Batson--Seed construction. We assume that the reader is familiar with the basics of Khovanov homology; see \cite{BarNatanCategorification} for an accessible overview.

Let $L$ be a $k$-component, oriented link diagram with $n$ crossings. Throughout the paper, let $R$ be a commutative ring with unit, and let $w = (w_1, \dots, w_k)$ be a weight assignment of an element of $R$ to each component of $L$. Let $n_+$ (resp.~$n_-$) denote the number of positive (resp.~negative) crossings. For each $v \in \{0,1\}^n$, let $\abs{v} = \sum_{i=1}^n v_i$, let $L_v$ be the corresponding resolution of $L$ according to the convention shown in Figure \ref{fig: resolutions}, and let $k_v$ be the number of components in $L_v$. Let $V$ denote the $R$-module $R[X]/(X^2)$, graded such that $\deg(1)=1$ and $\deg(X) = -1$. For $v,v' \in \{0,1\}^n$, we say that $v'$ is an \emph{immediate successor} of $v$, written $v \lessdot v'$, if $v'$ is obtained from $v$ by changing a single $0$ to a $1$.

\begin{figure}
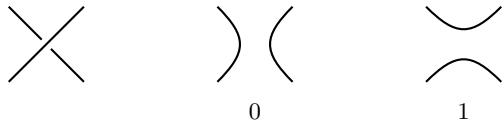

  \centering
  \begin{subfigure}[b]{0.15\textwidth}
    \centering
    \crossing
    \caption*{$ $}
  \end{subfigure}
  \hspace{0.1in}
  \begin{subfigure}[b]{0.15\textwidth}
    \centering
    \vres
    \caption*{$0$}
  \end{subfigure}
  \hspace{0.1in}
  \begin{subfigure}[b]{0.15\textwidth}
    \centering
    \hres
    \caption*{$1$}
  \end{subfigure}
\caption{The $0$- and $1$-resolutions of a crossing.}
\label{fig: resolutions}
\end{figure}

Let $C = \CKh(L;R) = \CKh(L) \otimes_\Z R$, the Khovanov complex of $L$ with coefficients in $R$. As an $R$-module, we have $C = \bigoplus_{v \in \{0,1\}^n} C_v$, where $C_v \cong V^{\otimes k_v}$. The complex $C$ comes equipped with a bigrading $(i,j)$, where:
\begin{itemize}
\item $i$ denotes the homological grading, where the summand $C_v$ sits in grading $\abs{v} - n_-$;

\item $j$ denotes the quantum grading, where the grading on $C_v$ comes from shifting the internal grading on $V^{\otimes k_v}$ by $\abs{v} + n_+ - 2n_-$.
\end{itemize}

Let $d_+ \co C \to C$ denote the standard Khovanov differential, which does not depend on $R$ or $w$. We do not dwell on the definition, other than to note three properties: First, $d_+$ is a sum of terms $d_+^{v,v'} \co C_v \to C_{v'}$ ranging over all $v \lessdot v'$, each determined by the merge and split maps in the Khovanov TQFT. Second, $d_+$ is homogeneous of degree $(+1,0)$ with respect to the $(i,j)$ bigrading. Third, we have $d_+^2=0$.

Batson and Seed defined a second differential $d_- \co C \to C$. This differential is the sum of ``backward'' terms $d_-^{v',v} \co C_{v'} \to C_v$, ranging over all $v \lessdot v'$, each of which is equal to the appropriate merge/split map times $\pm (w_a - w_b)$, where $w_a$ and $w_b$ are the weights of the components of $L$ at the crossing corresponding to the index that changes from $v$ to $v'$. (We will explain the sign conventions below.) This map is homogeneous of degree $(-1,-2)$. Batson and Seed proved that $d_-^2 = 0$ and that $d_- d_+ + d_+ d_- = 0$. It follows that the differential $d = d_+ + d_-$ satisfies $d^2 = 0$. Furthermore, $d$ is homogeneous of degree $-1$ with respect to the grading $l = j-i$, which we will refer to as the \emph{Batson--Seed grading} or $\emph{$l$-grading}$. Observe that the $l$-grading on each summand $C_v$ is obtained by shifting the internal grading on $V^{\otimes k_v}$ by $n_+ - n_-$; we will make ample use of this fact below.

While $d$ is not homogeneous with respect to the $j$ grading, this grading instead gives a filtration on $(C,d)$. The associated graded complex is then simply $(C,d_+)$, the ordinary Khovanov complex over $R$. We thus obtain a spectral sequence whose $E_2$ page is $\Kh(L;R)$ and which converges to $H_*(C,d)$. Batson and Seed prove that the filtered homotopy type of $(C,d)$ is invariant (as a singly-graded $R$-module) under Reidemeister moves; thus, the whole spectral sequence is a link invariant. We denote the total homology of $(C,d)$ by $\KhBS(L,w)$.

Finally, we review the crossing change map \cite[Section 2.4]{BatsonSeed}. Let $L'$ be obtained from $L$ by changing a single crossing $c$, and let $(C',d')$ denote the corresponding Batson--Seed complex, using the same weights $w$. Batson and Seed define a chain map $f \co C \to C'$ associated to the crossing change, which is homogeneous of degree $\pm 2$ with respect to the $l$ grading. If the difference of the weights of the components of $L$ involved at $c$ is invertible in $R$, then $f$ is in fact an isomorphism of chain complexes (with inverse given up to a unit by the reverse crossing change map). If all differences of weights are invertible, it then follows that $\KhBS(L,w)$ does not see the linking between components; it is isomorphic to the Khovanov homology of the totally split link obtained by separating the components of $L$. However, note that the crossing change map $f$ is neither homogeneous nor filtered with respect to the $j$ grading; it includes terms of degrees $-2$, $0$, and $2$. Using the latter property, Batson and Seed were able to use the spectral sequence from $\Kh(L;R)$ to $\KhBS(L,w)$ to obtain lower bounds on the number of crossing changes needed to split $L$.

\subsection{Homological algebra} \label{ssec: homological}

In order to establish functoriality for Batson--Seed homology, we will need a localized version of the theory for tangles, taking values in a category of formal chain complexes akin to Bar-Natan's work in \cite{BarNatanTangles}, but with modifications motivated by the structure of the Batson--Seed complex as described in the previous section. The objects of study will in fact be ``curved'' chain complexes in which $d^2 \ne 0$. Nevertheless, most standard properties from homological algebra continue to hold in this context, as stated below in Remark \ref{rmk: commonfacts}. Throughout let $R$ be a commutative ring with unit.

We begin by reminding the reader of several definitions from \cite{BarNatanTangles} (adapted for an arbitrary base ring $R$).

\begin{definition}
A category $\CC$ is \emph{$R$-pre-additive} if the set of morphisms between any two objects is an $R$-module, with the property that composition of morphisms is $R$-bilinear. We say $\CC$ is \emph{$R$-additive} if, in addition, it has a zero object and a well-defined direct sum operation.
\end{definition}

Note that if a category $\CC$ isn't $R$-pre-additive to begin with, it can be made so by considering instead the category in which the objects are the same as those of $\CC$, but the set of morphisms from the object $A$ to $B$ is given by the free $R$-module generated by $\Mor(A,B)$ of the original category $\CC$, with composition extended bilinearly. The category thus obtained is $R$-pre-additive.

\begin{definition} \label{def: Mat}
Following \cite[Definition 3.2]{BarNatanTangles}, for every $R$-pre-additive-category $\CC$, one can define an $R$-additive category $\Mat(\CC)$ with objects given by formal finite direct sums of objects of $\CC$, and a morphism $f$ from the object $A=\bigoplus _{j=1}^{n}A_j$ to the object $B=\bigoplus _{i=1}^{m}B_i$ given by an $m \times n$ matrix $f$, the $(i,j)$th entry of which is a morphism $A_j\to B_i$ of $\CC$. Composition in $\Mat(\CC)$ is given by matrix multiplication, where instead of multiplying numbers we compose the morphisms in $\CC$. It is easy to check that $\Mat(\CC)$ is $R$-additive, with the zero object given by the empty direct sum.
\end{definition}

\begin{definition} \label{def: graded}
As in \cite[Definition 6.1]{BarNatanTangles}, we say that an $R$-pre-additive category $\CC$ is \emph{graded} when it satisfies the following conditions:
\begin{enumerate}
\item  \label{item: graded-mor}For all objects $A,B$ in $\CC$, $\Mor(A,B)$ is a graded $R$-module, in a manner that \[\deg (f \circ g)= \deg (f)+ \deg (g)\] where applicable, and the degree of the identity is always $0$.

\item \label{item: graded-shift} There is a $\Z$-action on objects, denoted $(m, A) \mapsto A \{m\}$; and for any $A, B \in \Obj(\CC)$ and $m, n \in \Z$, $\Mor(A\{m\},B\{n\})$ is identified with $\Mor(A,B)$ as $R$-modules with a shift in the grading so that if $f \in \Mor(A,B)$ with $\deg(f) = d$, then the corresponding element of $\Mor(A\{m\},B\{n\})$ has degree $d+n-m$.
\end{enumerate}
\end{definition}

As explained in \cite[Section 6]{BarNatanTangles}, any $R$-pre-additive category satisfying \eqref{item: graded-mor} can be upgraded to satisfy \eqref{item: graded-shift} by formally introducing the shifted objects. Additionally, note that if $\CC$ is graded, then $\Mat(\CC)$ naturally inherits the structure of a graded category, where a matrix of morphisms is homogeneous of degree $d$ iff all of its entries are homogeneous of degree $d$.

\begin{definition} \label{def: Kom}
Following \cite[Definition 3.3]{BarNatanTangles}, for any $R$-additive category, we define the \emph{category of chain complexes} $\Kom(\CC)$ as follows. An object of $\Kom(\CC)$ consists of a sequence $(\Omega^r)_{r \in \Z}$, where $\Omega^r \in \Obj(\CC)$ and all but finitely many are $0$, along with morphisms $d^r \co \Omega^{r}\to \Omega^{r+1}$ satisfying $d^{r+1} \circ d^r = 0$ for all $r$. A morphism $f \co (\Omega_a, d_a) \to (\Omega_b, d_b)$ (a \emph{chain map}) is a family of morphisms $f^r \co \Omega_a^r\to \Omega_b^r$ such that $d_b^r \circ f^r = f^{r+1} \circ d_a^r$. As in standard homological algebra, the composition of morphisms is defined to be $(g \circ f)^r = g^r \circ f^r$.
\end{definition}

\begin{definition} \label{def: Kom-homotopy}
Following \cite[Definition 4.1]{BarNatanTangles}, given morphisms $f, g \co (\Omega_a, d_a) \to (\Omega_b, d_b)$, a \emph{homotopy} between $f$ and $g$ is a family of morphisms $h^r \co \Omega_a^r \to \Omega_b^{r-1}$ such that
\[
f^r - g^r = d_b^{r-1} \circ  h^r + h^{r+1} \circ d_a^r
\]
for all $r \in \Z$. We call $f$ and $g$ \emph{homotopic} if such a homotopy exists. A morphism $f \co (\Omega_a, d_a) \to (\Omega_b, d_b)$ is called a \emph{homotopy equivalence} if there exists a morphism $g \co (\Omega_b, d_b) \to (\Omega_a, d_a)$ such that $f \circ g$ and $g \circ f$ are each homotopic to the respective identity morphisms. The category $\Kom_{/h}(\CC)$ is defined to have the same objects as $\Kom(\CC)$, with morphisms given by homotopy classes of morphisms in $\Kom(\CC)$. Thus, a chain map $f$ is a homotopy equivalence iff it is an isomorphism in $\Kom_{/h}(\CC)$.
\end{definition}

We will now describe a modified version of $\Kom$ that is intended to mimic the structure of the Batson--Seed complex, as described in Section \ref{ssec: BSlink}.

\begin{definition} \label{def: CKom}
For any graded, $R$-additive category $\CC$, we define a  category $\CKom(\CC)$ as follows.

An object of $\CKom(\CC)$ consists of a sequence $(\Omega^r)_{r \in \Z}$, where $\Omega^r \in \Obj(\CC)$ and all but finitely many are $0$, along with morphisms $d^{r}_{+}\co \Omega^{r}\to \Omega^{r+1}$ and $d^{r}_{-}\co \Omega^{r}\to \Omega^{r-1}$ that are each homogeneous of degree $-1$ and satisfy $d_+^{r+1} \circ d_+^r = 0$ and $d_-^{r-1} \circ d_-^r = 0$ for all $r$. That is, $\Omega$ can be viewed as a chain of the form
\[
\xymatrix{
\cdots \ar@<.5ex>[r] &
\Omega^{r-1} \ar@<.5ex>[r]^{d_+^{r-1}} \ar@<.5ex>[l] &
\Omega^{r} \ar@<.5ex>[r]^{d_+^{r}} \ar@<.5ex>[l]^{d_-^{r}} &
\Omega^{r+1} \ar@<.5ex>[r] \ar@<.5ex>[l]^{d_-^{r+1}} &
\cdots \ar@<.5ex>[l]
}
\]
that is eventually zero on both ends. For convenience, we can combine the $d^{r}_{\pm}$ maps into a matrix $d_\pm \co \Omega \to \Omega$, where $\Omega$ is the formal direct sum of all $\Omega^r$, and write $d=d_+ +d_-$. Note that $d_+^2 = d_-^2 = 0$. However, $d^2 = d_+ d_- + d_- d_+$ need not be zero; it is a diagonal matrix consisting of the morphisms
\[
\lambda_r := d_+^{r-1} \circ d_-^r + d_-^{r+1} \circ d_+^r \co \Omega^r \to \Omega^r,
\]
which we call the \emph{curvature terms}. (See Remark \ref{rmk: curved} below regarding this terminology.) We refer to $r$ as the \emph{homological grading}.

Given objects $\Omega_a, \Omega_b$ in $\CKom(\CC)$, a morphism $f \co \Omega_a \to \Omega_b$ (a \emph{chain map}) is given by a family of morphisms $f^{r,r+2k} \co \Omega_a^r\to \Omega_b^{r+2k}$ for all $r,k\in \mathbb{Z}$ with the property that $fd=df$. We interpret the latter as a matrix equation, where $f$ is the matrix of all the terms $f^{r,r+2k}$; the finiteness assumption on the objects ensures that all entries of $df$ and $fd$ are finite sums. (Technically, to view $f$ as a matrix, we should include morphisms with odd homological degree shifts as well; we require these to be $0$.) The set of such chain maps is naturally an $R$-module. Composition of morphisms is given by the composition of their representative matrices, which again depends on finiteness; that is,
\[
(g \circ f)^{r, r+2k} = \sum_{l \in \Z} g^{r+2l, r+2k} \circ f^{r, r+2l}.
\]
The identity morphism of $\Omega$ is the direct sum of the identity morphisms of each $\Omega^r$.

For a morphism $f$, let us write $f = \sum_{n \in \Z} f_n$, where $f_n$ consists of the degree-$n$ terms of each of the terms $f^{r,r+2k}$ (which we denote by $f^{r,r+2k}_n$). Since all the terms in $d_+$ and $d_-$ are homogeneous of the same degree, each $f_n$ individually satisfies $f_n d = d f_n$. Thus, $\Mor(\Omega_a, \Omega_b)$ has the structure of a graded $R$-module. We may define a grading shift operator on $\Obj(\CKom(\CC))$ term-wise, giving $\CKom(\CC)$ the structure of a graded category as per Definition \ref{def: graded}. Whenever we refer to a morphism $f$ being \emph{homogeneous}, it is with respect to this grading unless otherwise specified.
\end{definition}

\begin{definition} \label{def: CKom-homotopy}
Given morphisms $f,g\co \Omega_a\to \Omega_b$ in $\CKom(\CC)$, a \emph{homotopy from $f$ to $g$} is a family of morphisms $h_{r,r+2k+1} \co \Omega_a^r\to \Omega_b^{r+2k+1}$ for all $r,k \in \Z$, such that $f-g=hd+dh$ (again interpreted as a matrix equation as above). If such a homotopy exists, we say $f$ and $g$ are homotopic. A morphism $f \co \Omega_a \to \Omega_b$ is called a \emph{homotopy equivalence} if there exists a morphism $g\co \Omega_b \to \Omega_a$ such that $f \circ g$ and $g \circ f$ are homotopic to the respective identity maps. As above, we define $\CKom_{/h}(\CC)$ to be the category whose objects are the same as $\CKom(\CC)$, and whose morphisms are homotopy classes of morphisms in $\Kom(\CC)$.
\end{definition}

\begin{remark} \label{rmk: commonfacts}
To check that $\CKom_{/h}(\CC)$ is actually a well-defined category, we must verify that standard facts of homological algebra continue to hold in $\CKom$: namely, that homotopy between maps is an equivalence relation, and that composition respects this relation. These follow from standard homological algebra arguments, interpreting differentials, homotopies, and chain maps as matrices. Likewise, the composition of homotopy equivalences is easily seen to be a homotopy equivalence.

Additionally, note that if $h_n$ denotes the grading $n$ part of a homotopy $n$ from $f$ to $g$ (consisting of all the degree-$n$ terms $h^{r,r+2k+1}_n$), then $h_n$ is a homotopy from $f_{n-1}$ to $g_{n-1}$. It follows that the grading of $\CKom(\CC)$ descends to $\CKom_{/h}(\CC)$.
\end{remark}

\begin{remark} \label{rmk: curved}
In our definition, $\CKom(\CC)$ is a hybrid between $\Kom(\CC)$ as above and the categories of \emph{matrix factorizations} or \emph{curved chain complexes}, which have appeared in various guises in the world of knot homologies; see, e.g., \cite{KhovanovRozansky1, ZemkeLinkFunctoriality}. The objects are equipped with a $\Z$-valued homological grading, as in $\Kom(\CC)$, but the morphisms (resp.~homotopies) are only required to preserve (resp.~shift) a $\Z/2$-valued grading, as in the category of matrix factorizations. However, unlike in the usual setting of matrix factorizations or curved chain complexes, we do not require the curvature to be a scalar multiple of the identity, and we allow for morphisms between objects with different curvatures. We have opted for the term ``curved chain complex'' (which is somewhat less standardized in the literature) and the notation $\CKom(\CC)$ to avoid this ambiguity. We will see that the Batson--Seed complex from Section \ref{ssec: BSlink}, which is an honest chain complex with $d^2=0$, can be obtained by ``tensoring'' together multiple formal complexes with nontrivial curvature, which mimics the pattern seen in \cite{KhovanovRozansky1} and elsewhere.
\end{remark}

Note that if $(\Omega, d_+, d_-)$ is a curved complex in $\CKom(\CC)$, then $(\Omega, d_+)$ is an honest complex (i.e., an object in $\Kom(\CC)$), since we assume that $d_+^2=0$. However, a morphism $f$ in $\CKom$ is not necessarily a morphism in $\Kom$, since we only require $f$ to commute with the total differential $d = d_+ + d_-$ and not individually with $d_+$ and $d_-$; moreover, $f$ typically does not preserve the homological grading, which is required for morphisms in $\Kom$. In order to extract morphisms in $\Kom$ from (certain) morphisms in $\CKom$, we now introduce a filtered version of $\CKom$, motivated by the $j$ filtration of the Batson--Seed complex from the previous section.

Informally, we wish to think of the ``filtration grading'' on a curved complex $(\Omega, d_+, d_-)$ as the sum of the homological grading (denoted $r$ above) and the internal grading coming from the category $\CC$. The differentials  $d_+$ and $d_-$ are then homogeneous of degrees $0$ and $-2$ with respect to this filtration grading, respectively, so the total differential $d$ is filtered (but not homogeneous), and $d_+$ can be thought of as the differential of the associated graded complex. Of course, in our formalism of graded categories, the gradings only make sense at the level of morphisms; there isn't actually a notion of taking quotients to form this associated graded complex. Nevertheless, we make the following definition:

\begin{definition} \label{def: filtered}
Let $\Omega_a, \Omega_b$ be objects in $\CKom(\CC)$. For any chain map $f \co \Omega_a \to \Omega_b$ and integer $m$, let $f^{(m)}$ consist of all the terms $f^{r,r+2k}_n $ with $n + 2k = m$. We call $f$ \emph{$p$-filtered} if $f^{(m)} = 0$ for all $m>p$.

Likewise, for a homotopy $h$ from $f$ to $g$, let $h^{(m)}$ consist of all terms $h^{r,r+2k+1}_n$ with $n+2k+1 = m$. We call $h$ \emph{$p$-filtered} if $h^{(m)}=0$ for all $m>p$. If $f$ and $g$ are $p$-filtered maps and are related by a $p$-filtered homotopy, we say that they are \emph{$p$-filtered-homotopic}. When $p=0$, we simply say \emph{filtered} rather than $0$-filtered throughout.

Let $\CKom^f(\CC)$ denote the subcategory of $\CKom(\CC)$ whose objects are the same as $\CKom(\CC)$, and whose morphisms are the filtered maps. (The composition of two filtered maps is easily seen to be filtered; more generally, the composition of a $p$-filtered map and a $q$-filtered map is $(p+q)$-filtered.) Let $\CKom^f_{/h}(\CC)$ denote the category with the same objects, whose morphisms are filtered homotopy classes of filtered maps.
\end{definition}

If a map $f$ is homogeneous of degree $p$ with respect to the internal grading of $\CC$, then $f^{(m)}$ can also be seen as the sum of all terms $f^{r,r+2k}$ with $2k = m-p$. In particular, if $f$ is both homogeneous of degree $p$ and $p$-filtered, then all nonzero terms in $f$ preserve or decrease the homological grading, and $f^{(p)}$ consists of precisely those terms that preserve the homological grading. (This will be the case with the cobordism maps considered in Section \ref{sec: cobordism-maps}.) If so, we may view $f^{(p)}$ as the ``associate graded'' morphism of $f$, in the following sense:

\begin{lemma} \label{lemma: filtered}
Let $(\Omega, d_+, d_-)$ and $(\Omega', d_+', d_-')$ be objects in $\CKom(\CC)$.
\begin{enumerate}
\item If $f \co \Omega \to \Omega'$ is a chain map that is homogeneous of degree $p$ and $p$-filtered, then $f^{(p)}$ is a valid morphism in $\Kom(\CC)$ from $(\Omega, d_+)$ to $(\Omega', d_+')$.

\item If $f,g \co \Omega \to \Omega'$ are chain maps that are homogeneous of degree $p$, $p$-filtered, and $p$-filtered-homotopic, then $f^{(p)}$ and $g^{(p)}$ are homotopic (in the sense of Definition \ref{def: Kom-homotopy}).
\end{enumerate}
\end{lemma}

\begin{proof}
For part (1), for each $r$, the portions of $df$ and $fd$ that go from $\Omega^r$ to $\Omega'^{r+1}$ are equal:
\begin{equation} \label{eq: filtered-chain}
d_+^{\prime r} f^{r,r} + d_-^{\prime r+2} f^{r,r+2} = f^{r+1,r+1} d_+^r + f^{r-1,r+1} d_-^r.
\end{equation}
Since $f$ is $p$-filtered, we have $f^{r,r+2} = f^{r-1,r+1} = 0$, so
\[
f^{r+1,r+1} d_+^r = d_+^{\prime r} f^{r,r}.
\]
Thus, $f^{(p)}$ is a chain map.

For part (2), let $h$ be a homotopy from $f$ to $g$, satisfying $f-g = d'h+hd$. Since $f$, $g$, $d$, and $d'$ are all homogeneous (of degrees $p$, $p$, $-1$, and $-1$, respectively), this equation still holds if we replace $h$ with its degree $p+1$ part; thus, we may assume $h$ is homogeneous of degree $p+1$ (and still $p$-filtered). Definition \ref{def: filtered} implies that $h^{r,r+2k+1}=0$ for all $k \ge 0$, while $h^{(p)}$ consists of all the terms $h^{r,r-1}$. Looking at the terms in the relation $f-g = d'h + hd$ that go from $\Omega^r$ to $\Omega^{\prime r}$ gives:
\begin{equation} \label{eq: filtered-htpy}
f^{r,r} - g^{r,r} = d'_+{}^{r-1} h^{r,r-1} + d'_-{}^{r+1} h^{r,r+1} + h^{r+1,r} d_+^r + h^{r-1,r} d_-^r.
\end{equation}
The second and fourth terms on the right vanish. Thus, $f^{(p)} - g^{(p)} = d'_+ h^{(p)} + h^{(p)} d_+$, as required.
\end{proof}

The next lemma, which will ultimately be useful for the proof of Theorem \ref{thm: main}, concerns a related case in which the morphisms and homotopies are not assumed to be filtered, but the $d_-$ maps on the individual complexes vanish.

\begin{lemma} \label{lemma: dminus0}
Let $(\Omega, d_+, d_-)$ and $(\Omega', d_+', d_-')$ be objects in $\CKom(\CC)$, and assume that $d_-=0$ and $d'_-=0$.
\begin{enumerate}
\item If $f \co \Omega \to \Omega'$ is a chain map that is homogeneous of degree $p$, then $f^{(p)}$ is a valid morphism in $\Kom(\CC)$ from $(\Omega, d_+)$ to $(\Omega', d_+')$.

\item If $f,g \co \Omega \to \Omega'$ are chain maps that are homogeneous of degree $p$, then $f^{(p)}$ and $g^{(p)}$ are homotopic (in the sense of Definition \ref{def: Kom-homotopy}).
\end{enumerate}
\end{lemma}

\begin{proof}
The proof is exactly the same as for Lemma \ref{lemma: filtered}, except that the necessary terms in \eqref{eq: filtered-chain} and \eqref{eq: filtered-htpy} now vanish because $d_-$ and $d'_-$ are both $0$.
\end{proof}

\begin{example} \label{ex: BS-Ckom-RMod}
Let $\Mod_R$ denote the category of $\Z$-graded modules over  $R$, which is an $R$-additive category. For a link diagram $L$ along with a choice of weights in $R$, the Batson--Seed complex $(C,d_+,d_-)$ constructed in Section \ref{ssec: BSlink} can be seen as object of $\CKom(\Mod_R)$. To be precise, as before, let $V = R[X]/(X^2)$, with $\deg(1)=1$ and $\deg(X)=-1$. The complex takes the form $(C, d_+, d_-)$, where
\[
C^r = \bigoplus_{v \in \{0,1\}^n \co \abs{v} = r+ n_-} V^{\otimes k_v} \{n_+ - n_-\}.
\]
Note that the internal grading of each summand represents the $l = j-i$ grading, rather than the $j$ grading as is more common. The shifts agree with the description in Section \ref{ssec: BSlink}. The $d_+$ and $d_-$ differentials are each homogeneous of degree $-1$ with respect to this grading, as required. In this particular case, we have $d^2 = 0$; that is, the curvature vanishes. We can reframe Batson and Seed's proof of invariance as saying that the isomorphism class of $C$ in $\CKom^f_{/h}(\Mod_R)$ is an invariant of the underlying weighted link; this formulation will follow from our more abstract statement below in Section \ref{sec: reidemeister}.
\end{example}

\subsection{The tangle cobordism category} \label{ssec: cob3}

We now describe the tangle cobordism categories $\Cob^3$, $\Cob^3_{/l}$, $\Cob^3_\bullet$, and $\Cob^3_{\bullet/l}$, introduced by Bar--Natan \cite[Sections 3, 6, 11.2]{BarNatanTangles}. The only difference is that we will allow formal linear combinations of morphisms over an arbitrary $R$ rather than just over $\Z$.

Let $D$ denote a disk in $\R^2$ and let $\CC$ be its boundary circle. Choose a base point $p_0 \in \CC$, and let $B$ be a finite set of points in $\CC - \{p_0\}$ with an even number of points. A \emph{planar tangle} with boundary $B$ is a properly embedded $1$-manifold in $D$ with boundary $B$. A \emph{tangle cobordism} from $T$ to $T'$ is a properly embedded, compact surface $S \subset D \times I$ with
\[
\partial S = (T \times \{0\}) \cup (B \times I) \cup (T' \times \{1\}).
\]
A \emph{dotted cobordism} is a cobordism $S$ along with a finite set of points on the surface $S$. The \emph{degree} of a cobordism with $d$ dots is defined to be
\begin{equation} \label{eq: degree}
\deg S:= \chi(S) - \abs{B}/2 - 2 d,
\end{equation}
where $\chi(S)$ is the Euler characteristic of $S$.

It is easy to verify that the degree is additive under stacking. We also note for future reference that the degree of a saddle cobordism is $-1$, and degree of a birth or death coordism is $+1$. (See \cite[Exercise 6.3]{BarNatanTangles}.) As in \cite{BarNatanTangles}, we will often use the following notation to indicate cobordisms using planar pictures:

\begin{equation} \label{eq: key}
\begin{array}{lccccc}
\text{Birth} \qquad & \vcenter{\hbox{\birth[0.5]}} & : & \emptyset & \longrightarrow & \vcenter{\hbox{\unknot[0.5]}} \\
\text{Death} \qquad & \vcenter{\hbox{\death[0.5]}} & : & \vcenter{\hbox{\unknot[0.5]}} & \longrightarrow & \emptyset  \\
\text{Saddle} \qquad & \vcenter{\hbox{\vsaddle[0.5]}} & : & \vcenter{\hbox{\vres[0.5]}} & \longrightarrow & \vcenter{\hbox{\hres[0.5]}}
\end{array}
\end{equation}

\begin{definition} \label{def: cob3}
Objects of both $\Cob^3(B)$ and $\Cob^3_\bullet(B)$ are pairs $T\{m\}$, where $T$ is a planar tangle with boundary $B$ and $m \in \Z$. Morphisms in $\Cob^3(B)$ (resp.~$\Cob^3_{\bullet}(B)$) from $T\{m\}$ to $T'\{n\}$ are formal $R$-linear combinations of cobordisms (resp.~dotted cobordisms) from $T$ to $T'$, considered up to isotopy rel boundary. Composition of morphisms is by stacking. The degree of a morphism $S \in \Mor(T\{m\}, T'\{n\})$ is given by $\deg(S) +n-m$. A linear combination of cobordisms is considered homogeneous of degree $d$ if all the constituent cobordisms have degree $d$. This gives $\Cob^3(B)$ and $\Cob^3_\bullet(B)$ the structure of graded categories.
\end{definition}

When $B=\emptyset$, we may ignore the disk $D$ and consider the circles to lie in $\R^2$ and the cobordisms in $\R^2 \times I$.

\begin{definition} \label{def: cob3l}
Let $\Cob^3_{/l}(B)$ (resp.~$\Cob^3_{\bullet/l}(B)$) denote the quotient of $\Cob^3(B)$ (resp.~$\Cob^3_{\bullet}(B)$) in which the objects are the same as those of $\Cob^3$, but the morphisms are quotiented by the local relations in Figure \ref{fig: local-nodots} (resp.~Figure \ref{fig: local-relations}). Since each relation is homogeneous, each of these inherits the structure of a graded $R$-additive category. Note that the local relations for $\Cob^3_{/l}(B)$ follow from those for $\Cob^3_{\bullet/l}(B)$; as a result, the inclusion of $\Cob^3(B)$ into $\Cob^3_\bullet(B)$ descends to the $/l$ versions.
\end{definition}

\begin{figure}
\[
\vcenter{\hbox{\sphere[0.75]}} =0 \qquad \vcenter{\hbox{\torus[0.75]}} = 2
\]
\bigskip
\[
\vcenter{\hbox{\begin{tikzpicture}[scale=0.75]
\begin{scope}[rotate=45]
\draw[] (0,1) ellipse (.5 and .25);
\draw[] (0,-1) ellipse (.5 and .25);
\draw[] (-1,0) ellipse (.25 and .5);
\draw[] (1,0) ellipse (.25 and .5);
\filldraw[white] (-.1,1) arc (180:0:.1 and .05) arc (180:270:.9) arc (90:-90:.05 and .1) arc (270:180:1.1);
\draw[] (-.1,1) arc (180:0:.1 and .05) arc (180:270:.9) arc (90:-90:.05 and .1) arc (270:180:1.1);
\draw[very thin] (-.1,1) arc (180:360:.1 and .05);
\draw[very thin] (1,.1) arc (90:270:.05 and .1);
\end{scope}
\end{tikzpicture}}}
+
\vcenter{\hbox{\begin{tikzpicture}[scale=0.75]
\begin{scope}[rotate=225]
\draw[] (0,1) ellipse (.5 and .25);
\draw[] (0,-1) ellipse (.5 and .25);
\draw[] (-1,0) ellipse (.25 and .5);
\draw[] (1,0) ellipse (.25 and .5);
\filldraw[white] (-.1,1) arc (180:0:.1 and .05) arc (180:270:.9) arc (90:-90:.05 and .1) arc (270:180:1.1);
\draw[] (-.1,1) arc (180:0:.1 and .05) arc (180:270:.9) arc (90:-90:.05 and .1) arc (270:180:1.1);
\draw[very thin] (-.1,1) arc (180:360:.1 and .05);
\draw[very thin] (1,.1) arc (90:270:.05 and .1);
\end{scope}
\end{tikzpicture}}}
\ = \
\vcenter{\hbox{\begin{tikzpicture}[scale=0.75]
\begin{scope}[rotate=315]
\draw[] (0,1) ellipse (.5 and .25);
\draw[] (0,-1) ellipse (.5 and .25);
\draw[] (-1,0) ellipse (.25 and .5);
\draw[] (1,0) ellipse (.25 and .5);
\filldraw[white] (-.1,1) arc (180:0:.1 and .05) arc (180:270:.9) arc (90:-90:.05 and .1) arc (270:180:1.1);
\draw[] (-.1,1) arc (180:0:.1 and .05) arc (180:270:.9) arc (90:-90:.05 and .1) arc (270:180:1.1);
\draw[very thin] (-.1,1) arc (180:360:.1 and .05);
\draw[very thin] (1,.1) arc (90:270:.05 and .1);
\end{scope}
\end{tikzpicture}}}
+
\vcenter{\hbox{\begin{tikzpicture}[scale=0.75]
\begin{scope}[rotate=135]
\draw[] (0,1) ellipse (.5 and .25);
\draw[] (0,-1) ellipse (.5 and .25);
\draw[] (-1,0) ellipse (.25 and .5);
\draw[] (1,0) ellipse (.25 and .5);
\filldraw[white] (-.1,1) arc (180:0:.1 and .05) arc (180:270:.9) arc (90:-90:.05 and .1) arc (270:180:1.1);
\draw[] (-.1,1) arc (180:0:.1 and .05) arc (180:270:.9) arc (90:-90:.05 and .1) arc (270:180:1.1);
\draw[very thin] (-.1,1) arc (180:360:.1 and .05);
\draw[very thin] (1,.1) arc (90:270:.05 and .1);
\end{scope}
\end{tikzpicture}}}
\]
\caption{The sphere, torus, and four-tube relations in $\Cob^3_{/l}$.} \label{fig: local-nodots}
\end{figure}
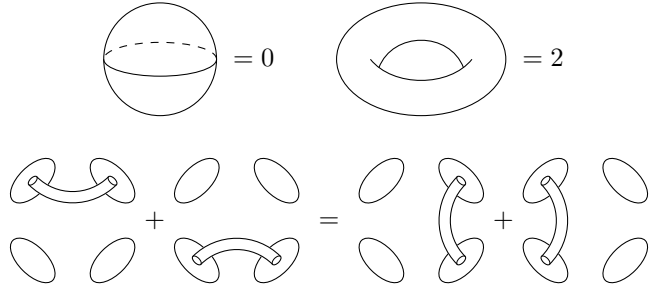

\begin{figure}
\[
\vcenter{\hbox{\sphere[0.75]}} =0 \qquad \vcenter{\hbox{\dottedsphere[0.75]}} = 1 \qquad
\vcenter{\hbox{
\begin{tikzpicture}[scale=0.75]
\draw (-1.5,-1) -- (0.5,-1) -- (1.5,1) -- (-0.5,1) -- (-1.5,-1);
\node at (-0.5,0) {$\bullet$};
\node at (0.5,0) {$\bullet$};
\end{tikzpicture}}} = 0
\]

\bigskip

\[
\vcenter{\hbox{
\begin{tikzpicture}[scale=0.75]
\draw (0,0) ellipse (0.3 and 1);
\draw (3,1) arc (90:-90:0.3 and 1);
\draw[dashed] (3,1) arc (90:270:0.3 and 1);
\draw (0.02,1) .. controls (1,.5) and (2,.5) .. (2.98,1);
\draw (0.02,-1) .. controls (1,-.5) and (2,-.5) .. (2.98,-1);
\draw (1.5,.62) arc (90:-90:0.186 and .62);
\draw[dashed] (1.5,.62) arc (90:270:0.186 and .62);
\end{tikzpicture}}}
\, =
\vcenter{\hbox{
\begin{tikzpicture}[scale=0.75]
\draw (0,0) ellipse (0.3 and 1);
\draw (3,1) arc (90:-90:0.3 and 1);
\draw[dashed] (3,1) arc (90:270:0.3 and 1);
\draw (0.02,1) .. controls (1.5, .5) and (1.5, -.5) .. (0.02,-1);
\draw (2.98,1) .. controls (1.5, .5) and (1.5, -.5) .. (2.98,-1);
\node at (0.75,0) {$\bullet$};
\end{tikzpicture}}}
\, +
\vcenter{\hbox{
\begin{tikzpicture}[scale=0.75]
\draw (0,0) ellipse (0.3 and 1);
\draw (3,1) arc (90:-90:0.3 and 1);
\draw[dashed] (3,1) arc (90:270:0.3 and 1);
\draw (0.02,1) .. controls (1.5, .5) and (1.5, -.5) .. (0.02,-1);
\draw (2.98,1) .. controls (1.5, .5) and (1.5, -.5) .. (2.98,-1);
\node at (2.25,0) {$\bullet$};
\end{tikzpicture}}}
\]
\caption{The sphere, dotted sphere, two-dot, and neck-cutting relations for $\Cob^3_{\bullet/l}$; see \cite[p. 1493]{BarNatanTangles}.}\label{fig: local-relations}
\end{figure}
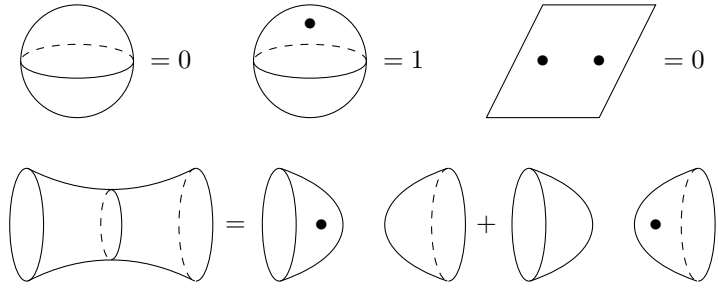

Finally, akin to Bar-Natan's $\Kob(B)$, we define
\begin{align*}
\CKob(B) &= \CKom(\Mat(\Cob^3_{/l}(B))) & \CKob_\bullet(B) &= \CKom(\Mat(\Cob^3_{\bullet/l}(B))) \\
\CKob^f(B) &= \CKom^f(\Mat(\Cob^3_{/l}(B))) & \CKob_\bullet^f(B) &= \CKom^f(\Mat(\Cob^3_{\bullet/l}(B))).
\end{align*}
We will sometimes omit $B$ from the notation if it is understood from context. Note also that up to natural isomorphism, the categories above depend only on the number of points in $B$, so we may also write $\CKob(k)$ or $\CKob_\bullet(k)$ for a set $B$ with $k$ points. We also define $\CKob_{/h}(B)$, $\CKob_{\bullet/h}(B)$, $\CKob_{/h}^f(B)$, and $\CKob_{\bullet/h}^f(B)$ to be the homotopy categories of the above categories.

\subsection{Planar algebras} \label{ssec: planar}

In this section we adapt Bar-Natan's ideas of planar algebra actions (see \cite[Section 5]{BarNatanTangles}) to our homological setting.

\begin{definition}
Following \cite[Definition 5.1]{BarNatanTangles}, a \emph{$d$-input (oriented) planar arc diagram} $D$ consists of a closed disk $\Delta \subset \R^2$; $d$ pairwise disjoint disks $\Delta_1, \dots, \Delta_d$ in the interior of $\Delta$; basepoints $p_i \in \partial \Delta_i$ and $p_0 \in \partial \Delta$, and a properly embedded, compact, (oriented) $1$-dimensional submanifold $A \subset \Delta \minus \operatorname{int}(\Delta_1 \cup \dots \cup \Delta_d)$ whose ends are disjoint from the basepoints.
\end{definition}

Given a planar arc diagram $D$, let $B_i = \partial \Delta_i \cap A$ and $B = \partial \Delta \cap A$. As explained in \cite[Section 5]{BarNatanTangles}, the planar arc diagram induces an operation
\[
D \co \Obj(\Cob^3(B_1)) \times \dots\times \Obj(\Cob^3(B_d)) \to \Obj(\Cob^3(B))
\]
given by plugging in tangles to $\Delta_1, \dots, \Delta_d$ to obtain a tangle in $\Delta$. Likewise, for tangles $T_i, T_i' \in \Obj(\Cob^3(B_i))$, $D$ gives a multi-linear operation
\[
D \co \Mor_{\Cob^3(B_1)}(T_1, T_1') \times \dots\times \Mor_{\Cob^3(B_d)}(T_d, T_d') \to \Mor_{\Cob^3(B)}( D(T_1, \dots, T_d), D(T_1', \dots, T_d'))
\]
obtained by gluing surfaces in $\Delta_i \times I$ to $A \times I$ and extending $R$-linearly. The same applies to the variants $\Cob^3_\bullet$ and $\Cob^3_{\bullet/l}$. These operations satisfy the \emph{identity} and \emph{associativity} axioms (see \cite[p.~1465]{BarNatanTangles}), giving both $\Obj(\Cob^3)$ and $\Mor(\Cob^3)$ the structure of \emph{planar algebras}.

We now discuss the planar algebra operations on $\CKob$ and its variants. Analogous to \cite[Theorem 2]{BarNatanTangles}, we have:

\begin{theorem} \label{thm: CKob-planar}
The collections of categories $\{\CKob(k)\}$ and $\{\CKob_{\bullet}(k)\}$ (as well as their filtered versions) each have the structure of an unoriented planar algebra; moreover, the $D$ operations preserve homotopy of maps and homotopy equivalence of complexes.
\end{theorem}

\begin{proof}
Bar-Natan's proof extends even in our slightly modified homological setting, \emph{mutatis mutandis}. It suffices to point out a few points.

Let $D$ be a planar arc diagram with $d$ inputs, as above, and consider complexes $(\Omega_i,d_{i,\pm}) \in \CKob(B_i)$. Just as in Bar-Natan's setting, we define $D(\Omega_1,\dots,\Omega_d):=(\Omega, d_\pm) \in \CKob(B)$ by
\begin{align}
\label{eq: product-Omega}
\Omega^r&:= \bigoplus _{r=r_1+\dots+r_d} D(\Omega_1^{r_1},\dots,\Omega_d^{r_d}) \\
\label{eq: product-d}
d_\pm |_{D(\Omega_1^{r_1},\dots,\Omega_d^{r_d})} &:= \sum_{i=1}^d (-1)^{\sum_{j<i} r_j} D(I_{\Omega_1^{r_1}}, \dots, d_{i,\pm}^{r_i}, \dots, I_{\Omega_d^{r_d}})
\end{align}
Note that the signs in \eqref{eq: product-d} are exactly the same as those in a tensor product of ordinary chain complexes. The verification that $d_+^2 = d_-^2 = 0$ follows exactly as in the classical proof that the tensor product of chain complexes yields a chain complex.

Next, we consider how to perform the $D$ operation to morphisms in $\CKob$. Given chain maps $\psi_i\co \Omega_{i}\to \Omega'_{i}$ (a morphism in $\CKob(B_i)$) for $i=1, \dots, d$, we obtain an induced chain map
\[
D(\psi_1, \dots, \psi_d) \co D(\Omega_1, \dots, \Omega_d) \to D(\Omega'_1, \dots, \Omega'_d)
\]
which is the sum of the morphisms
\[
D(\psi_1^{r_1, r_1+2k_1}, \dots, \psi_d^{r_d,r_d+2k_d}) \co D(\Omega_1^{r_1}, \dots, \Omega_d^{r_d}) \to D(\Omega_1^{\prime r_1+2k_1}, \dots, \Omega_d^{\prime r_d + 2k_d}),
\]
ranging over all $k_1, \dots, k_d \in \Z$. Continuing with the above analogy, we may think of this operation as taking the tensor product of chain maps; the proof that it commutes with the differential and behaves well with respect to composition is again just like in the classical setting. Moreover, if $\psi_1, \dots, \psi_d$ are filtered morphisms (in the sense of Definition \ref{def: filtered}), then $D(\psi_1, \dots, \psi_d)$ is also filtered.

Chain homotopies between maps can be extended in a similar fashion, which shows that the planar algebra structure descends to $\{\CKob_{/h}(k)\}$.
\end{proof}

\begin{remark} \label{rmk: planar-commute}
In \eqref{eq: product-d}, the definition of $d_\pm$ depends on the order of the inputs in $D$ --- i.e., the indexing of the input circles of $D$ by $\Delta_1, \dots, \Delta_d$. However, different orders yield canonically isomorphic complexes; the proof is the same as in the setting of tensor products of chain complexes over a commutative ring.
\end{remark}

\begin{lemma} \label{lemma: curv-planar}
Let $D$ and $\Omega_1, \dots, \Omega_d$ be as in the proof of Theorem \ref{thm: CKob-planar}, and let $\Omega = D(\Omega_1, \dots, \Omega_d)$. Let $\lambda_i^r \co \Omega_i^r \to \Omega_i^r$ be the curvature of $\Omega_i$, i.e., $\lambda_i^r = d^{r+1}_{i,-} d^r_{i,+} + d^{r-1}_{i,+} d^r_{i,-}$. Then the curvature $\lambda^r \co \Omega^r \to \Omega^r$ is diagonal with respect to the decomposition \eqref{eq: product-Omega}, and the entry
\[
\lambda^r|_{D(\Omega_1^{r_1},\dots,\Omega_d^{r_d})} \co D(\Omega_1^{r_1},\dots,\Omega_d^{r_d}) \to D(\Omega_1^{r_1},\dots,\Omega_d^{r_d})
\]
is given by
\[
\sum_{i=1}^d D(I_1^{r_1}, \dots, \lambda_i^{r_i},  \dots I_d^{r_d}),
\]
where $I_i^r$ denotes the identity morphism of $\Omega_i^r$.
\end{lemma}

\begin{proof}
We compute:
\begin{align*}
d_\mp &d_\pm |_{D(\Omega_1^{r_1},\dots,\Omega_d^{r_d})} \\
&= d_{\mp} \circ \sum_i (-1)^{\sum_{j<i} r_j} D(I_1^{r_1}, \dots, d_{i,\pm}^{r_i}, \dots, I_d^{r_d}) \\
&= \sum_{i=1}^d \sum_{k < i}   (-1)^{\sum_{j=k}^{i-1} r_j}
D(I_1^{r_1}, \dots, d_{k,\mp}^{r_k}, \dots, I_i^{r_i \pm 1}, \dots, I_d^{r_d}) \circ
D(I_1^{r_1}, \dots, I_k^{r_k}, \dots, d_{i,\pm}^{r_i}, \dots, I_d^{r_d}) \\
& \quad + \sum_{i=1}^d D(I_1^{r_1}, \dots, d_{i,\mp}^{r_i\pm 1}, \dots, I_d^{r_d}) \circ D(I_1^{r_1}, \dots, d_{i,\pm}^{r_i}, \dots, I_d^{r_d}) \\
& \quad +
\sum_{i=1}^d \sum_{k > i}  (-1)^{\sum_{j=i}^{k-1} r_j + 1}
D(I_1^{r_1}, \dots, I_i^{r_i\pm 1} , \dots, d_{k,\mp}^{r_k}, \dots, I_d^{r_d}) \circ
D(I_1^{r_1}, \dots, d_{i,\pm}^{r_i}, \dots, I_k^{r_k}, \dots I_d^{r_d})
\end{align*}
The two double sums cancel out, so we deduce:
\[
d_\mp d_\pm |_{D(\Omega_1^{r_1},\dots,\Omega_d^{r_d})} = \sum_{i=1}^d D(I_1^{r_1}, \dots, d_{i,\mp}^{r_i\pm 1} \circ d_{i,\pm}^{r_i}, \dots, I_d^{r_d}).
\]
and hence
\[
(d_+ d_- + d_- d_+)|_{D(\Omega_1^{r_1},\dots,\Omega_d^{r_d})} = \sum_{i=1}^d D(I_1^{r_1}, \dots, d_{i,+}^{r_i-1} d_{i,-}^{r_i} + d_{i,-}^{r_i+1} d_{i,+}^{r_i}, \dots, I_d^{r_d})
\]
as required.
\end{proof}

\section{The Batson--Seed complex of a tangle} \label{sec: BScomplex}

In this section, we extend Batson--Seed's variant of Khovanov homology to tangles, taking values in the category $\CKob$ constructed in the previous section.

Throughout this section, let $T$ be an oriented tangle, represented by a diagram in a disk $\Delta$. Let $B = \partial T$. Let us assume that the diagram has $n$ crossings, labeled $c_1, \dots, c_n$, and let $n_+$ (resp.~$n_-$) denote the number of positive (resp.~negative) crossings. For each $v \in \{0,1\}^n$, let $\abs{v} = \sum_{i=1}^n v_i$, let $T_v$ be the corresponding resolution of $T$ according to the convention shown in Figure \ref{fig: resolutions}, viewed as an unoriented tangle.

The Batson-Seed complex will require two additional pieces of data. A \emph{weighting} on $T$ consists of a choice $w$ of an element of $R$ for each component of $T$. (Note that components may be either arcs or circles.) We also need to specify a \emph{checkerboard shading} of the tangle diagram. A checkerboard shading gives rise to a \emph{sign assignment} $s$ which assigns to each crossing $c_i$ a sign as indicated in Figure \ref{fig: checkerboard}. This function is easily seen to satisfy Batson--Seed's definition of a sign assignment \cite [Section 2.2] {BatsonSeed}, since adjacent crossings have the same sign if and only if the diagram alternates over-under on the segment joining them. Reversing the checkerboard coloring negates $s$. (Batson and Seed's definition of a sign assignment is slightly more general in the case of a disconnected diagram, but we shall restrict our attention to sign assignments arising from a checkerboard coloring.)

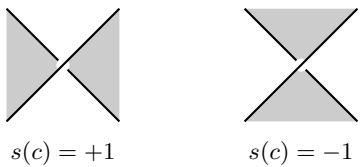
\begin{figure}
  \begin{subfigure}[b]{0.2\textwidth}
  \centering
    \begin{tikzpicture}[scale=0.15]
      \filldraw[black!20!white] (0,10) -- (5,5) -- (0,0);
      \filldraw[black!20!white] (10,10) -- (5,5) -- (10,0);
      \draw[thick] (0,10) -- (10,0);
      \node[crossing] at (5,5) {};
      \draw[thick] (0,0) -- (10,10);
    \end{tikzpicture}
    \caption*{$s(c)=+1$}
  \end{subfigure}
  \begin{subfigure}[b]{0.2\textwidth}
    \centering
    \begin{tikzpicture}[scale=0.15]
      \filldraw[black!20!white] (0,10) -- (5,5) -- (10,10);
      \filldraw[black!20!white] (0,0) -- (5,5) -- (10,0);
      \draw[thick] (0,10) -- (10,0);
      \node[crossing] at (5,5) {};
      \draw[thick] (0,0) -- (10,10);
    \end{tikzpicture}
    \caption*{$s(c)=-1$}
  \end{subfigure}
\caption{The sign assignment coming from a checkerboard shading.}
\label{fig: checkerboard}
\end{figure}

Let $T$ be an oriented, weighted, shaded tangle diagram, with weight $w$ and checkerboard shading/sign assignment $s$. We will now define a complex in $\CKob(B)$, which we call the \emph{Batson--Seed complex} of $(T,w)$, written $BS(T,w)$ or simply $BS(T)$.  (We suppress the shading from the notation.) The chain spaces and $d_+$ differential will not depend at all on the weighting or shading; indeed, they will be identical (up to a slight sign modification) to Bar--Natan's version of the Khovanov complex, thought of as an object in $\Kob(B)$. The $d_-$ differential will depend significantly on both $w$ and $s$.

\begin{definition} \label{def: BS-chain}
The piece of $BS(T,w)$ in homological grading $r-n_-$ is defined (as an object in $\Mat(\Cob^3(B))$) to be:
\begin{equation} \label{eq: BS-chain}
BS(T,w)^{r-n_-} = \bigoplus_{v \in \{0,1\}^n : \abs{v} = r} T_v\{n_+-n_-\}.
\end{equation}
That is, $BS(T,w)$ is just the sum of all the objects $T_v$ (for all $v \in \{0,1\}$), each in homological grading $\abs{v} - n_-$ and with an internal shift of $n_+ - n_-$.
\end{definition}

\begin{remark} \label{rmk: grading}
Note that Bar--Natan \cite[Definition 6.4]{BarNatanTangles} uses a shift of $r+n_+ - n_-$ for the $r$th chain space rather than $n_+ - n_-$.  The reason for this difference is that the degree in Bar--Natan's work is meant to imitate the $j$ grading from Section \ref{ssec: BSlink}, whereas here the degree captures the $l=j-i$ grading, with respect to which the Batson--Seed total differential is homogeneous. To avoid ambiguity, we will refer to the degree in this setting (i.e., the internal grading of the category $\CKob(B)$ as the \emph{Batson--Seed grading}.
\end{remark}

We now discuss the $d_+$ and $d_-$ differentials.

\subsection{The \texorpdfstring{$d_+$}{d+} differential}  \label{ssec: dplus}

Following the template from Section \ref{ssec: BSlink}, the $d_+$ differential on $BS(T,w)$ will be the ordinary Khovanov differential from \cite{BarNatanTangles}, albeit with a slight sign modification.

For each $v \lessdot v'$, let $S_{v,v'}$ be the canonical saddle cobordism from $T_v$ to $T_{v'}$, viewed as a morphism between the corresponding summands of $BS(T,w)$. As in \cite{BarNatanTangles}, if $v, v'_1, v'_2, v'' \in \{0,1\}^n$ are vectors such that $v \lessdot v'_1 \lessdot v''$ and $v \lessdot v'_2 \lessdot v''$ (so that $v$ and $v''$ differ in two indices), then we have
\begin{equation} \label{eq: saddles-commute}
S_{v'_1, v''} \circ S_{v,v'_1} = S_{v'_2, v''} \circ S_{v,v'_2}
\end{equation}
as morphisms in $\Cob^3(B)$.

We must also ``sprinkle signs'' --- i.e., assign a sign to each edge of the cube $[0,1]^n$, such that the boundary of each $2$-dimensional face has an odd number of minus signs. To do this, for an immediate successor pair $v \lessdot v'$, let $i$ be the index in which $v$ and $v'$ differ, and let $m(v,v')$ be the number of indices $j<i$ for which either
\begin{inparaenum}
  \item $v_j = 1$ and $c_j$ is a positive crossing; or
  \item $v_j = 0$ and $c_j$ is a negative crossing.
\end{inparaenum}
We then define $d_+^{r-n_{-}} \co BS(T)^{r-n_-}\to BS(T)^{r-n_-+1}$ to be the morphism in $\Mat(\Cob^3)$ whose $(v,v')$ entry is $(-1)^{m(v,v')} S_{v,v'}$ when $v \lessdot v'$, and is $0$ otherwise. (Note that this is a different sign convention than in \cite{BarNatanTangles} and elsewhere; see Section \ref{ssec: signs}.)

Observe that $d_+$ is homogeneous of degree $-1$ with respect to the Batson--Seed grading, as all saddles have degree $-1$ and the degree shifts of all terms in \eqref{eq: BS-chain} are the same.

The fact that $(d_+)^2 = 0$ then follows from \eqref{eq: saddles-commute} together with the following simple lemma:

\begin{lemma} \label{lemma: sprinkle}
If $v, v'_1, v'_2, v'' \in \{0,1\}^n$ are vectors such that $v \lessdot v'_1 \lessdot v''$ and $v \lessdot v'_2 \lessdot v''$, then \begin{equation} \label{eq: sprinkle}
m(v,v'_1) + m(v'_1,v'') \not\equiv m(v,v'_2) + m(v'_2,v'') \pmod 2.
\end{equation}
\end{lemma}

\begin{proof}
Without loss of generality, assume that $v$ and $v'_1$ differ in index $i$ and that $v$ and $v'_2$ differ in index $j$, where $i < j$. The crossings $c_k$ for $k \ne i$ contribute equally to both sides of \eqref{eq: sprinkle}; we must consider the contributions to $m(v'_1,v'')$ and $m(v,v'_2)$ from $c_i$. If $c_i$ is a positive crossing, then it contributes to $m(v'_1,v'')$ but not to $m(v,v'_2)$; if $c_i$ is a negative crossing, then it contributes to $m(v,v'_2)$ but not to $m(v,v'')$. In either case, the two sides differ mod $2$.
\end{proof}

\subsection{The \texorpdfstring{$d_-$}{d-} differential} \label{ssec: dminus}

As in \cite{BatsonSeed}, the terms in the $d_-$ differential on $BS(T,w,s)$ consist of the reverse cobordisms of those in $d_+$, weighted appropriately. To be precise, for each immediate successor pair $v \lessdot v'$, let $S_{v',v}$ be the reverse of the cobordism $S_{v,v'}$ from above, viewed as a saddle cobordism from $T_{v'}$ to $T_v$. The morphism
\[
d_-^{r} \co BS(T)^{r}\to BS(T)^{r-1}
\]
is then defined to be the matrix of all the morphisms
\begin{equation} \label{eq: dminus-term}
(-1)^{m(v,v')} s(c_i) (w^{i}_{\text{over}}-w^{i}_{\text{under}}) S_{v',v},
\end{equation}
where $i$ is the index in which $v$ and $v'$ differ, and $w^{i}_{\text{over}}$ and $w^{i}_{\text{under}}$ are the weights of the overstrand and understrand at the crossing $c_i$. Observe that $d_-$ is also homogeneous of degree $-1$.

The proof that $(d_-)^2=0$ proceeds just as with $d_+$, keeping track of the extra information in \eqref{eq: dminus-term}. (See Case 1 in the proof of \cite[Proposition 2.2]{BatsonSeed}.) This completes the verification that $(BS(T,w,s),d_+, d_-)$ is a valid object of $\CKob^3(B)$. Note that we do not necessarily have $d^2=0$; we will return to this in Section \ref{ssec: curvature}.

\begin{remark} \label{rmk: change-shading}
From \eqref{eq: dminus-term}, we see that changing the checkerboard shading on $T$ has the same effect on $d_-$ as multiplying all weights by $-1$.  Thus, when arbitrary weights are being considered, it often suffices to examine only one choice of shading.
\end{remark}

\subsection{Local description} \label{localstory}

In this section we will use the planar algebra structure of $\CKob$ (from Section \ref{ssec: planar}) to provide a completely local description of the $BS$ complex.

Let $D$ be a $d$-input oriented planar arc diagram as in Section \ref{ssec: planar}, and let $T_1, \dots, T_d$ be oriented tangles that can be inserted to the inputs of $D$ consistent with orientations. Let $T = D(T_1, \dots, T_d)$ denote the resulting tangle. We say that a collection of weightings (resp.~shadings) on $T_1, \dots, T_d$ is \emph{compatible} with $D$ if there is a weighting (resp.~shading) on $T$ that restricts to the chosen one on each $T_i$.

\begin{figure}
\begin{tabular}{|c|c|c|} \hline
$s(c)$ & Positive crossing & Negative crossing \\ \hline
$1$ &
\begin{tikzpicture}
\node[](poscross) at (-2,0) {
\begin{tikzpicture}
      \filldraw[black!20!white] (0,1) -- (.5,.5) -- (0,0);
      \filldraw[black!20!white] (1,1) -- (.5,.5) -- (1,0);
      \draw [->,thick] (1,0) -- (0,1); 
      \node[crossing] at (.5,.5) {};
      \draw [->,thick] (0,0) -- (1,1); 
      \node[] at (-0.1,.8) {$a$};
      \node[] at (1.1,.8) {$b$};
\end{tikzpicture}};
\node[label=north:{$0$}](0res) at (0,0) {\vres} ;
\node[label=north:{$1$}](1res) at (3,0) {\hres} ;
\draw [->] (0res.10) -- (1res.170) node[midway, above] {\usebox{\vsaddlebox}};
\draw [->] (1res.190) -- (0res.350) node[midway, below] {$(b-a)\usebox{\hsaddlebox}$};
\end{tikzpicture} &
\begin{tikzpicture}
\node[](negcross) at (-2,0) {
\begin{tikzpicture}
      \filldraw[black!20!white] (0,1) -- (.5,.5) -- (0,0);
      \filldraw[black!20!white] (1,1) -- (.5,.5) -- (1,0);
      \draw [->,thick] (0,1) -- (1,0); 
      \node[crossing] at (.5,.5) {};
      \draw [->,thick] (0,0) -- (1,1); 
      \node[] at (-0.1,.8) {$a$};
      \node[] at (1.1,.8) {$b$};
\end{tikzpicture}};
\node[label=north:{$-1$}](0res) at (0,0) {\vres} ;
\node[label=north:{$0$}](1res) at (3,0) {\hres} ;
\draw [->] (0res.10) -- (1res.170) node[midway, above] {\usebox{\vsaddlebox}};
\draw [->] (1res.190) -- (0res.350) node[midway, below] {$(b-a)\usebox{\hsaddlebox}$};
\end{tikzpicture}
 \\ \hline
$-1$ &
\begin{tikzpicture}
\node[](poscross) at (-2,0) {
\begin{tikzpicture}
      \filldraw[black!20!white] (0,1) -- (.5,.5) -- (1,1);
      \filldraw[black!20!white] (0,0) -- (.5,.5) -- (1,0);
      \draw [->,thick] (1,0) -- (0,1); 
      \node[crossing] at (.5,.5) {};
      \draw [->,thick] (0,0) -- (1,1); 
      \node[] at (-0.1,.8) {$a$};
      \node[] at (1.1,.8) {$b$};
\end{tikzpicture}};
\node[label=north:{$0$}](0res) at (0,0) {\vres} ;
\node[label=north:{$1$}](1res) at (3,0) {\hres} ;
\draw [->] (0res.10) -- (1res.170) node[midway, above] {\usebox{\vsaddlebox}};
\draw [->] (1res.190) -- (0res.350) node[midway, below] {$(a-b)\usebox{\hsaddlebox}$};
\end{tikzpicture} &
\begin{tikzpicture}
\node[](negcross) at (-2,0) {
\begin{tikzpicture}
      \filldraw[black!20!white] (0,1) -- (.5,.5) -- (1,1);
      \filldraw[black!20!white] (0,0) -- (.5,.5) -- (1,0);
      \draw [->,thick] (0,1) -- (1,0); 
      \node[crossing] at (.5,.5) {};
      \draw [->,thick] (0,0) -- (1,1); 
      \node[] at (-0.1,.8) {$a$};
      \node[] at (1.1,.8) {$b$};
\end{tikzpicture}};
\node[label=north:{$-1$}](0res) at (0,0) {\vres} ;
\node[label=north:{$0$}](1res) at (3,0) {\hres} ;
\draw [->] (0res.10) -- (1res.170) node[midway, above] {\usebox{\vsaddlebox}};
\draw [->] (1res.190) -- (0res.350) node[midway, below] {$(a-b)\usebox{\hsaddlebox}$};
\end{tikzpicture}
 \\ \hline
\end{tabular}

\caption{Batson--Seed complexes associated to the four possible orientations and shadings of a one-crossing, four-end tangle. The numbers above the resolutions indicate the homological grading, and $s(c)$ denotes the value of the sign assignment at the crossing.}
\label{fig: 1crossing}
\end{figure}

\begin{theorem} \label{thm: BS-local}
Let $D$ be a $d$-input oriented planar arc diagram, and let $T_1, \dots,T_d$ be oriented, weighted, shaded tangles that are compatible with $D$. Then
\[
BS(D(T_1, \dots,T_d))=D(BS(T_1), \dots,BS(T_d)).
\]
\end{theorem}

\begin{proof}
 This proof mimics the proof of \cite[Theorem 2]{BarNatanTangles}, making use of the planar algebra structure of $\CKob$ discussed in Section \ref{ssec: planar}. Just as in that proof, it suffices to assume that each tangle $T_i$ consists of two strands crossing once; the general case then follows from the associativity of the planar algebra structure. Taking shadings and orientation into account, there are four possibilities for $BS(T_i)$, shown in Figure \ref{fig: 1crossing}.

Equality on the level of chain spaces is immediate: each resolution of $T$ is obtained by applying $D$ to resolutions of $T_1, \dots, T_d$ in a unique way. Likewise, each of the cobordisms $S_{v,v'}$ (or $S_{v',v}$) in the definition of $d_\pm$ for $T$ is obtained by extending a saddle cobordism between resolutions of some $T_i$ by the identity on the remaining tangles.

It remains to verify that the signs agree. Consider the term of $d_+$ corresponding to an immediate successor pair $v \lessdot v'$ differing in the $i\Th$ entry. By definition, the sign that appears in the differential on $BS(D(T_1, \dots,T_d))$ is $(-1)^{m(v,v')}$, while the sign in the differential on $D(BS(T_1), \dots, BS(T_d))$ is, according to \eqref{eq: product-d}, $(-1)^{\sum_{j<i} r_j}$, where $r_j$ denotes the homological grading of the resolution of $T_j$ being used. As seen in Figure \ref{fig: 1crossing}, for a single-crossing tangle, the $1$-resolution of a positive crossing and the $0$-resolution of a negative crossing are the ones that lie in odd homological grading. Thus, we have $\sum_{j<i} r_j \equiv m(v,v') \pmod 2$, as required. The same analysis holds for terms of $d_-$, using the fact that the factors $s(c_i)$ and $w^{i}_{\text{over}}-w^{i}_{\text{under}}$ occurring in \eqref{eq: dminus-term} both depend only on the local picture in $T_i$.
\end{proof}

As a consequence of Theorem \ref{thm: BS-local}, we obtain a method of extending chain maps:

\begin{theorem} \label{thm: BS-extend}
Suppose $T_1, T_2$ are oriented, weighted, shaded tangle diagrams that agree (including in their orientations, weightings, and shadings) outside a disk $\Delta$, and let $T_1' = T_1 \cap \Delta$ and $T_2' = T_2 \cap \Delta$, with induced orientations, weightings, and shadings. Then:
\begin{enumerate}
\item
For any chain map $f \co BS(T_1') \to BS(T_2')$, there is a natural extension $f \co BS(T_1) \to BS(T_2)$.

\item
If $f, f' \co BS(T_1') \to BS(T_2')$ are homotopic, then so are their extensions.
\end{enumerate}
\end{theorem}

\begin{proof}
This follows directly from Theorem \ref{thm: BS-local} and Theorem \ref{thm: CKob-planar}. To be precise, we may choose small disks around each of the crossings of $T_1$ that are not in $\Delta$, and view the remainder of $T$ as a planar arc diagram $D$. We then extend $f$ by taking its planar algebra product (as in the proof of Theorem \ref{thm: CKob-planar}) with the identity morphisms of the $BS$ complexes of the single-crossing tangles. The proof for homotopies is similar.
\end{proof}

\subsection{Curvature terms} \label{ssec: curvature}

With the local story in hand, we now show that the curvature terms in $BS(T)$ have a particularly nice form when we make use of the dotted cobordism category.

\begin{figure}
\centering
\begin{subfigure}[b]{0.3\textwidth}
\centering
\begin{tikzpicture}[scale=2]
\filldraw[black!20!white] (0,.5) -- (0,1) arc (90:60:1) -- (.5,.5);
\draw (0,1) arc (90:120:1);
\draw (0,1) arc (90:60:1);
\draw[thick] (0,.5) -- (0,1);
\node[blackdot] at (0,1) {};
\end{tikzpicture}
\caption*{$s(p)=1$}
\end{subfigure}
~
\begin{subfigure}[b]{0.3\textwidth}
\centering
\begin{tikzpicture}[scale=2]
\filldraw[black!20!white] (0,.5) -- (0,1) arc (90:120:1) -- (-.5,.5);
\draw (0,1) arc (90:120:1);
\draw (0,1) arc (90:60:1);
\draw[thick] (0,.5) -- (0,1);
\node[blackdot] at (0,1) {};
\end{tikzpicture}
\caption*{$s(p)=-1$}
\end{subfigure}
\caption{Signs associated to a boundary point $p \in B$.}
\label{fig: boundary-sign}
\end{figure}
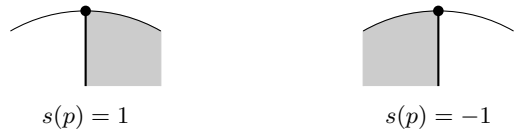

Given an oriented, weighted, shaded, oriented tangle $T$ with boundary $B$, let us assign signs to the boundary points $p \in B$ as follows:  let $s(p) = + 1$ (resp.~$s(p)=-1$) if, as we traverse the boundary of the disk counterclockwise, we pass from shaded to unshaded (resp.~unshaded to shaded) at $p$, as seen in Figure \ref{fig: boundary-sign}. Let $w(p)$ denote the weight of the strand containing $p$. Let $X_p \co BS(T,w) \to BS(T,w)$ denote the chain map that consists of, for each $v \in \{0,1\}^n$, the identity cobordism $T_v \times [0,1]$ with a single dot placed on the component containing $p$.

\begin{proposition} \label{prop: BS-curv}
Let $T$ be an oriented, weighted, shaded tangle diagram with boundary $B$, and consider the complex $BS(T)$ as an object in $\CKob_\bullet(B)$. Then the curvature $\lambda_T$ of $BS(T)$ is diagonal with respect to the decomposition $BS(T) = \bigoplus_{v \in \{0,1\}^n} T_v$, and
\begin{equation} \label{eq: BS-curv}
\lambda_T|_{T_v} = \sum_{p \in B} s(p) w(p) X_p.
\end{equation}
\end{proposition}

\begin{proof}
We begin by considering the case where $T$ is a single-crossing tangle, as shown in Figure \ref{fig: 1crossing}. Let $s(c)$ be the sign associated to the crossing by Figure \ref{fig: checkerboard}, and let $a$ and $b$ denote the weights of the understrand and overstrand, respectively, as shown. Note that the two points in $B$ with $w(p)=b$ have $s(p)=s(c)$, and the two points with $w(p) = a$ have $s(p) = -s(c)$.

In either case, the curvature, restricted to each $T_v$, is then equal to $s(c)(b-a)$ times the composite of two saddle cobordisms, which is obtained by tubing together the two components of $T_v \times [0,1]$. By the neck-cutting relation from Figure \ref{fig: local-relations}, this cobordism is equal (in $\Mor_{\Cob^3_{\bullet/l}}$) to the sum of the two ways of putting a dot on either component of $T_v \times [0,1]$. Distributing out the multiplication, we see that $\lambda_T$ has the desired form.

For the general case, let $T$ be a tangle diagram with $n$ crossings. We may think of $T$ as obtained by inserting $1$-crossing tangles $T_1, \dots, T_n$ into the inputs of a planar arc diagram $D$. By Theorem \ref{thm: BS-local}, we have $BS(T) = D(BS(T_1), \dots, BS(T_n))$. By Lemma \ref{lemma: curv-planar}, the $\lambda_T$ is diagonal with respect to the decomposition $BS(T) = \bigoplus_v T_v$, and
\begin{equation} \label{eq: BS-curv2}
\lambda_T|_{T_v} = \sum_{p \in \partial T_1 \cup \dots \cup \partial T_n} s(p) w(p) X_p,
\end{equation}
where the signs $s(p)$ come from the boundary orientations of the disks containing the small tangles $T_i$.
If $p$ and $p'$ are points of $\partial T_1 \cup \dots \partial T_n$ that are connected by an arc of $D$, then $s(p)=-s(p')$, so the corresponding terms of \eqref{eq: BS-curv2} cancel. Likewise, for $p, p' \in \partial T$, the corresponding terms in \eqref{eq: BS-curv} cancel. On the other hand, if $p \in \partial T_1 \cup \dots \partial T_n$ and $p' \in \partial T$ are connected by an arc of $D$, then $p$ contributes equally to the sums in \eqref{eq: BS-curv} and \eqref{eq: BS-curv2}, which completes the proof.
\end{proof}

As an immediate consequence of Proposition \ref{prop: BS-curv}, we see that when $L$ is a link, the total differential $d$ on $BS(L,w)$ satisfies $d^2=0$, as shown by Batson and Seed \cite[Proposition 2.2]{BatsonSeed} for the original version of their chain complex. As we noted in Remark \ref{rmk: curved}, this is akin to obtaining an honest chain complex from a tensor product of matrix factorizations, as in \cite{KhovanovRozansky1} and elsewhere.

\subsection{More on signs} \label{ssec: signs}

We now elaborate on our convention for sprinkling signs, as discussed in Section \ref{ssec: dplus}, and how it differs from the convention elsewhere in the literature. (The casual reader may safely skip this section.)

In place of our $m(v,v')$, Bar-Natan uses the quantity $m'(v,v') := \sum_{j<i} v_j$ when defining the sign associated to each edge of the cube $[0,1]^n$ in the Khovanov complex, and this convention has become standard in the literature. As seen in the proof of Theorem \ref{thm: BS-local} above, our sign convention is necessary in order to guarantee that the signs agree when taking planar algebra products using \eqref{eq: product-d}.

Without the shift of $-n_-$ in the homological grading, the complex defined using $m'(v,v')$ would behave correctly under planar algebra products. Indeed, in Bar-Natan's expository paper on Khovanov homology \cite{BarNatanCategorification}, he defines a (classical) complex $\Brackets{L}$ for any unoriented link diagram $L$, using the $m'$ sign convention, with no shift in the homological or quantum gradings; he then defines the Khovanov complex of an oriented link to be obtained from $\Brackets{L}$ by an overall shift in both gradings. (This approach is motivated by the relationship between the Kauffman bracket and the Jones polynomial.) In \cite{BarNatanTangles}, Bar-Natan likewise defines $\Brackets{T}$ for an unoriented tangle diagram $T$, again using the $m'$ convention, but with the homological grading shift of $-n_-$ already built in. If one instead defines $\Brackets{T}$ without the $-n_-$ shift, then the signs work out exactly right so that $\Brackets{ \, }$ behaves nicely with respect to planar algebra products, as stated in \cite[Theorem 2]{BarNatanTangles}. (With the shifts, the theorem still holds as a statement of isomorphism of complexes, but the isomorphism has some ``$-$" signs built in.)

In any event, the choice of how to sprinkle signs is in some sense immaterial up to isomorphism. Indeed, we have a more general principle, which is well-known to experts (see, e.g., \cite[Lemma 4.5]{ClarkMorrisonWalkerFunctoriality}, \cite[Lemma 2.2]{OzsvathRasmussenSzaboOdd}, and, in a more general setting, \cite[Theorem 6.6]{ChandlerPosets}):

\begin{lemma} \label{lemma: sprinkle-indep}
Let $m$ be any function on the set of all immediate successor pairs in $\{0,1\}^n$ that satisfies \eqref{eq: sprinkle}, and let $d_+^m$ and $d_-^m$ be the differentials on $BS(L,w)$ defined as above using the signs $(-1)^{m(v,v)}$ where applicable. Then the isomorphism type of $(BS(L,w), d_+^m, d_-^m)$ (in $\CKob^f$) is independent of $m$.
\end{lemma}

\begin{proof}
Pairs $v \lessdot v'$ correspond to edges of the cube $[0,1]^n$ with its standard cell structure. A function $m$ as in the definition may thus viewed as a cellular cochain $\mu_m \in C^1([0,1]^n;\Z/2)$; the property \eqref{eq: sprinkle} is equivalent to saying that $\delta \mu_m$ evaluates to $1$ on every $2$-dimensional face. If $m'$ is another such function, then $\delta(\mu_{m'} - \mu_m) = 0$, so $\mu_{m'} = \mu_m + \delta\phi$ for some $\phi \in C^0([0,1]^n;\Z/2)$ since $H^1([0,1]^n;\Z/2)=0$. That is, for each $v \lessdot v'$, we have
\[
m'(v,v') = m(v,v') + \phi(v) - \phi(v').
\]
We then obtain an isomorphism
\[
\Phi \co (BS(L), d_+^m, d_-^m) \to (BS(L), d_+^{m'}, d_-^{m'})
\]
consisting of $(-1)^{\phi(v)}$ times the identity cobordism of each resolution $T_v$.
\end{proof}

Next, note that the function $m$ depends on the orientation of $T$, since it is based on the signs of the crossings, while Bar-Natan's $m'$ does not. However, this dependence also disappears (up to isomorphism) thanks to the preceding lemma. To be precise, let $T$ and $T'$ denote different orientations on the same underlying unoriented (weighted, shaded) tangle diagram. Let $n_{\pm}$ and $n_{\pm}'$ denote the number of $\pm$ crossings in $T$ and $T'$, respectively. On the level of chain spaces, we have
\[
BS(T) = BS(T')[-n_- + n_-']\{ n_+ - n_- - n_+' + n_-'\}
\]
where $[ \cdot ]$ represents a shift in the homological grading, and the differentials are the same up to the formula for sprinkling. By Lemma \ref{lemma: sprinkle-indep}, the two complexes are isomorphic. We leave the construction of an explicit isomorphism to the reader.

Finally, while the definition of $m$ depends on the ordering of the crossings of $T$, this dependence again disappears up to isomorphism thanks to Theorem \ref{thm: BS-local} and Remark \ref{rmk: planar-commute}.

\subsection{Recovering the original chain complex}

To come full circle, we now discuss how to recover Batson and Seed's original construction from the description above.

Let $\FF \co \Cob^3_\bullet(\emptyset) \to \Mod_R$ denote the Khovanov TQFT functor, using the Frobenius algebra $V = R[X]/(X^2)$. To be precise, $\FF$ assigns $V^{\otimes k}$ to a crossingless diagram with $k$ circles, with a grading shift if applicable. (In Bar-Natan's notation, $v_+$ corresponds to $1$ and $v_-$ to $X$.) The maps associated to elementary birth, saddle, and death cobordisms in $\R^2 \times [0,1]$ are as described in \cite[Section 2.2]{Khov}; the map associated to a dotted cylinder is multiplication by $X$. The functor respects the relations in Figure \ref{fig: local-relations}, and hence descends to a functor $\Cob^3_{\bullet/l}(\emptyset) \to \Mod_R$. (This is spelled out for the undotted relations in Figure \ref{fig: local-nodots} in \cite[Proposition 7.2]{BarNatanTangles}; the proof for the dotted relations proceeds similarly.)

If $L$ is a weighted link diagram, applying $\FF$ to $BS(L,w)$ gives a curved chain complex in $\CKom(\Mod_R)$, which agrees precisely with the original Batson--Seed complex, as described in Section \ref{ssec: BSlink} and Example \ref{ex: BS-Ckom-RMod}.

\section{Reidemeister invariance maps for the Batson--Seed complex} \label{sec: reidemeister}

In this section, we prove the invariance of the $BS$ complex up to filtered homotopy equivalence, generalizing Batson and Seed's proof of invariance \cite[Proposition 4.2]{BatsonSeed}. More precisely, just as in \cite[Section 4.3]{BarNatanTangles} (and somewhat different from \cite{BatsonSeed}), we will explicitly construct a filtered homotopy equivalence for each Reidemeister move. These maps will then be used in the construction of cobordism maps in Section \ref{sec: cobordism-maps}.

The precise statement is as follows:

\begin{theorem} \label{thm: BS-invariant}
Suppose $T$ and $T'$ are oriented, weighted, shaded tangle diagrams that are related by a sequence of Reidemeister moves (where the weightings and shadings agree). Then there is a filtered, degree-$0$, undotted homotopy equivalence $f \co BS(T) \to BS(T')$ whose associated graded map $f^{(0)}$ agrees with the homotopy equivalence on Khovanov complexes defined by Bar-Natan.
\end{theorem}

\begin{proof}
Thanks to Theorem \ref{thm: BS-extend}, this reduces to a local construction for each individual Reidemeister move, which is then extended by the identity for an arbitrary tangle.

For each of the three Reidemeister moves, let $T$ and $T'$ denote the ``before'' and ``after'' tangles,  equipped with compatible orientations, weightings, and shadings. The weights will be assumed to be arbitrary elements of an arbitrary commutative ring $R$. By Remark \ref{rmk: change-shading}, since the weights are arbitrary, it suffices to consider only one possible checkerboard shading in each case. Moreover, we will write down the homotopy equivalences using a particular choice of orientation and ordering of crossings for each tangle; for an arbitrary choice of orientation and ordering of crossings, we obtain the map by pre- and post-composing with the isomorphisms discussed in Section \ref{ssec: signs}. We refer to \eqref{eq: key} for the symbol convention.

We consider each of the Reidemeister moves separately.

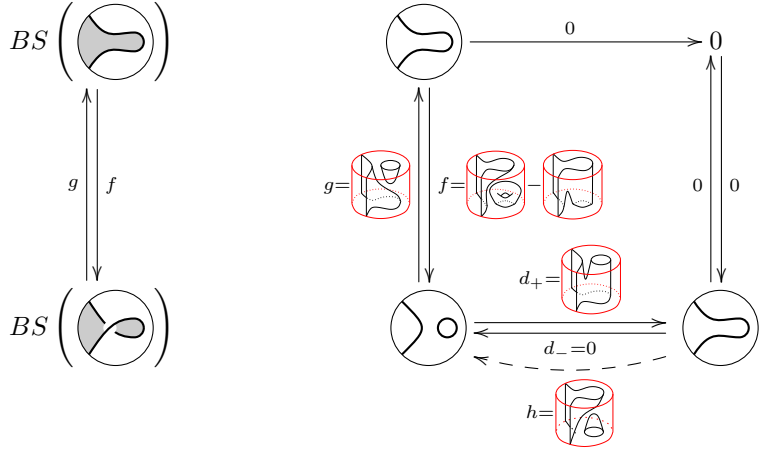
\begin{figure}
\[
\xymatrix @R=1in@C=1in{
BS\left(
\tikzbox{\begin{tikzpicture}[scale=0.5]
\filldraw[black!20!white] (135:1) .. controls (-.25,0) and (0,.25) .. (.5,.25) arc (90:-90:.25) .. controls (0,-.25) and (-.25,0) .. (225:1) arc (225:135:1);
\draw[] (0,0) circle (1);
\draw[thick] (135:1) .. controls (-.25,0) and (0,.25) .. (.5,.25) arc (90:-90:.25) .. controls (0,-.25) and (-.25,0) .. (225:1);
\end{tikzpicture}}\right)
\ar@<2pt>[d]^{f} &
\tikzbox{\begin{tikzpicture}[scale=0.5]
\draw[] (0,0) circle (1);
\draw[thick] (135:1) .. controls (-.25,0) and (0,.25) .. (.5,.25) arc (90:-90:.25) .. controls (0,-.25) and (-.25,0) .. (225:1);
\end{tikzpicture}}
\ar[r]^0
\ar@<2pt>[d]^{f=\vcenter{\hbox{\includegraphics[width=8mm]{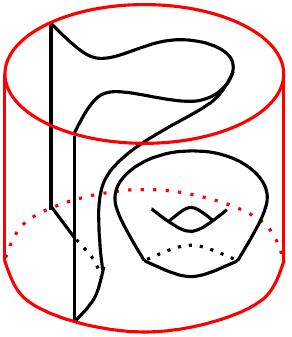}}}-\vcenter{\hbox{\includegraphics[width=8mm]{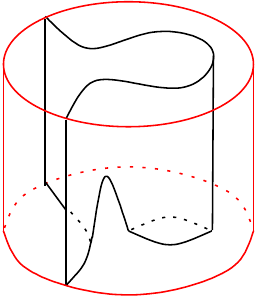}}}}
& 0 \ar@<2pt>[d]^{0} \\
BS\left(
\tikzbox{\begin{tikzpicture}[scale=0.5]
\filldraw[black!20!white] (135:1) .. controls (-.25,0) and (0,-.25) .. (.5,-.25) arc (-90:90:.25) .. controls (0,.25) and (-.25,0) .. (225:1) arc (225:135:1);
\draw[] (0,0) circle (1);
\draw[thick] (135:1) .. controls (-.25,0) and (0,-.25) .. (.5,-.25) arc (-90:0:.25);
\node[crossing] at (-.15,0) {};
\draw[thick] (.75,0) arc (0:90:.25) .. controls (0,.25) and (-.25,0) .. (225:1);
\end{tikzpicture}}\right)
\ar@<2pt>[u]^{g} &
\tikzbox{
\begin{tikzpicture}[scale=0.5]
\draw[] (0,0) circle (1);
\draw[thick] (135:1) .. controls (0,0) .. (225:1);
\draw[thick] (.5,0) circle (.25);
\end{tikzpicture}}
\ar@<2pt>[u]^{g=\vcenter{\hbox{\includegraphics[width=8mm]{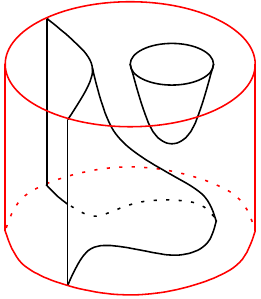}}}} \ar@<2pt>[r]^{d_+=\vcenter{\hbox{\includegraphics[width=8mm]{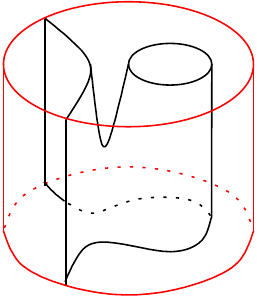}}}} &
\tikzbox{
\begin{tikzpicture}[scale=0.5]
\draw[] (0,0) circle (1);
\draw[thick] (135:1) .. controls (-.25,0) and (0,.25) .. (.5,.25) arc (90:-90:.25) .. controls (0,-.25) and (-.25,0) .. (225:1);
\end{tikzpicture}}
\ar@<2pt>[u]^0
\ar@<2pt>[l]^{d_- = 0}
\ar@{-->}@<4pt>@/^1pc/[l]^{h=\vcenter{\hbox{\includegraphics[width=8mm]{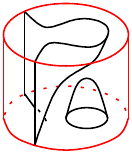}}}}
}
\]
\caption{Reidemeister 1 invariance maps, adapted from \cite[Figure 5]{BarNatanTangles}.} \label{fig: R1}
\end{figure}

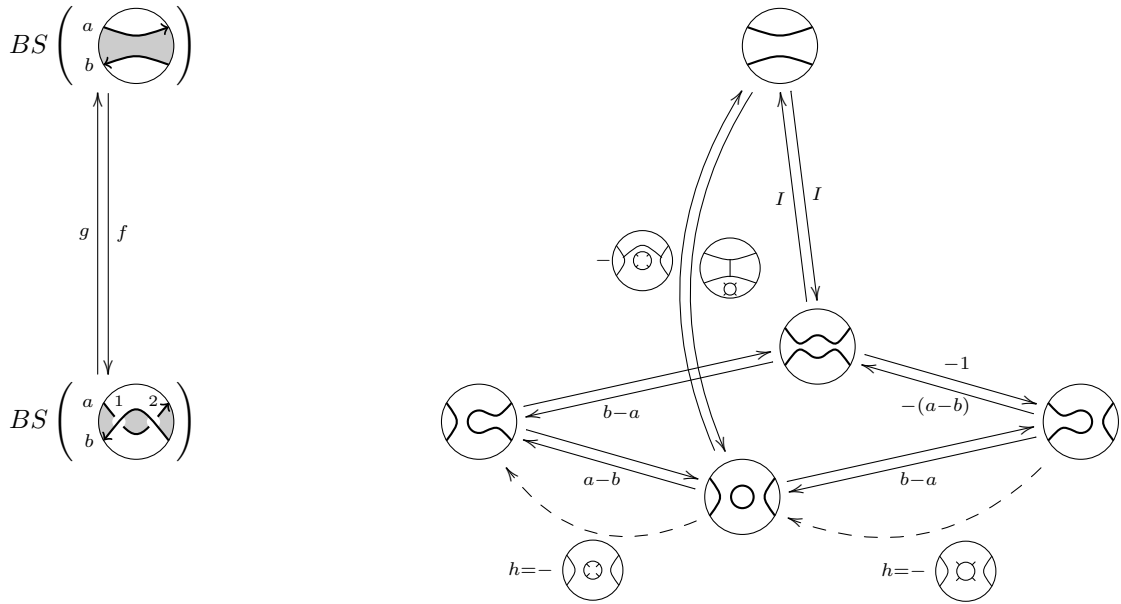
\begin{figure}[htp!]
\[
\xy
(-50,50)*+{BS\left(\tikzbox{\begin{tikzpicture}[scale=0.5]
\filldraw[black!20!white] (30:1) .. controls (0,0.2) .. (150:1) arc (150:210:1) .. controls (0,-0.2) .. (-30:1) arc (-30:30:1);
\draw[] (0,0) circle (1);
\draw[thick, ->] (150:1) .. controls (0,0.2) .. (30:1) node [at start, left] {\scriptsize $a$};
\draw[thick, ->] (-30:1) .. controls (0,-0.2) .. (210:1) node [at end, left] {\scriptsize $b$};
\end{tikzpicture}}\right)}="BST";
(40,50)*+{
\tikzbox{\begin{tikzpicture}[scale=0.5]
\draw[] (0,0) circle (1);
\draw[thick] (30:1) .. controls (0,0.2) .. (150:1);
\draw[thick] (-30:1) .. controls (0,-0.2) .. (210:1);
\end{tikzpicture}}} = "Tres";
(-50,0)*+{BS\left(\tikzbox{\begin{tikzpicture}[scale=0.5]
\filldraw[black!20!white] (30:1) .. controls (0,-.6) .. (150:1) arc (150:210:1) .. controls (0,.6) .. (-30:1) arc (-30:30:1);
\draw[] (0,0) circle (1);
\draw[thick, ->] (150:1) .. controls (0,-.6) .. (30:1) node [at start, left] {\scriptsize $a$};
\node[crossing, label = above:{\tiny $1$}] at (-.46,0) {};
\node[crossing, label = above:{\tiny $2$}] at (.46,0) {};
\draw[thick, ->] (-30:1) .. controls (0,.6) .. (210:1) node [at end, left] {\scriptsize $b$};
\end{tikzpicture}}\right)}="BST'";
(0,0)*+{\begin{tikzpicture}[scale=0.5]
\draw[] (0,0) circle (1);
\draw[thick] (30:1) .. controls (.5,.1) .. (60:.3) arc (60:300:.3) .. controls  (.5,-.1) .. (-30:1);
\draw[thick] (150:1) .. controls (-.5,0) .. (210:1);
\end{tikzpicture}}="T'00";
(45,10)*+{
\begin{tikzpicture}[scale=0.5]
\draw[] (0,0) circle (1);
\draw[thick] (30:1) .. controls (.5,0) .. (45:.3) arc (45:135:.3) .. controls (-.5,0) .. (150:1);
\draw[thick] (-30:1) .. controls (.5,0) .. (-45:.3) arc (-45:-135:.3) .. controls (-.5,0) .. (210:1);
\end{tikzpicture}}="T'10";
(35,-10)*+{\begin{tikzpicture}[scale=0.5]
\draw[] (0,0) circle (1);
\draw[thick] (30:1) ..  controls  (.5,0) .. (-30:1);
\draw[thick] (150:1) .. controls (-.5,0) .. (210:1);
\draw[thick] (0,0) circle (0.3);
\end{tikzpicture}}="T'01";
(80,0)*+{\begin{tikzpicture}[scale=0.5]
\draw[] (0,0) circle (1);
\draw[thick] (150:1) .. controls (-.5,.1) .. (120:.3) arc (120:-120 :.3) .. controls  (-.5,-.1) .. (210:1);
\draw[thick] (30:1) .. controls (.5,0) .. (-30:1);
\end{tikzpicture}}="T'11";
{\ar@<2pt>^{f} "BST";"BST'"};
{\ar@<2pt>^{g} "BST'";"BST"};
{\ar@<2pt> "T'00";"T'01"};
{\ar@<2pt> "T'00";"T'10"};
{\ar@<2pt> "T'01";"T'11"};
{\ar@<2pt>^{-1} "T'10";"T'11"};
{\ar@<2pt>^{b-a} "T'11";"T'01"};
{\ar@<2pt>^{-(a-b)} "T'11";"T'10"};
{\ar@<2pt>^{a-b} "T'01";"T'00"};
{\ar@<2pt>^(0.6){b-a} "T'10";"T'00"};
{\ar@<2pt>^{I} "Tres";"T'10"};
{\ar@<2pt>^{I} "T'10";"Tres"};
{\ar@<2pt>@/_1cm/^{\begin{tikzpicture}[scale=0.4]
\draw[] (0,0) circle (1);
\draw[] (30:1) .. controls (0,0.2) .. (150:1);
\draw[] (-30:1) .. controls (0,-0.2) .. (210:1);
\draw[] (0,-.3) -- (0,.3);
\begin{scope}[rotate around={45:(0,-.7)}]
\draw[] (0,-.7) circle (0.2);
\draw[] (0,-.5) -- (0,-.4);
\draw[] (0,-.9) -- (0,-1);
\draw[] (.2,-.7) -- (.3,-.7);
\draw[] (-.2,-.7) -- (-.3,-.7);
\end{scope}
\end{tikzpicture}} "Tres";"T'01"};
{\ar@<2pt>@/^1cm/^{-\vcenter{\hbox{\begin{tikzpicture}[scale=0.4]
\draw[] (0,0) circle (1);
\draw[] (30:1) ..  controls  (.5,0) .. (-30:1);
\draw[] (150:1) .. controls (-.5,0) .. (210:1);
\begin{scope}[rotate around={45:(0,0)}]
\draw[] (0,0) circle (0.3);
\draw[] (0,0.3) -- (0,0.2);
\draw[] (0,-0.3) -- (0,-0.2);
\draw[] (0.3,0) -- (0.2,0);
\draw[] (-0.3,0) -- (-0.2,0);
\end{scope}
\draw[] (-.65,.1) .. controls (0,.65) .. (.65,.1);
\end{tikzpicture}}}} "T'01";"Tres"};
{\ar@/^1cm/@{-->}^{h=-\tikzbox{
\begin{tikzpicture}[scale=0.4]
\draw[] (0,0) circle (1);
\draw[] (30:1) ..  controls  (.5,0) .. (-30:1);
\draw[] (150:1) .. controls (-.5,0) .. (210:1);
\draw[] (0,0) circle (0.3);
\draw[] (45:0.3)--(45:0.45);
\draw[] (135:0.3)--(135:0.45);
\draw[] (225:0.3)--(225:0.45);
\draw[] (315:0.3)--(315:0.45);
\end{tikzpicture}
}} "T'11";"T'01"};
{\ar@/^1cm/@{-->} ^{h=-\tikzbox{
\begin{tikzpicture}[scale=0.4]
\draw[] (0,0) circle (1);
\draw[] (30:1) ..  controls  (.5,0) .. (-30:1);
\draw[] (150:1) .. controls (-.5,0) .. (210:1);
\draw[] (0,0) circle (0.3);
\draw[] (45:0.3)--(45:0.15);
\draw[] (135:0.3)--(135:0.15);
\draw[] (225:0.3)--(225:0.15);
\draw[] (315:0.3)--(315:0.15);
\end{tikzpicture}
}} "T'01";"T'00"};
\endxy
\]
    \caption{Reidemeister 2 invariance maps, adapted from \cite[Figure 6]{BarNatanTangles}. The weights of the two strands are given by $a$ and $b$, and the small numbers inside the lower tangle indicate the order of the crossings. Each arrow in $d_+$ and $d_-$ indicates a multiple of the appropriate saddle cobordism, with the coefficient indicated on the arrow.}
    \label{fig: R2}
\end{figure}

\subsubsection*{Reidemeister 1}

In the case of a Reidemeister 1 move, let us take $T$ to be the $0$-crossing tangle and $T'$ to be the $1$-crossing tangle. Since the sole crossing in $T'$ occurs between two strands of the same tangle component, the $d_-$ differential on $BS(T')$ vanishes, as do both the $d_+$ and $d_-$ differentials on $BS(T)$. Therefore, the chain maps $f \co BS(T) \to BS(T')$ and $g \co BS(T') \to BS(T)$ defined by Bar-Natan \cite{BarNatanTangles}, which are shown in Figure \ref{fig: R1}, are also chain maps in the sense of $\CKob$. Moreover, $g \circ f = \id$, while $f \circ g$ is homotopic to the identity using the same homotopy $h$ as in \cite{BarNatanTangles}, which is also shown in Figure \ref{fig: R1}.

\subsubsection*{Reidemeister 2}

Similarly, in the case of a Reidemeister 2 move, let us take $T$ to be the $0$-crossing tangle and $T'$ to be the $2$-crossing tangle. In this case, the $d_-$ differential on $BS(T')$ may be nontrivial. Let $f \co BS(T) \to BS(T')$ and $g \co BS(T') \to BS(T)$ be the maps defined by Bar-Natan \cite{BarNatanTangles}, which are shown in Figure \ref{fig: R2}. Bar-Natan verifies that these are chain maps in the sense of $\Kob$, which gives $d_+ f = 0 = f d_+$ and $g d_+ = 0 = d_+ g$, and an almost identical argument (using the fact that the difference of weights is the same at both crossings of $T'$) shows that $d_- f = 0 = f d_- $ and $g d_- = 0 = d_- g$. Thus, $f$ and $g$ are also chain maps in the $\CKob$ sense. Moreover, $g f = \id$ and $f g - \id = d_+ h + h d_+$, where $h \co BS(T') \to BS(T')$ is a homotopy shown in \cite[Figure 6]{BarNatanTangles}. It is easy to check that $d_- h + h d_- = 0$, which then shows that $h$ is also a homotopy in the sense of $\CKob$.

\subsubsection*{Reidemeister 3}

The case of the Reidemeister 3 move is the most interesting. Let $T$ and $T'$ denote the ``before'' and ``after'' tangles, equipped with weights, shadings, orientations, and order of crossings as in Figure \ref{fig: R3}. Because the first two crossings (as ordered) are positive and the third is negative, our formula for sprinkling signs agrees with Bar-Natan's (see Section \ref{ssec: signs}). Let $f_0 \co BS(T) \to BS(T')$ be the Reidemeister 3 map constructed by Bar-Natan, which is represented by the blue arrows in Figure \ref{fig: R3}. (For simplicity, we have suppressed the labels of most of the arrows in the $d_{\pm}$ differentials on $BS(T)$ and $BS(T')$.) Bar-Natan showed that $f_0$ is a chain map (in the sense of $\Kob$) with respect to the ordinary Khovanov differential $d_+$; that is, $f_0 d_+ = d_+ f_0$. However, observe that $f_0 d_- \ne d_- f_0$. For instance, if we consider only the components going from the $111$ resolution of $T$ to the $011$ resolution of $T'$, we have $f_0 d_-=0$ but $d_- f_0 = b-c$ times a saddle cobordism.

To resolve this issue, let $f_1$ consist of the red arrows in Figure \ref{fig: R3}. It is easy to verify that $f_1$ commutes with $d_-$ and that the failure of $f_1$ to commute with $d_+$ exactly cancels the failure of $f_0$ to commute with $d_-$:
\[
f_1 d_+ - d_+ f_1 = d_- f_0 - f_0 d_-.
\]
For instance, going from the $111$ resolution of $T$ to the $011$ resolution of $T'$, we have $f_1 d_+ = 0$ and $d_+ f_1 = -(b-c)$ times a saddle cobordism, which exactly cancels the contribution from $d_- f_0$ given above. Verifying the remaining components is left to the reader as an exercise; some components involve using the local relations from Figure \ref{fig: local-nodots}. Defining $f = f_0 + f_1$, we thus have $df = fd$, so $f$ is a chain map in $\CKob$. Note that $f$ is homogeneous of degree $0$ with respect to the Batson--Seed grading, and it is $0$-filtered (in the sense of Definition \ref{def: filtered}) with associated graded map equal to $f_0$.

To see why $f$ is a homotopy equivalence, we first note that up to checkerboard shading and orientations, the source and target tangles of the third Reidemeister move are rotations of one another by $\pi$. Thus, up to changing the signs of some arrows, we can obtain the chain map $g$ for the reverse Reidemeister 3 move by simply rotating each resolution in Figure \ref{fig: R3} by $\pi$ and taking all the same arrows (with the slight change that the weights $a,b,c$ are negated). The composition $g \circ f$ is then given by the blue and red arrows in Figure \ref{fig: R3-homotopy}. (Note that several potential terms in this composition terms vanish due to the sphere relation.) The green dashed arrows in the figure then provide a chain homotopy between $g \circ f$ and the identity, as the reader can verify. An almost identical argument applies to $f \circ g$, which concludes the proof.
\end{proof}

\begin{figure}
\labellist
\small
\pinlabel {{\color{blue} $I$}}  [tr] at 209 159
\pinlabel {{\color{blue} $I$}}  [bl] at 332 173
\pinlabel {{\color{blue} $I$}}  [bl] at 404 182
\pinlabel {{\color{blue} $I$}}  [bl] at 449 182
\pinlabel {{\color{blue} $-I$}}  [bl] at 546 182
\pinlabel {{\color{red} $(b-c)$}}  [r] at 510 169
\pinlabel {{\color{red} $(b-c)I$}}  [br] at 206 106
\pinlabel $a$ [r] at 18 249
\pinlabel $b$ [br] at 24 265
\pinlabel $c$ [bl] at 62 265
\pinlabel $a$ [r] at 18 58
\pinlabel $b$ [br] at 24 77
\pinlabel $c$ [bl] at 62 77
\scriptsize
\pinlabel $b-c$ [tr] at 172 226
\pinlabel $c-a$ [b] at 193 247
\pinlabel $a-b$ [br] at 210 270
\pinlabel $-$ [b] at 296 204
\pinlabel $-$ [b] at 311 260
\pinlabel $-$ [b] at 341 234
\pinlabel $-$ [b] at 465 246
\pinlabel $-$ [b] at 357 14
\pinlabel $-$ [b] at 401 42
\pinlabel $-$ [t] at 371 72
\pinlabel $-(a-c)$ [b] at 529 56
\pinlabel $c-b$ [bl] at 534 83
\pinlabel $b-a$ [tl] at 514 37
\tiny
\pinlabel $1$ at 48 245
\pinlabel $2$ at 54 251
\pinlabel $3$ at 32 251
\pinlabel $1$ at 43 67
\pinlabel $2$ at 37 50
\pinlabel $3$ at 48 50
\pinlabel $000$ [b] at 118 264
\pinlabel $100$ [b] at 209 222
\pinlabel $010$ [b] at 241 264
\pinlabel $001$ [b] at 275 306
\pinlabel $110$ [b] at 366 222
\pinlabel $101$ [b] at 399 264
\pinlabel $011$ [b] at 433 306
\pinlabel $111$ [b] at 524 264
\endlabellist
\includegraphics[width=6in]{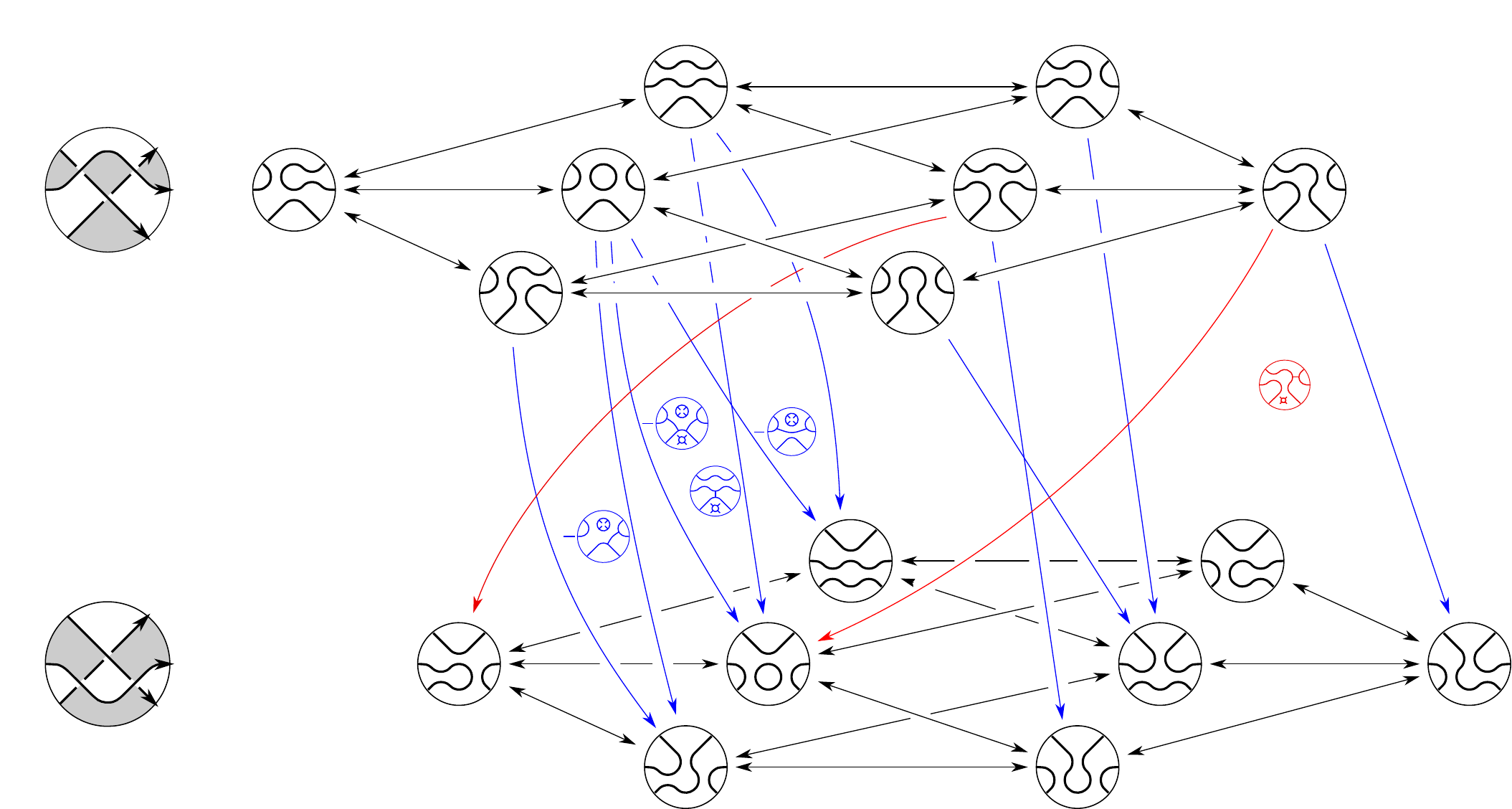}
\caption{The Reidemeister 3 invariance map, adapted from \cite[Figure 9]{BarNatanTangles}. The weights of the three tangle components are $a$, $b$, and $c$ as shown, and the numerals in the tangle diagrams indicate the order of the crossings. The resolutions of $T$ are decorated with their corresponding binary sequences. For conciseness, we have denoted the terms in $d_+$ and $d_-$ with double-headed arrows (cf.~Figures \ref{fig: R1} and \ref{fig: R2}), and we have omitted the factors $s(c_i) (w_{\text{over}} - w_{\text{under}})$ for most of the $d_-$ terms apart from the ones in the upper left and lower right of the figure. The $-$ signs on certain arrows denote the signs $(-1)^{m(v,v')}$ that appear in both $d_+$ and $d_-$ differentials. (Note that the horizontal arrow in the lower-right corner includes both this $-$ sign and the factor of $s(c_i) (w_{\text{over}} - w_{\text{under}}) = a-c$.)}
\label{fig: R3}
\end{figure}
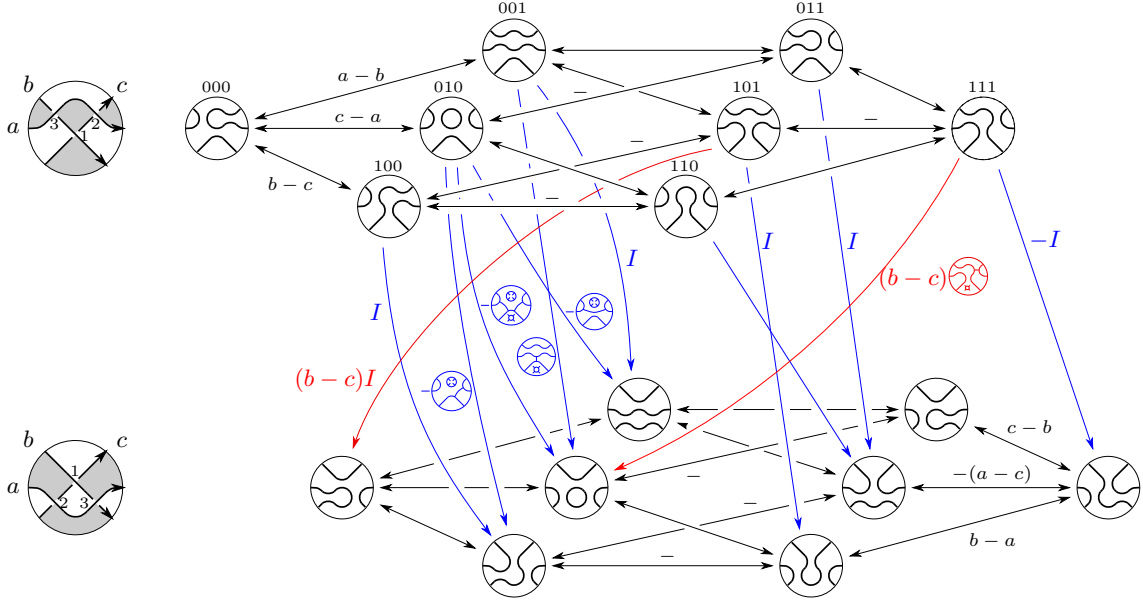

\begin{figure}
\labellist
\small
\pinlabel $a$ [r] at 18 249
\pinlabel $b$ [br] at 24 265
\pinlabel $c$ [bl] at 62 265
\pinlabel $a$ [r] at 18 58
\pinlabel $b$ [br] at 24 77
\pinlabel $c$ [bl] at 62 77
\scriptsize
\pinlabel {{\color{blue} $I$}}  [l] at 435 172
\pinlabel {{\color{blue} $I$}}  [tr] at 388 159
\pinlabel {{\color{blue} $I$}}  [tr] at 418 204
\pinlabel {{\color{blue} $I$}}  [bl] at 564 138
\pinlabel {{\color{blue} $I$}}  [r] at 212 136
\pinlabel {{\color{blue} $I$}}  [l] at 348 145
\pinlabel {{\color{red} $(c-b)I$}}  [br] at 217 93
\pinlabel {{\color{red} $(b-c)$}}  [l] at 482 169
\pinlabel $b-c$ [tr] at 172 226
\pinlabel $c-a$ [b] at 193 247
\pinlabel $a-b$ [br] at 210 270
\pinlabel $-$ [b] at 296 204
\pinlabel $-$ [b] at 311 260
\pinlabel $-$ [b] at 327 231
\pinlabel $-$ [b] at 465 246
\pinlabel $-$ [b] at 357 14
\pinlabel $-$ [b] at 401 42
\pinlabel $-$ [t] at 371 72
\pinlabel $-(c-a)$ [b] at 529 56
\pinlabel $b-c$ [bl] at 534 82
\pinlabel $a-b$ [tl] at 514 39
\tiny
\pinlabel $1$ at 48 245
\pinlabel $2$ at 54 251
\pinlabel $3$ at 32 251
\pinlabel $1$ at 48 55
\pinlabel $2$ at 54 61
\pinlabel $3$ at 32 61
\pinlabel $000$ [b] at 118 264
\pinlabel $100$ [b] at 209 222
\pinlabel $010$ [b] at 241 264
\pinlabel $001$ [b] at 275 306
\pinlabel $110$ [b] at 366 222
\pinlabel $101$ [b] at 401 264
\pinlabel $011$ [b] at 433 306
\pinlabel $111$ [b] at 524 264
\endlabellist
\includegraphics[width=6in]{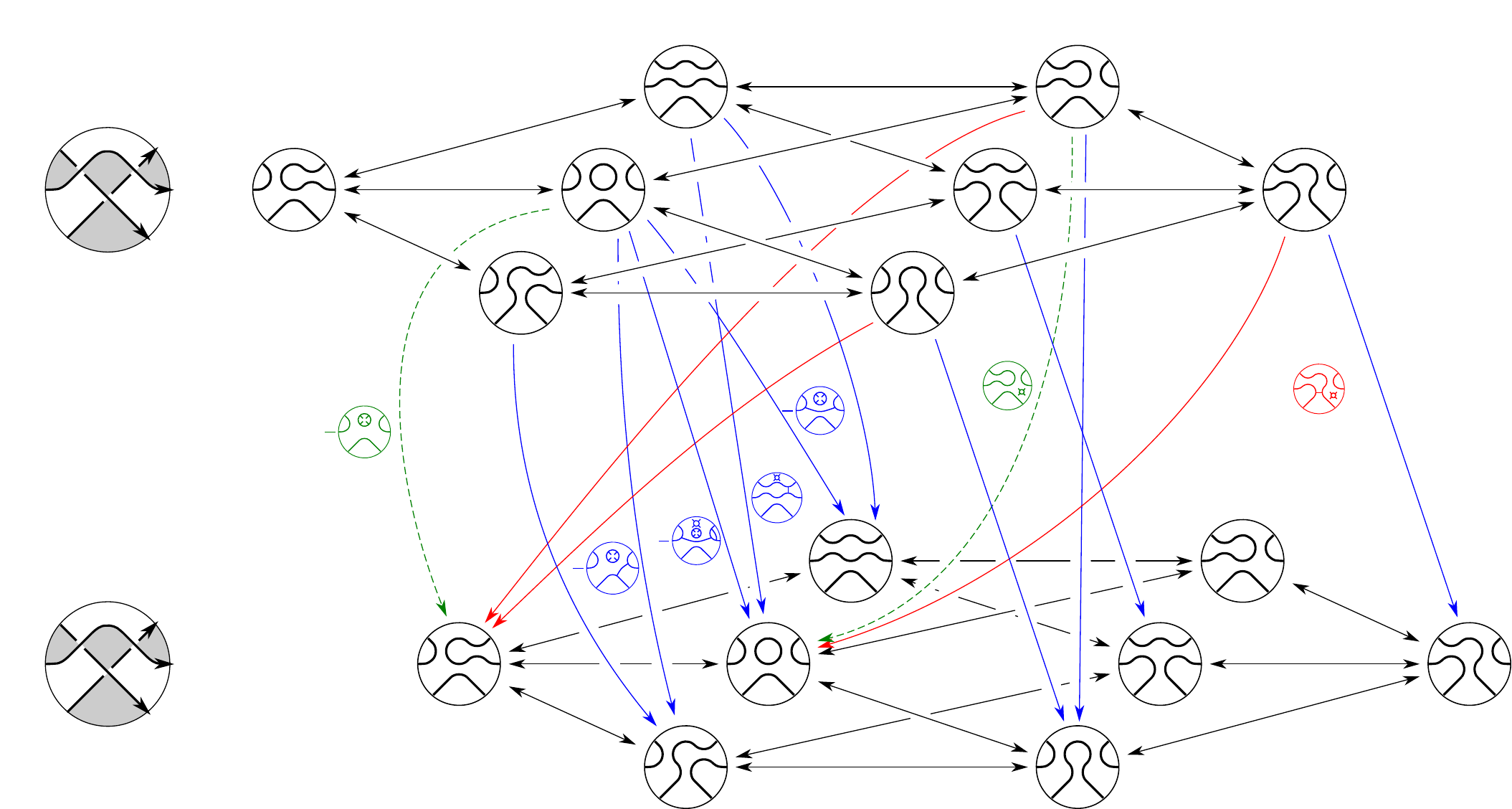}
\caption{A homotopy between the identity and the composite of two inverse Reidemeister 3 maps. The red $(c-b)I$ label in the lower left applies to both of the nearby red arrows.}
\label{fig: R3-homotopy}
\end{figure}
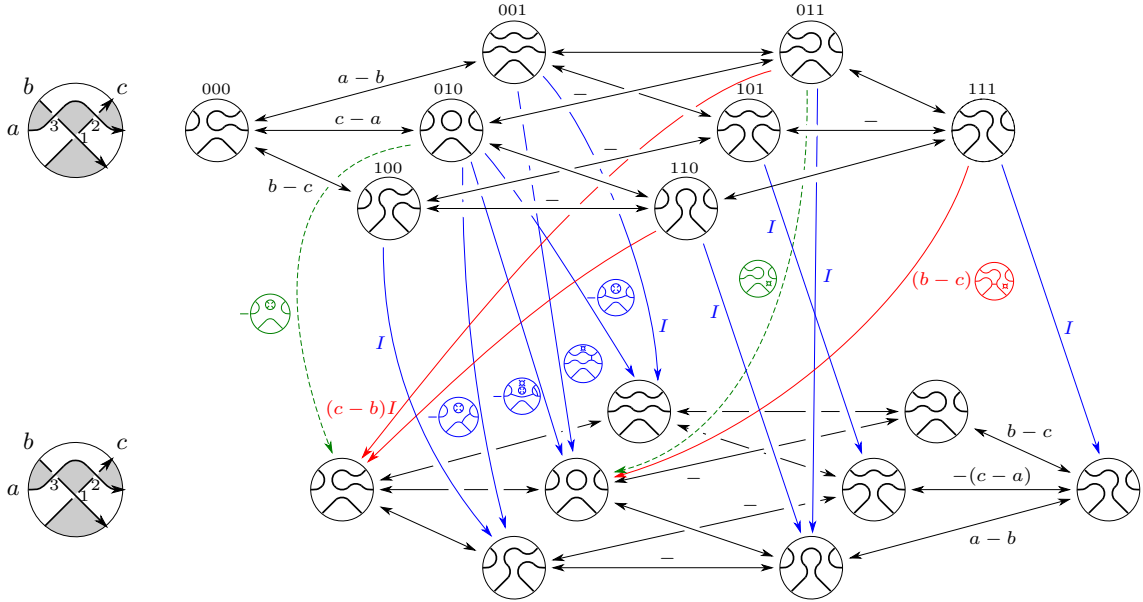

\begin{remark} \label{rmk: lessimp}
The structure of the Reidemeister 3 map makes clear why we needed to define chain maps in $\CKom$ to include terms with nontrivial shifts in the homological grading. Furthermore, the fact that these maps end up being homogeneous with respect to the internal grading of $\CKob$ shows why this grading (rather than the quantum grading) is the correct choice.
\end{remark}

\subsection{Simple tangles} \label{ssec: simple}

Just as in \cite{BarNatanTangles}, our argument will rely on showing that the $BS$ complexes of certain simple tangles admit no non-trivial automorphisms, as we now discuss. The proofs follow Bar-Natan's template nearly verbatim, but in the modified setting of our categories $\CKob$ and $\CKob^f$.

\begin{definition} \label{def: simple}
Let $(\Omega, d_+, d_-)$ be an object of $\CKob$, and consider $(\Omega, d_+)$ as an object in $\Kob$. For $\CC \in \{\CKob, \CKob^f, \Kob\}$, we say that $\Omega$ is \emph{$\CC$-simple} if every degree-$0$ endomorphism of $\Omega$ in $\CC$ is homotopic to a unique scalar multiple of the identity, i.e., if the map
\[
R \to \Mor_{\CC_{/h}}(\Omega, \Omega)_0, \qquad r \mapsto [r \id_\Omega]
\]
is an isomorphism.

A weighted tangle $(T,w)$ is called \emph{$BS$-simple} if $BS(T,w)$ is $\CC$-simple for all three choices of $\CC$.
\end{definition}

In particular, observe that any degree-$0$ self-homotopy equivalence of a $\CC$-simple complex must be (filtered) homotopic to a \emph{unit} scalar multiple of the identity (cf. \cite[Definition 8.5]{BarNatanTangles}).

The following lemma is an easy exercise in diagram-chasing:

\begin{lemma} \label{lemma: simple-complex}
Let $\Omega, \Omega'$ be objects in $\CKob$.
\begin{enumerate}
\item If $\Omega$ is $\CC$-simple, and $f \co \Omega \to \Omega'$ is a degree-$0$ $\CC$-homotopy equivalence (i.e., a morphism in $\CC$ that is an isomorphism in $\CC_{/h}$), then $\Omega'$ is $\CC$-simple.

\item If $\Omega$ and $\Omega'$ are (filtered) simple, and $f, g \co \Omega \to \Omega'$ are degree-$d$ $\CC$-homotopy equivalences, then $f$ and $g$ are $\CC$-homotopic up to multiplication by a unit in $R$.
\end{enumerate}
\end{lemma}

\begin{definition} \label{def: pairing}
A \emph{pairing} is a tangle diagram without crossings or closed components. We say that two tangles in $B^3$ are \emph{freely isotopic} if they are isotopic, not necessarily rel boundary; that is, the endpoints are permitted to move. A tangle $T \subset B^3$ is called \emph{trivial} if it is freely isotopic to a pairing.
\end{definition}

Because the mapping class group of an $n$-punctured sphere is generated by Dehn twists, two tangle diagrams represent freely isotopic tangles iff they are related by a sequence of Reidemeister moves plus the operation of adding or removing a crossing between strands at the boundary, as in Figure \ref{fig: add-crossing}. The main result of this section is:

\begin{lemma} \label{lemma: trivial-tangle}
Any trivial tangle is $BS$-simple.
\end{lemma}

\begin{figure}
\begin{subfigure}[c]{0.2\textwidth}
\begin{tikzpicture}
\draw[] (0,0) circle (1);
\draw[thick] (30:.5) -- (30:1);
\draw[thick] (-30:.5) -- (-30:1);
\draw[thick] (90:.5) -- (90:1);
\draw[thick] (-90:.5) -- (-90:1);
\draw[thick] (150:.5) -- (150:1);
\draw[thick] (-150:.5) -- (-150:1);
\node [] at (0,0) {$T$};
\end{tikzpicture}
\end{subfigure}
\begin{subfigure}[c]{0.3\textwidth}
\begin{tikzpicture}
\draw[] (0,0) circle (1);
\draw[thick] (30:.5) -- (30:1)  arc (120:45:.5) -- +(1,-1);
\node[crossing] at (1.89,0) {};
\draw[thick] (-30:.5) -- (-30:1) arc (240:315:.5) -- +(1,1);
\draw[thick] (90:.5) -- (90:1) -- (90:1.4);
\draw[thick] (-90:.5) -- (-90:1) -- (-90:1.4);
\draw[thick] (150:.5) -- (150:1) -- (150:1.2);
\draw[thick] (-150:.5) -- (-150:1) -- (-150:1.2);
\node [] at (0,0) {$T$};
\draw[] (0.7,0) ellipse [x radius = 1.9, y radius = 1.5];
\end{tikzpicture}
\end{subfigure}
\begin{subfigure}[c]{0.3\textwidth}
\begin{tikzpicture}
\draw[] (0,0) circle (1);
\draw[thick] (30:1)  arc (120:45:.5) -- +(1,-1);
\draw[thick] (-30:1) arc (240:315:.5) -- +(1,1);
\filldraw[white] (1.89,0) circle (0.5);
\draw[] (1.89,0) circle (0.5);
\draw[thick] (90:1) -- (90:1.4);
\draw[thick] (-90:1) -- (-90:1.4);
\draw[thick] (150:1) -- (150:1.2);
\draw[thick] (-150:1) -- (-150:1.2);
\draw[] (0.7,0) ellipse [x radius = 1.9, y radius = 1.5];
\end{tikzpicture}
\end{subfigure}
\caption{Adding a crossing to the outside of a $k$-strand tangle $T$ (left) to produce a new tangle $TX$ (right), via a planar arc diagram $D_k$ (right).}
\label{fig: add-crossing}
\end{figure}
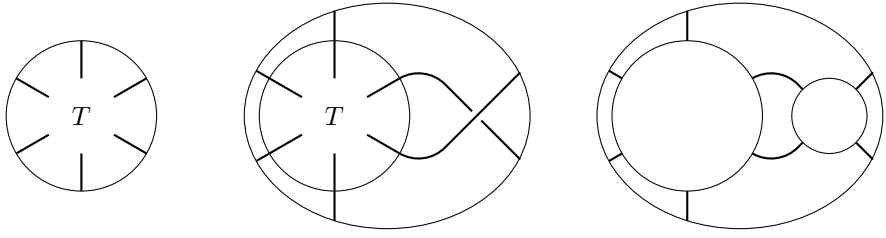

\begin{proof}
The proof follows exactly along the lines of \cite[Lemmas 8.6--8.9]{BarNatanTangles}. We briefly summarize this argument. The same argument works for all three choices of $\CC$; the case of $\Kob$ is Bar-Natan's argument.

First, we show that any pairing is $BS$-simple. Observe that if $T$ is a pairing, then $BS(T)$ consists of a single tangle (namely $T$, without orientation) in homological grading $0$, with vanishing differential, so it agrees with Bar-Natan's formal Khovanov complex $\Kh(T)$. Moreover, any morphism in $\Mor_{\CKob}(T,T)$ or $\Mor_{\CKob^f}(T,T)$ is homogeneous with respect to the homological grading. The proof of \cite[Lemma 8.6]{BarNatanTangles} then goes through (in all three categories) to show that any such morphism is a multiple of the identity. For the uniqueness in Definition \ref{def: simple}, first note that $T$ can be closed up with a planar arc diagram to obtain an unknot $U$. Formally, we can take a one-input planar arc diagram $D$, with no strands going to the outer boundary, such that $D(T)=U$. Then $D$ gives a functor from $\CKob_{/h}(B)$ to $\CKob_{/h}(\emptyset)$. The composition of this functor with the Khovanov TQFT takes $T$ to $\Kh(U)$, which is a free $R$-module of rank $2$. If $r \id_T = 0$ in $\Mor_{\CKob_{/h}}(T,T)$, then $r \id_{\Kh(U)} = 0$, so $r=0$. Thus, we deduce that $\Mor_{\CKob_{/h}}(T,T) = R$, as required. The same argument works in $\CKob^f$.

Second, we deduce from Lemma \ref{lemma: simple-complex} that if $T$ is $BS$-simple and $T'$ is obtained from $T$ by Reidemeister moves, then $T'$ is $BS$-simple. (See \cite[Lemma 8.7]{BarNatanTangles}.) The key point is that the Reidemeister maps constructed above (as well as their homotopy inverses) are degree-$0$, filtered homotopy equivalences.

Third, if $T$ is any tangle, let $TX$ denote the tangle obtained by adding a crossing as above, equipped with appropriate weighting and shading. Formally, we may compute $BS(TX)$ from $BS(T)$ by ``tensoring'' with $BS$ of a $1$-crossing tangle, using the planar arc diagram shown in Figure \ref{fig: add-crossing}. The argument from \cite[Lemmas 8.7 and 8.8]{BarNatanTangles} applies to show that $T$ is $BS$-simple if and only if $TX$ is. Further details are left to the reader.
\end{proof}

\section{Cobordism maps on the Batson--Seed complex} \label{sec: cobordism-maps}

In this section we will establish functoriality for our Batson--Seed tangle invariant. Namely, we describe how weighted cobordisms between links or tangles in $4$ dimensions induce filtered chain maps between the $BS$ complexes at either end. We then prove that these maps are invariant (up to filtered homotopy and multiplication by a unit) under isotopy rel boundary of the cobordisms. This argument is essentially parallel to Bar-Natan's argument for Khovanov homology \cite[Section 8]{BarNatanTangles}.

\subsection{Construction of the maps} \label{ssec: cobordism-construction}

To begin, recall that any smooth cobordism $C \subset \R^3 \times [0,1]$ between links $L, L'$ can be represented by a \emph{movie}, a sequence of link diagrams $L = L_0, \dots, L_k = L'$, where each $L_i$ is obtained from $L_{i-1}$ by a single \emph{elementary cobordism}: a Reidemeister move, birth, saddle, or death, as shown in Figure \ref{fig: elem}. Similarly, if $T,T'$ are oriented tangles in $D^3$ with $\partial T = \partial T'$, a cobordism from $T$ to $T'$ is a properly embedded, oriented surfaces $C \subset D^3 \times [0,1]$ with
\[
\partial C = (-T \times \{0\}) \cup (\partial T \times [0,1]) \cup (T \times \{1\}).
\]
A tangle cobordism can likewise be described by a series of tangle diagrams in $D^2$ with fixed boundary, each obtained from the previous one by an elementary cobordism. (See the book by Carter and Saito \cite{CarterSaitoBook} for a complete exposition.)

\begin{figure}[htp]
	\centering
	\includegraphics[width=12cm]{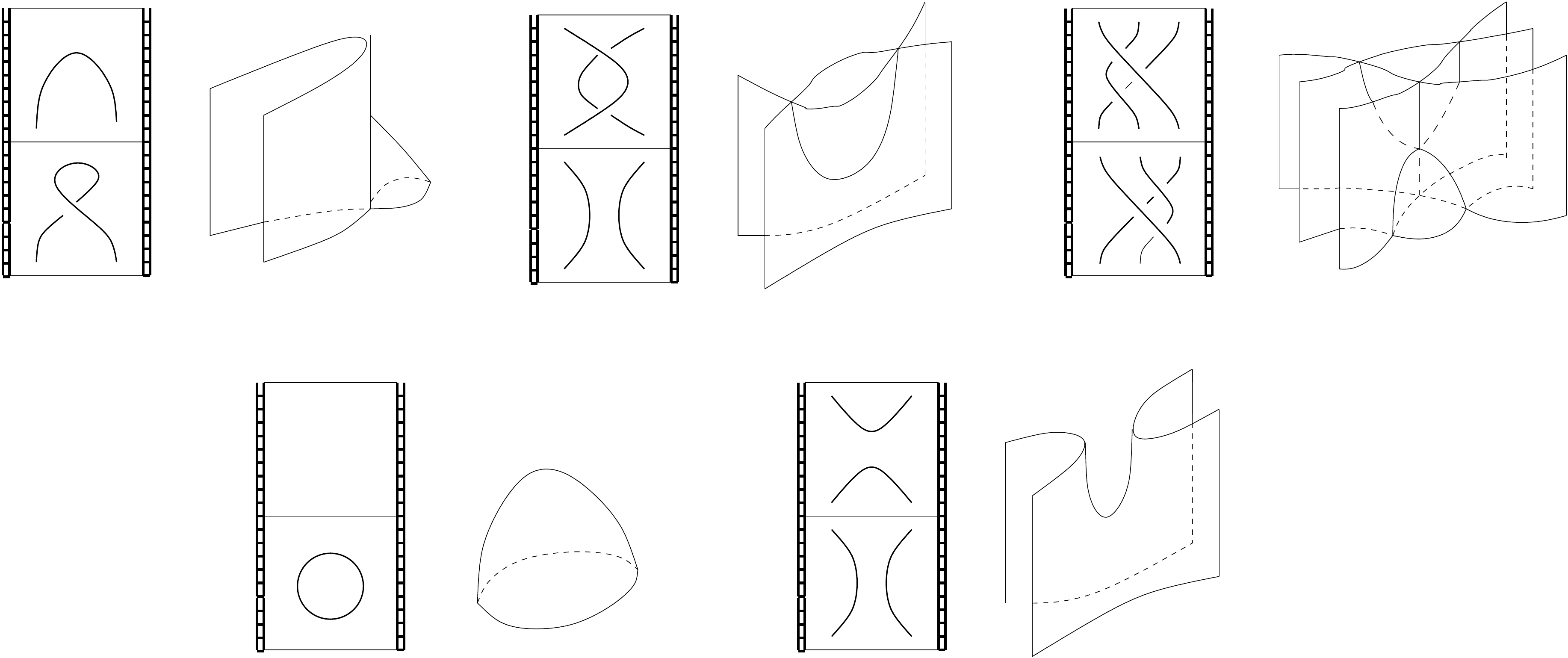}
	\caption{Elementary cobordisms; picture taken from \cite [Figure 10]{BarNatanTangles}.}
	\label{fig: elem}
\end{figure}

A \emph{weighted cobordism} is a link or tangle cobordism $C$ along with a weighting $w$ that assigns an element of $R$ to each component of $C$. A \emph{weighted movie} is a movie
\[
T_0 \xrightarrow{e_1} T_1 \xrightarrow{e_2} \dots \xrightarrow{e_k} T_k
\]
along with weights on each each $T_i$, such that the weights of any components in $T_{i-1}$ and $T_i$ that are connected by a component of $e_i$ agree. In particular, if $e_i$ is a saddle move that merges two components of $T_{i-1}$ into one component of $T_i$ or vice versa, then all three tangle components involved must have the same weight. Note that any movie for a weighted cobordism naturally becomes a weighted movie (where each component of each frame acquires the weight of the component of $C$ on which it lies), and any weighted movie determines a weighted cobordism.

For each weighted elementary cobordism $e$ between tangles $T,T'$, we define a morphism
\[
BS(e) \co BS(T,w) \to BS(T',w')
\]
as follows. The tangles $T$ and $T'$ are identical outside of a local ``region of interest,'' in which the two tangles differ as in Figure \ref{fig: elem}. In each case, we may define a chain map on the $BS$ complexes of the local region (using orientation, weighting, and shading induced from the larger tangles), and then extend it to the larger tangles using Theorem \ref{thm: BS-extend}, just like in \cite{BarNatanTangles}. More specifically:

\begin{enumerate}
\item For the elementary cobordism that corresponds to one of the three Reidemeister moves, we define the induced chain map on the $BS$ complex to be the homotopy equivalence constructed in Section \ref{sec: reidemeister}.

\item For the elementary cobordism that corresponds to a birth or a death, we define the induced chain map on the $BS$ complex locally as the birth or death cobordism between the empty tangle and the tangle consisting of a single closed circle. In either case, the weight of the closed circle is permitted to be an arbitrary element of $R$.

\item For a saddle cobordism, observe that both strands in both the source and target pictures are required to have the same weight, since they are all in the same component of the cobordism. The map $BS(T) \to BS(T')$ is defined to consist of a standard saddle cobordism (in $D^2 \times [0,1]$).
\end{enumerate}

We also note that $BS(e)$ is homogeneous of degree $p$ (with respect to the Batson--Seed grading) and $p$-filtered, where $p$ equals $0$ in the case of a Reidemeister move, $+1$ in the case of a birth or death, and $-1$ in the case of a saddle. The associated graded morphism in $\Kob$ (see Lemma \ref{lemma: filtered}) is given by the terms that preserve homological grading, which coincide precisely with the Khovanov map associated to $e$ in \cite{BarNatanTangles}, which we will denote $\Kh(e)$. Indeed, for all but the Reidemeister 3 move, we have $\Kh(e) = BS(e)$, since there are no terms that strictly decrease the homological grading.

Now, for a weighted cobordism $C$ from $T$ to $T'$, we represent $C$ by a weighted movie as above, and then define $BS(C)$ to be the composition of the maps induced by the corresponding elementary cobordisms $BS(e_i)$. Just as in the previous paragraph, the associated graded morphism is precisely $\Kh(C)$, as defined by Bar-Natan.

\subsection{Isotopy invariance} \label{ssec: isotopy}

We will now show that the $BS$ cobordism maps constructed in the previous section are invariant under isotopy rel boundary, which is precisely the content of Theorem \ref{thm: BS-func}, parts 1 and 2.

\begin{figure}
	\centering
	\includegraphics[width=9cm]{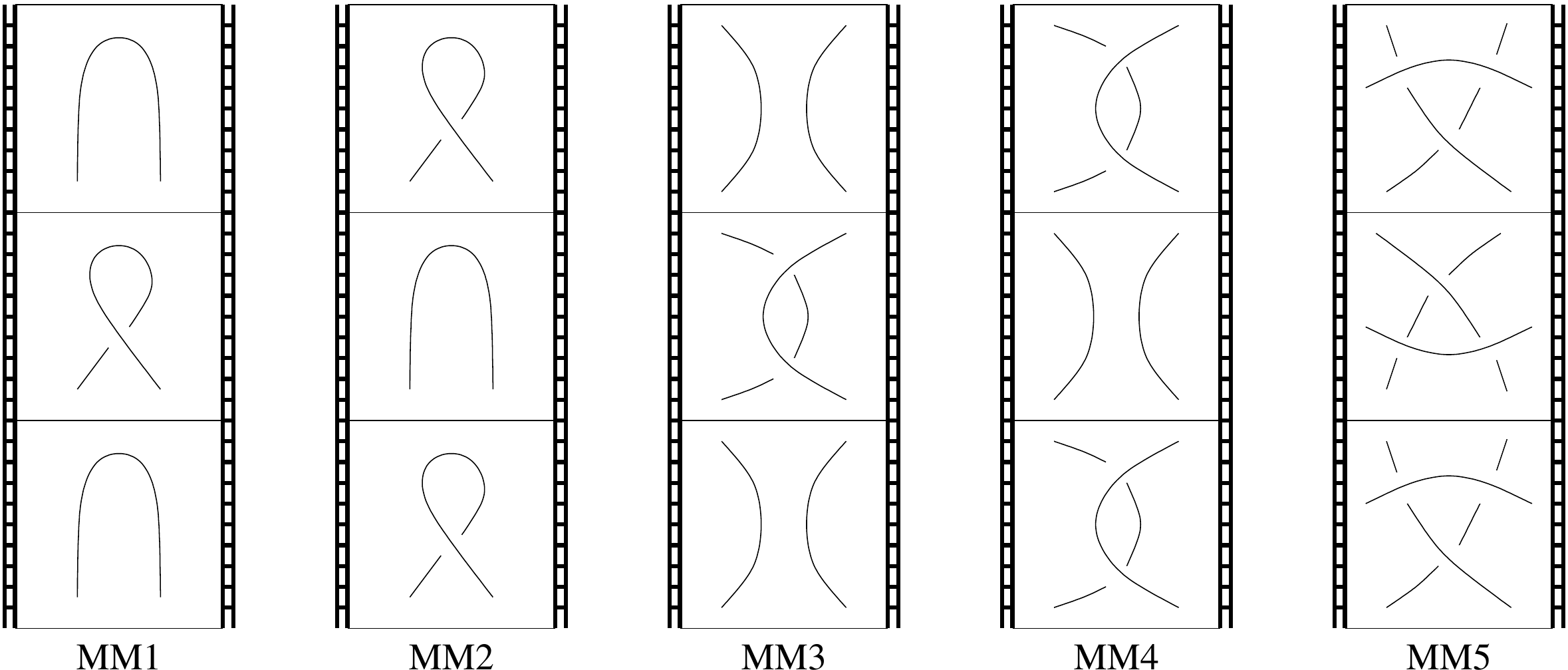}

\bigskip \bigskip

	\includegraphics[width=12cm]{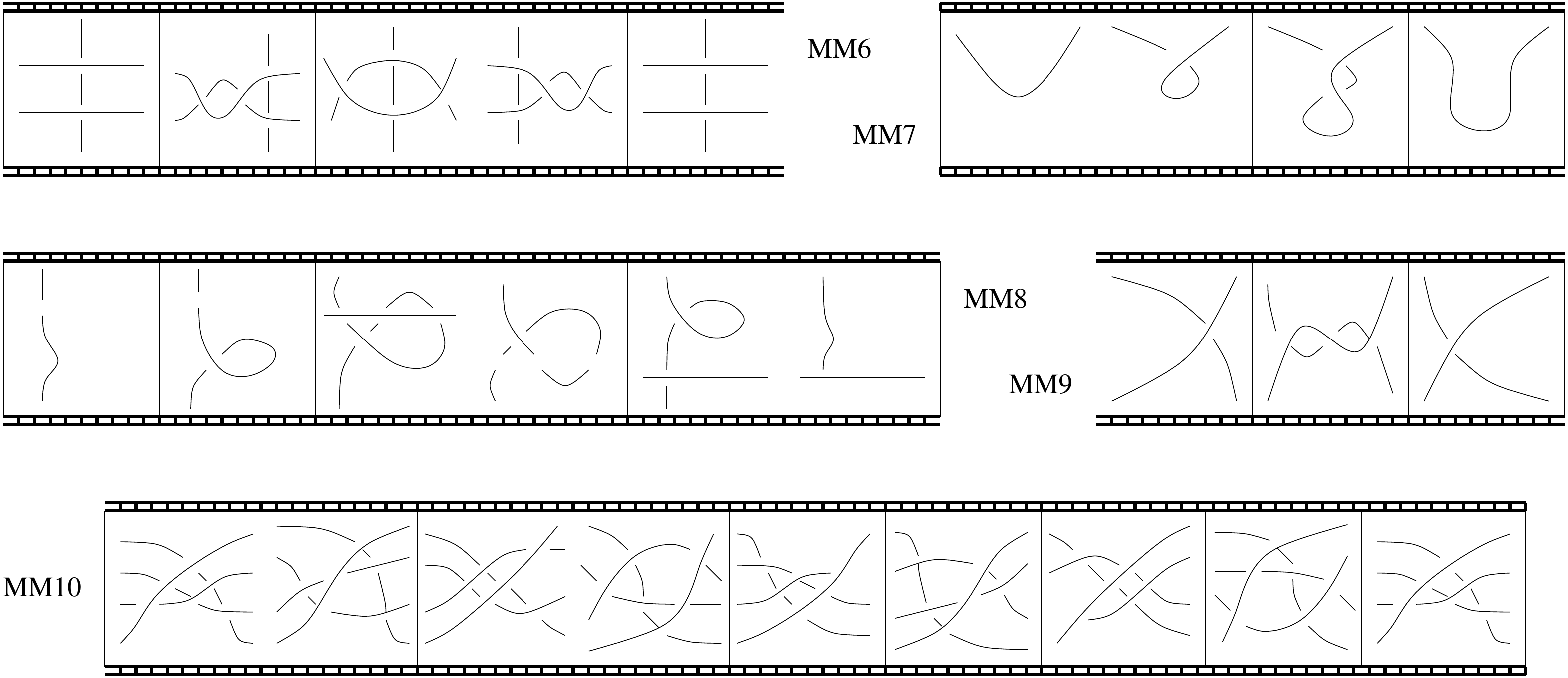}

\bigskip \bigskip

	\includegraphics[width=15cm]{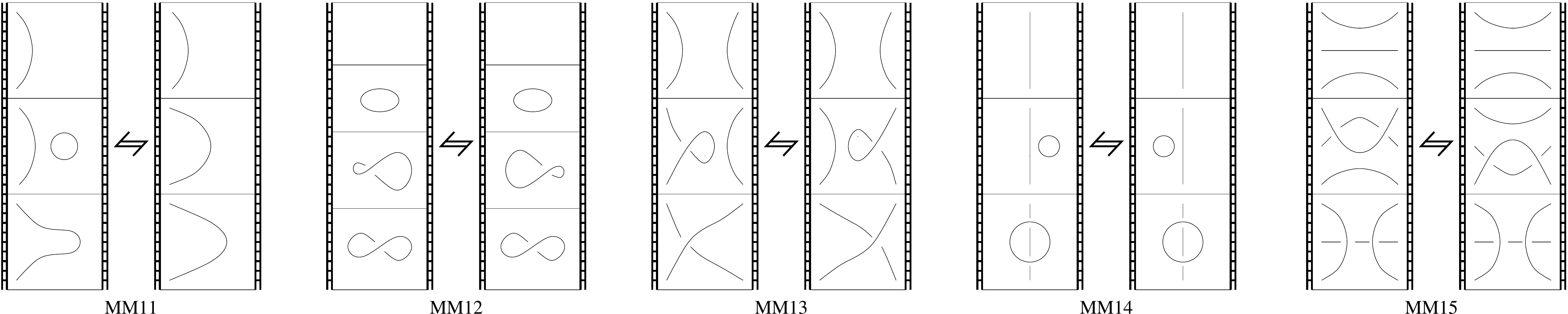}
\caption{Movie moves; figure taken from \cite[Figures 11--13]{BarNatanTangles}.}
\label{fig: movie-moves}
\end{figure}

Just as in \cite{KhovanovCobordism, JacobssonCobordisms, BarNatanTangles} and numerous other papers, the proof of isotopy invariance reduces to a verification for each of the \emph{movie moves} of Carter and Saito \cite{CarterSaitoBook}, shown in Figure \ref{fig: movie-moves}. That is, two movies represent isotopic cobordisms (rel boundary) iff they differ by a sequence of movie moves. (To clarify the meaning of these figures, each of the movies in MM1 through MM10 can be replaced by the constant movie, while each of the pairs of movies in MM11 through MM15 can be interchanged. In each case, the movies can be read in either direction.) The technical statement of invariance is the following:

\begin{theorem} \label{thm: movie-moves}
Let $L_0$ and $L_1$ be weighted tangles, and let $C$ and $C'$ be weighted cobordisms from $L$ to $L'$, presented by weighted movies $M$ and $M'$ respectively. If $M$ and $M'$ are related by a sequence of movie moves, then $BS(M)$ and $BS(M')$ are filtered homotopic up to multiplication by $\pm 1$.
\end{theorem}

\begin{proof}
The argument is very similar to the argument in \cite[Section 8]{BarNatanTangles}. By Theorem \ref{thm: BS-extend}, it suffices to give a local verification for each of the tangles appearing in the movie moves. Since the $BS$ chain maps are dependent on the weighting information, this means that we need to show that they are unchanged (up to filtered homotopy and multiplication by $\pm 1$) for an arbitrary weighting of each diagram that is compatible with the definition of a weighted cobordism. We now consider each of the cases:

\begin{itemize}
\item
The movies in MM1 through MM5 each consist of a Reidemeister move followed by its inverse. As seen in Section \ref{sec: reidemeister}, the composition is filtered homotopic to the identity, as required.

\item
The movies in MM6 through MM10 are each composed entirely of Reidemeister moves, so in each case, the induced map $f \co BS(T,w) \to BS(T,w)$ is a filtered, degree-$0$, undotted homotopy equivalence, where $(T,w)$ is the weighted tangle at the beginning and end of the movie. By Lemma \ref{lemma: trivial-tangle}, $T$ is $BS$-simple, so $f$ is filtered-homotopic to a unit scalar multiple of the identity, say $r \id$. Therefore, the associated graded map $f^{(0)}$ is homotopic to $r \id$ in the sense of $\Kob$. By Theorem \ref{thm: BS-invariant}, $f^{(0)}$ is precisely the map on formal Khovanov complexes defined by Bar-Natan, which is homotopic to $\pm \id$ \cite[p.~1479]{BarNatanTangles}. By the uniqueness provision in our definition of $BS$-simplicity, we deduce that $r = \pm 1$, as required.

\item
For MM11, MM12, and MM13, observe the weights of all components in each tangle picture must be the same. Hence, the $d_-$ differentials vanish, and the $BS$ complexes and morphisms coincide exactly with the $\Kh$ complexes and morphisms from \cite{BarNatanTangles}. Therefore, the explicit homotopies constructed in \cite[Page 1480]{BarNatanTangles} apply here as well.

\item
In MM14 and MM15, the $d_-$ differentials may be nontrivial depending on the weights. However, because the movies are composed only of Reidemeister 2, birth, death, and saddle moves (and no Reidemeister 3 moves), the induced chain maps coincide exactly with Bar-Natan's maps. As seen in \cite[p.~1481]{BarNatanTangles}, the chain maps for the two movies in each case are actually equal up to a sign, not just homotopic.

\end{itemize}

This completes the proof.
\end{proof}

\section{The Crossing Change map} \label{sec: crossing-change}

In this section, we describe the analogue of the crossing change map of \cite[Section 2.4]{BatsonSeed} in purely local terms. This will be essential for studying the effect of partition-homotopy on the maps described in the previous section.

\subsection{Construction of the maps}

Let $T$ be an oriented, weighted, shaded tangle diagram, and choose a crossing $c$ of $T$. Let $T'$ be the tangle that is the same as $T$, except that the crossing $c$ has been reversed; we may equip $T'$ with the same weighting and shading as $T$. We will define a chain map
\[
CC \co BS(T,w) \to BS(T',w),
\]
which is a chain map in the sense of Definition \ref{def: CKom}. Unlike the cobordism maps from Section \ref{sec: cobordism-maps}, $C$ will not be a filtered map; it may include terms that both increase and decrease the homological grading. If the crossings of $T$ are labeled $c_1, \dots, c_n$ and $c_i$ is the crossing being changed, we may refer to $CC$ as $CC(i)$ if needed.

\begin{definition} \label{def: crossmap}
First we specialize to the case where $T$ consists of a single crossing $c$ between two strands, with no closed components, as in Figure \ref{fig: CC}. Let $w_{\text{over}}$ and $w_{\text{under}}$ be the weights of the overstrand and understrand, respectively. Note that the $0$-resolution of $T$ is identical to the $1$-resolution of $T'$, and vice versa. We define $CC$ to consist of the identity cobordism from $T_0$ to $T'_1$, and $s(c)(w_{\text{over}} - w_{\text{under}})$ times the identity cobordism from $T_1$ to $T'_0$.

In the general case, we define $CC$ by extending the map from the previous (special) case by the identity using Theorem \ref{thm: BS-extend}.
\end{definition}

Figure \ref{fig: CC} illustrates the local crossing change maps (in both directions) for one possible choice of shadings. Each row shows the $BS$ complex of a $1$-crossing oriented tangle, where each arrow corresponds to a saddle cobordism times the ring element indicated (either $1$ or $b-a$). The numbers above and below the resolutions indicate the homological gradings, taking into account the shift of $-n_-$ in the definition. The upward and downward arrows give the entries in the crossing change maps in both directions; each corresponds to the identity cobordism times the ring element indicated. To be precise, if we take $T$ to be the upper tangle and $T'$ to be the lower one, we have $s(c)=1$, $w_{\text{over}} = b$, and $w_{\text{under}} = a$; while if $T$ is the lower tangle and $T'$ is the upper, we have $s(c)=-1$, $w_{\text{over}} = a$, and $w_{\text{under}} = b$; in either case, $s(c)(w_{\text{over}} - w_{\text{under}}) = b-a$. For the reverse shading, we simply replace $b-a$ with $a-b$ in all four places where it appears, by Remark \ref{rmk: change-shading}.

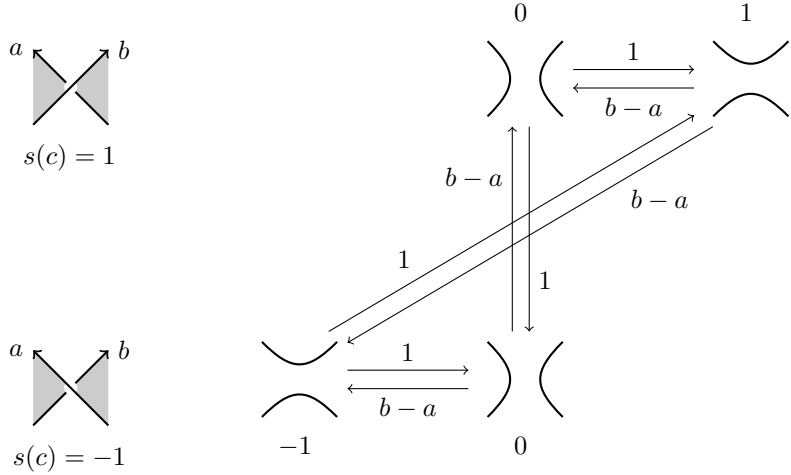
\begin{figure}
\begin{tikzpicture}
\node[label=south:{$s(c)=1$}](poscross) at (-6,0) {
\begin{tikzpicture}
      \filldraw[black!20!white] (0,1) -- (.5,.5) -- (0,0);
      \filldraw[black!20!white] (1,1) -- (.5,.5) -- (1,0);
      \draw [->,thick] (1,0) -- (0,1) node[at end, left] {$a$};
      \node[crossing] at (.5,.5) {};
      \draw [->,thick] (0,0) -- (1,1) node[at end, right] {$b$};
\end{tikzpicture}};

\node[label=south:{$s(c)=-1$}](negcross) at (-6,-4) {
\begin{tikzpicture}
      \filldraw[black!20!white] (0,1) -- (.5,.5) -- (0,0);
      \filldraw[black!20!white] (1,1) -- (.5,.5) -- (1,0);
      \draw [->,thick] (0,0) -- (1,1) node[at end, right] {$b$};
      \node[crossing] at (.5,.5) {};
      \draw [->,thick] (1,0) -- (0,1) node[at end, left] {$a$};
\end{tikzpicture}};

\node[label=north:{$0$}](top0res) at (0,0) {\vres} ;
\node[label=north:{$1$}](top1res) at (3,0) {\hres} ;
\node[label=south:{$-1$}](bottom0res) at (-3,-4) {\hres} ;
\node[label=south:{$0$}](bottom1res) at (0,-4) {\vres} ;
\draw [->] (top0res.10) -- (top1res.170) node[midway, above] {$1$};
\draw [->] (bottom0res.10) -- (bottom1res.170) node[midway, above] {$1$};
\draw [->] (top1res.190) -- (top0res.350) node[midway, below] {$b-a$};
\draw [->] (bottom1res.190) -- (bottom0res.350) node[midway, below] {$b-a$};

\draw [->] (top0res.280) -- (bottom1res.80) node[near end, right] {$1$};
\draw [->] (bottom1res.100) -- (top0res.260) node[near end, left] {$b-a$};

\draw [->] (top1res.235) -- (bottom0res.35) node[near start, below right] {$b-a$};
\draw [->] (bottom0res.55) -- (top1res.215) node[near start, above left] {$1$};
\end{tikzpicture}
\caption{The crossing change map. For the reverse shading, we replace $b-a$ with $a-b$ throughout.}
\label{fig: CC}
\end{figure}

\begin{lemma} \label{lemma: CC-chainmap}
For any tangle $T$, the map $CC$ is a valid chain map in $\CKob$. Moreover, if $w_{\text{over}} - w_{\text{under}}$ is invertible in $R$, then $CC$ is an isomorphism, with inverse given by $s(c) (w_{\text{over}} - w_{\text{under}})^{-1}$ times the reverse crossing change map.
\end{lemma}

\begin{proof}
In the $1$-crossing case, each of these properties can be verified directly from Figure \ref{fig: CC}. Namely, each crossing change map commutes with the total differential, and the composition of the crossing change maps in both directions equals $(b-a)$ times the identity.

The general case then follows by applying the planar algebra structure of $\CKob$.
\end{proof}

\begin{remark} \label{rmk: CC-degree}
Due to the grading shifts of $n_+ - n_-$ in \eqref{eq: BS-chain}, the map $CC$ is of degree $+2$ (with respect to the Batson--Seed grading) when a negative crossing is changed to a positive crossing, and of degree $-2$ when a positive crossing is changed to a negative one. Additionally, note that $CC$ consists of terms of homological degree $0$ and $+2$ (resp. $0$ and $-2$) in those two cases. While $CC$ is a filtered map in the positive-to-negative case, its inverse map is not filtered, so it is not an isomorphism in $\CKob^f$.
\end{remark}

%

\begin{remark}
It is simple to describe the the crossing change map explicitly in the non-local case as well. Note that for each $v \in \{0,1\}^n$, there is a natural identification between the $v$-resolution of $T$ and the $v'$ resolution of $T'$, where $v'$ is obtained by changing the $i\Th$ entry of $v$ from $0$ to $1$ or $1$ to $0$. The morphism $CC(i) \co BS(T) \to BS(T')$ consists of the identity cobordism from $T_v$ to $T'_{v'}$ when $v_i = 0$, and of $s(c) (w_{\text{over}} - w_{\text{under}})$ times the identity cobordism from $T_v$ to $T'_{v'}$ when $v_i=1$. This definition essentially agrees with the one given by Batson and Seed \cite[Proof of Proposition 2.3]{BatsonSeed}, except that the signs are different (indeed, somewhat simplified) because of our modified sign convention, as discussed in Section \ref{ssec: signs}.
\end{remark}

\subsection{Commutation of crossing change maps and cobordism maps}

Because the crossing change maps are defined in purely local terms, the following two lemmas are immediate, using the planar algebra structure of $\CKob$:

\begin{lemma} \label{lemma: CC-commute}
Let $T$ be an oriented, weighted tangle diagram. Let $T'_i$ and $T'_j$ be obtained by respectively changing the $i\Th$ and $j\Th$ crossings of $T$ (where $i \ne j$), and let $T''$ be obtained by changing both crossings. Then the following diagram commutes:
\[
\xymatrix{
BS(T) \ar[r]^{CC(i)} \ar[d]_{CC(j)} & BS(T'_i) \ar[d]^{CC(j)} \\
BS(T'_j) \ar[r]^{CC(i)} & BS(T'')  }
\]

\end{lemma}

\begin{lemma} \label{lemma: CC-elem-commute}
Let $T_0$ and $T_1$ be oriented, weighted tangle diagrams that differ by a single elementary cobordism $e$ (a Reidemeister move, birth, saddle, death). Let $c_i$ be a crossing of $T_0$ that is not contained in the region of interest of $e$, and let $T'_0$ and $T'_1$ be obtained from $T_0$ and $T_1$ by changing $T_i$. Then the following diagram commutes:
\[
\xymatrix{ BS(T_0) \ar[r]^{BS(e)} \ar[d]_{CC(i)} & BS(T_1) \ar[d]^{CC(i)} \\
BS(T'_0) \ar[r]^{BS(e)} & BS(T'_1) }
\]
\end{lemma}

Lemma \ref{lemma: CC-commute} essentially says that distant crossing change maps commute with one another, while Lemma \ref{lemma: CC-elem-commute} says that crossing change maps that occur distant from an elementary cobordism map commute with that elementary cobordism map.

In order to prove our main theorem, we will also need to understand the interaction between Reidemeister 2 and 3 moves and crossing changes that take place within the region of interest. Let $T$ and $T'$ denote the initial and final local tangles of a Reidemeister $k$ move, where $k=2$ or $3$. We may label the strands of $T$ by $s_1,\dots, s_k$, and those of $T'$ by $s_1', \dots, s_k'$, such that $s_i$ corresponds to $s'_i$, and for $i<j$, $s_i$ crosses under $s_j$ and $s'_i$ under $s'_j$.

Let $\sigma$ be a permutation of $\{1,\dots, k\}$. Let $T_\sigma$ and $T'_\sigma$ be obtained from $T$ and $T'$, respectively, by changing the crossing between $s_i$ and $s_j$ (for $i<j$) iff $\sigma(i)>\sigma(j)$. (In other words, we change the order in which the strands are stacked.) It is easy to verify that $T'_\sigma$ then differs from $T_\sigma$ by a Reidemeister move of the same type, which we denote $e_\sigma$.

\begin{lemma} \label{lemma: CC-reid}
Let $T, T', T_\sigma, T'_\sigma$ be as above, and equip all four tangles with compatible orientation, weighting, and shading. Let $w_i$ denote the weight on $s_i$ (and on the corresponding strands in the other tangles), and assume that if $i<j$ and $\sigma(i)>\sigma(j)$, then $w_i - w_j$ is invertible in $R$. Then:
\begin{enumerate}
\item \label{item: CC-reid-iso} The maps $CC \co BS(T) \to BS(T_\sigma)$ and $CC \co BS(T') \to BS(T_{\sigma}')$ obtained by composing the relevant crossing change maps are each isomorphisms in $\CKob$, and they are both homogeneous of the same degree with respect to the Batson--Seed grading.

\item \label{item: CC-reid-commute} The diagram
\begin{equation} \label{eq: CC-reid}
\xymatrix{
BS(T) \ar[r]^{BS(e)} \ar[d]_{CC} & BS(T') \ar[d]^{CC} \\
BS(T_\sigma) \ar[r]^{BS(e_\sigma)}  & BS(T'_\sigma)  }
\end{equation}
commutes up to homotopy and multiplication by a unit.
\end{enumerate}
\end{lemma}

\begin{proof}
For statement \ref{item: CC-reid-iso}, the $CC$ maps are isomorphisms by Lemma \ref{lemma: CC-chainmap} and the hypotheses, since the crossing changes only occur between strands whose differences of weights are invertible. Moreover, the numbers of positive-to-negative and negative-to-positive crossing changes are the same for both $T \to T_\sigma$ and $T' \to T'_\sigma$, so the  degrees of the two $CC$ maps agree by Remark \ref{rmk: CC-degree}. (However, note that they are not filtered isomorphisms, or even filtered homotopy equivalences.)

To prove statement \ref{item: CC-reid-commute}, note that the tangles $T$ and $T'_\sigma$ are each $BS$-simple, by Lemma \ref{lemma: trivial-tangle}. The compositions $BS(e_\sigma) \circ CC$ and $CC \circ BS(e)$ are each homotopy equivalences from $BS(T)$ to $BS(T'_\sigma)$, with the same quantum degree. Therefore, by part 2 of Lemma \ref{lemma: simple-complex}, they are homotopic up to a unit scalar, as required.
\end{proof}

\section{Splitting cobordisms} \label{sec: main-theorem}

Having assembled all the necessary ingredients, we now turn to the proofs of Theorem \ref{thm: main}, Corollary \ref{cor: closed-component}, and Theorem \ref{thm: BS-func}\eqref{item: BS-func-homotopy}. All of these will follow from a more technical statement in the world of tangles.

To begin, we describe a topological operation on tangles and movies that we call \emph{vertical separation}. Let $T$ and $T'$ be oriented tangles with the same boundary, and let $C$ be a tangle cobordism from $T$ to $T'$, represented by a movie $M$ of the form
\[
T = T_0 \xrightarrow{e_1} T_1 \xrightarrow{e_2} \cdots \xrightarrow{e_{r-1}} T_{r-1} \xrightarrow{e_r} T_r = T'.
\]
Assume that we are given a decomposition $C = C^1 \cup \dots \cup C^k$, where each $C^j$ is a possibly disconnected surface. For $i=0, \dots, r$, let $T_i^j$ denote the portion of $T_i$ contained on $C^j$, which we may think of as a tangle diagram in its own right; we refer to $T_i^1, \dots, T_i^k$ as the \emph{parts} of $T_i$. (Some of the $T_i^j$ may be empty.) For each $j$, we also obtain a valid movie $M^j$ of the form
\[
T_0^j \xrightarrow{e_1^j} T_1^j \xrightarrow{e_2^j} \cdots \xrightarrow{e_{r-1}^j} T_{r-1}^j \xrightarrow{e_r^j} T_r^j
\]
(where each $e_i^j$ is either an elementary cobordism or a planar isotopy) by simply erasing all but the $T_i^j$ component of each from of $M$. Let $\tilde T_i$ be the tangle diagram obtained from $T_i$ by changing the crossings between $T_i^j$ and $T_i^{j'}$ (for every pair $j,j'$ with $j<j'$) to make the strand of $T_i^{j'}$ the overstrand at every crossing. Let $\tilde T_i^j$ denote the part of $\tilde T_i$ obtained from $T_i^j$; by ignoring the rest of the diagram, we see that $\tilde T_i^j = T_i^j$.

\begin{figure}
\labellist
 \pinlabel $M$ at 54 5
 \pinlabel $\sep(M)$ at 179 5
 \pinlabel $\splitt(M)$ at 312 5
 \small
 \pinlabel $L_0^1$ [r] at 16 188
 \pinlabel {{\color{red} $L_0^2$}} [l] at 91 188
 \pinlabel $L_1^1$ [r] at 16 136
 \pinlabel {{\color{red} $L_1^2$}} [l] at 91 136
 \pinlabel $L_2^1$ [r] at 16 84
 \pinlabel {{\color{red} $L_2^2$}} [l] at 91 84
 \pinlabel $L_3^1$ [r] at 16 32
 \pinlabel {{\color{red} $L_3^2$}} [l] at 91 32
\endlabellist
\includegraphics{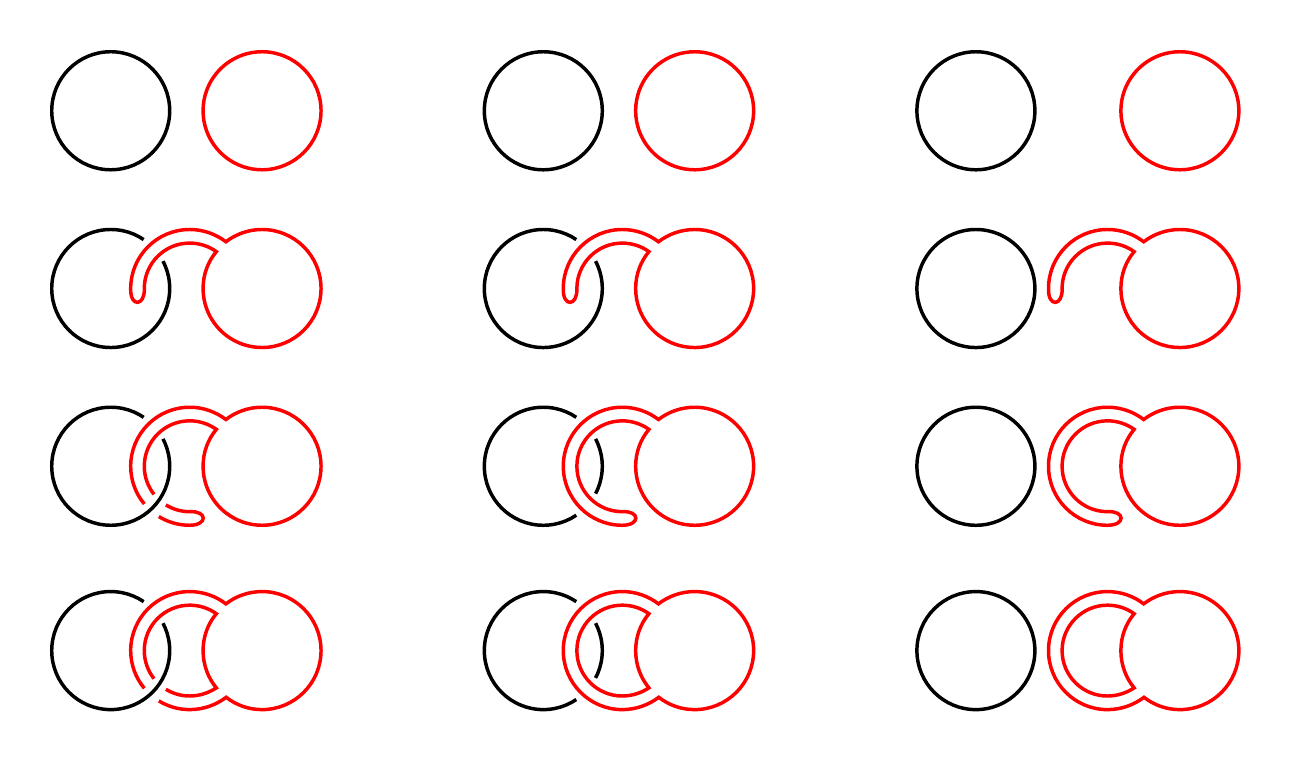}
\caption{A movie $M$ for a link cobordism (left column, consisting of two Reidemeister 2 moves followed by a saddle), and the corresponding movies $\sep(M)$ (center) and $\splitt(M)$ (right).}
\label{fig: splitting-movie}
\end{figure}
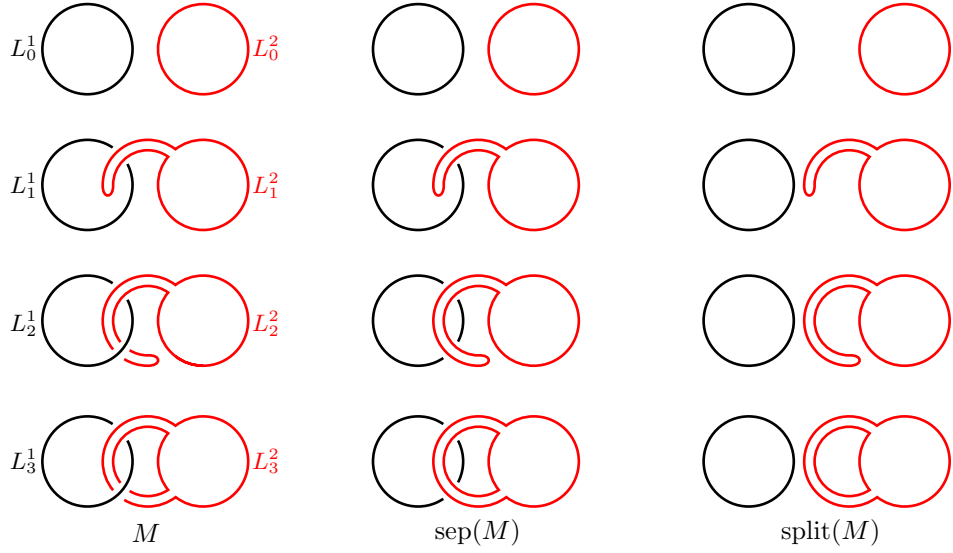

Observe that for each $i=1, \dots, r$, $\tilde T_i$ is obtained from $\tilde T_{i-1}$ by an elementary cobordism $\tilde e_i$ of the same type as $e_i$. This is trivial to see in the case where $e_i$ is a birth, a saddle, or a Reidemeister move that only involves strands that all belong to the same part. In the case of a Reidemeister 2 or 3 move between strands belonging to different parts, the local picture is just like in Lemma \ref{lemma: CC-reid}, where we have simply permuted the heights of the strands belonging to different parts, so $\tilde T_i$ is again obtained from $\tilde T_{i-1}$ by a Reidemeister move. Thus, the movie
\begin{equation} \label{eq: sep(M)-def}
\tilde T := \tilde T_0 \xrightarrow{\tilde e_1} \tilde T_1 \xrightarrow{\tilde e_2} \cdots \xrightarrow{\tilde e_{r-1}} \tilde T_{r-1} \xrightarrow{\tilde e_r} \tilde T_r =: \tilde T'
\end{equation}
is a valid weighted movie, which we denote by $\sep(M)$. See Figure \ref{fig: splitting-movie} for an example.

We may think of our tangles as lying in the $3$-ball $D^2 \times [0,1]$ and our tangle cobordisms as lying in $D^2 \times [0,1] \times [0,1]$, where a tangle diagram is obtained by projection onto the first factor. We use $(x,y,z,t)$ as coordinates. The tangle diagrams $\tilde T_i$ can be lifted into three dimensions in such a way that the lift of $\tilde T_i^j$ lies in $D^2 \times [\frac{j-1}{k}, \frac{j}{k}]$. Likewise, the movie $\sep(M)$ represents a cobordism $\sep(C) = \tilde C^1 \cup \dots \cup \tilde C^k$, where each $\tilde C^j$ lies in $D^2 \times [\frac{j-1}{k}, \frac{j}{k}] \times [0,1]$. Indeed, $\tilde C^j$ is simply obtained by compressing $C^j$ in the $z$ direction to lie within the specified interval.

Our main technical result says that under the right weighting hypotheses, the cobordisms $C$ and $\sep(C)$ induce homotopic maps on Batson-Seed complexes.

\begin{theorem} \label{thm: sep-tangle}
Let $(C,w)$ be a weighted tangle cobordism from $T$ to $T'$, represented by a movie $M$ and equipped with a decomposition $C = C^1 \cup \dots \cup C^k$ as above. Assume that if $w^j$ is the weight of some component of $C^j$ and $w^{j'}$ is the weight of some component of $C^{j'}$ for $j \ne j'$, then $w^j - w^{j'}$ is invertible in $R$. Give each tangle $T_i$ and $\tilde T_i$ the weighting induced from $w$. Then:
\begin{enumerate}
\item \label{item: sep-CCiso}
For each $i=0, \dots, r$, the morphism $CC \co BS(T_i) \to BS(\tilde T_i)$ given as the composition of the necessary crossing change maps is an isomorphism in $\CKob$, and it is homogeneous with respect to the Batson--Seed grading, with degree independent of $i$.

\item \label{item: sep-commute}
Each square in the diagram
\begin{equation} \label{eq: sep(M)-broken-down}
\xymatrix@C=0.6in{
BS(T_0) \ar[r]^{BS(e_1)} \ar[d]_{CC} & BS(T_1) \ar[r]^-{BS(e_2)} \ar[d]_{CC} & \dots \ar[r]^-{BS(e_{r-1})} & BS(T_{r-1}) \ar[r]^{BS(e_r)} \ar[d]_{CC} & BS(T_r) \ar[d]^{CC} \\
BS(\tilde T_0) \ar[r]^{BS(\tilde e_1)}  & BS(\tilde T_1) \ar[r]^-{BS(\tilde e_2)} & \dots \ar[r]^-{BS(\tilde e_{r-1})} & BS(\tilde T_{r-1}) \ar[r]^{BS(\tilde e_r)} & BS(\tilde T_r)
}
\end{equation}
commutes up to homotopy and multiplication by a unit, and therefore the same is true for the composite square
\begin{equation} \label{eq: sep(M)}
\xymatrix@C=0.6in{
BS(T) \ar[r]^{BS(M)} \ar[d]_{CC} & BS(T') \ar[d]^{CC}\\
BS(\tilde T) \ar[r]^{BS(\sep(M))} & BS(\tilde T').
}
\end{equation}
\end{enumerate}
\end{theorem}

\begin{proof}
For statement \ref{item: sep-CCiso}, the hypothesis on the weights together with Lemma \ref{lemma: CC-chainmap} guarantees that each crossing map is an isomorphism. (Note that the order of composition of the individual crossing change maps does not matter, thanks to Lemma \ref{lemma: CC-commute}.) Moreover, it is easy to verify that the signed count of crossings that are changed in going from $T_i$ to $\tilde T_i$ is independent of $i$; together with Remark \ref{rmk: CC-degree}, this shows that the degree of the composite crossing map is independent of $i$ as well.

For statement \ref{item: sep-commute}, in the case where the elementary cobordism $e_i$ is is a birth, a saddle, or a Reidemeister move that only involves strands that all belong to the same $T^j_i$, the square commutes on the nose by Lemma \ref{lemma: CC-elem-commute} since all crossing changes are outside the region of interest. The one remaining case is where $e_i$ is a Reidemeister 2 or 3 move between strands belonging to different $T^j_i$, where some crossings may have changed inside the region of the interest.

Let $\intt(T_{i-1})$, $\intt(T_i)$, $\intt(\tilde T_{i-1})$, and $\intt(\tilde T_i)$ denote the intersections of each of the four tangles with the region of interest, where $\intt(T_{i-1})$ and $\intt(T_i)$ are the before and after pictures of a Reidemeister move $e$. Let $\ext(T_{i-1})$, $\ext(T_i)$, $\ext(\tilde T_{i-1})$, and $\ext(\tilde T_i)$ denote the intersections with a disk that contains all the remaining crossings, so that the complement of the two subtangles forms a planar arc diagram $D$. Then $\ext(T_{i-1})$ and $\ext(T_i)$ are identical, as are $\ext(\tilde T_{i-1})$ and $\ext(\tilde T_i)$, while the former is related to the latter by crossing changes. Each of the subtangles acquires its orientation, weighting, and shading from the larger tangle. Moreover, the $\intt$ tangles are related precisely as in Lemma \ref{lemma: CC-reid}, for some permutation $\sigma$.

By Theorem \ref{thm: BS-local}, we have $BS(T_i) = D(BS(\intt(T_i)), BS(\ext(T_i)))$, and so on. Moreover, the square from \eqref{eq: sep(M)-broken-down} can be broken down as
\[
\xymatrix@C=1in{
D(BS(\intt(T_{i-1})), BS(\ext(T_i))) \ar[r]^{D(BS(e),\id)}  \ar[d]_{D(CC, \id)}  & D(BS(\intt(T_{i})), BS(\ext(T_i))) \ar[d]^{D(CC, \id)} \\
D(BS(\intt(\tilde T_{i-1})), BS(\ext(T_i))) \ar[r]^{D(BS(\tilde e),\id)}  \ar[d]_{D(\id, CC)}  & D(BS(\intt(\tilde T_i)), BS(\ext(T_i))) \ar[d]^{D(\id,CC)} \\
D(BS(\intt(\tilde T_{i-1})), BS(\ext(\tilde T_i))) \ar[r]^{D(BS(\tilde e),\id)}   & D(BS(\intt(\tilde T_i)), BS(\ext(\tilde T_i))).
}
\]
Here, the upper pair of vertical arrows are induced by the crossing changes inside the region of interest, while the lower pair are induced by the remaining crossing changes. The upper square commutes up to homotopy by Lemma
\ref{lemma: CC-reid}, while the lower square commutes on the nose by Lemma \ref{lemma: CC-elem-commute}. This completes the proof.
\end{proof}

We now specialize to the setting of links rather than tangles. Here, there is a further operation to consider: \emph{horizontal splitting}.

Let $\Delta_1, \dots, \Delta_k$ be disjoint disks in $\R^2$. If $L$ is a link diagram with a partition $L = L^1 \cup \dots \cup L^k$ (where the parts $L^j$ may be links, or may even empty), let $\splitt(L)$ be the diagram consisting of a copy of $L^j$ inside of each $\Delta_j$, which is uniquely determined up to planar isotopy. The diagrams $\sep(L)$ and $\splitt(L)$ represent isotopic links and are related by a series of Reidemeister 2 and 3 moves; to see this, we simply slide the parts of $\sep(L)$ apart linearly in generically chosen directions. For any system of weights $w$, we have
\begin{equation} \label{eq: BS-split}
BS(\splitt(L), w) = BS(L^1, w^1) \otimes \dots \otimes BS(L^k, w^k),
\end{equation}
where $w^j$ denotes the restriction of $w$ to $L^j$, and the tensor product represents the planar algebra multiplication on $\Kob(\emptyset)$ coming from the $k$-input planar arc diagram with no strands.

Similarly, if $C = C^1 \cup \dots \cup C^k$ is a partition-preserving cobordism from $L = L^1 \cup \dots \cup L^k$ to $L' = L'{}^1 \cup \dots \cup L'{}^k$, represented by a movie $M$ of the form
\[
L = L_0 \xrightarrow{e_1} L_1 \xrightarrow{e_2} \cdots \xrightarrow{e_{r-1}} L_{r-1} \xrightarrow{e_r} L_r = L'
\]
as above, let $\splitt(M)$ be the movie consisting of a copy of the movie $M^j$ in each $\Delta_j$. (We may also think of $\splitt(M)$ as the \emph{split union} of the movies $M^1, \dots, M^k$, and denote it $M^1 \sqcup \dots \sqcup M^k$.) Again, with the identifications coming from \eqref{eq: BS-split}, the map $BS(\splitt(M))$ (with any choice of compatible weights) decomposes as the product of maps:
\[
BS(\splitt(M)) = BS(M^1) \otimes \dots \otimes BS(M^k).
\]
See Figure \ref{fig: splitting-movie}.

Observe that the movies $\sep(M)$ and $\splitt(M)$ represent isotopic cobordisms in $\R^3 \times I$. Technically, since the initial and final diagrams of $M$ are different diagrams from the initial and final diagrams of $L$, this isotopy is not fixed on the boundary. The precise statement of the isotopy in terms of movies is thus as follows: Let $N$ be the movie consisting of Reidemeister moves from $\sep(L)$ to $\splitt(L)$ as above, followed by $\splitt(M)$ (which runs from $\splitt(L)$ to $\splitt(L')$, followed by Reidemeister moves from $\splitt(L')$ to $\sep(L')$. Then $\sep(M)$ and $N$ are related by a sequence of movie moves.

From the isotopy invariance of maps, we thus deduce:

\begin{lemma} \label{lemma: sep-split}
Let $L$ and $L'$ be link diagrams with decompositions $L = L^1 \cup \dots \cup L^k$ and $L'  = L'{}^1 \cup \dots \cup L'{}^k$, and let $C = C^1 \cup \dots \cup C^k$ be a partition-preserving cobordism from $L$ to $L'$, represented by a movie $M$. Let $w$ be any choice of weights on $M$. Then there are homotopy equivalences making the following diagram commute up to filtered homotopy (up to sign):
\begin{equation} \label{eq: sep-split}
\xymatrix@C=1in{
BS(\sep(L),w) \ar[r]^{BS(\sep(M))} \ar[d]^{\simeq} & BS(\sep(L'),w) \ar[d]^{\simeq} \\
\bigotimes_j BS(L^j,w^j) \ar[r]^{\bigotimes_j BS(M^j)} & \bigotimes_j BS(L'{}^j, w^j).
}
\end{equation}
\end{lemma}

Recall that two partition-preserving cobordisms $C = C^1 \cup \dots \cup C^k$ and $D = D^1 \cup \dots \cup D^k$ are called \emph{partition-homotopic} if $C^j$ is isotopic rel boundary to $D^j$ for all $j$. Observe that if $C$ and $D$ are partition-homotopic, then $\sep(C)$ and $\sep(D)$ are isotopic rel boundary, as are $\splitt(C)$ and $\splitt(D)$.

\begin{proof}[Proof of Theorem \ref{thm: BS-func}(\ref{item: BS-func-homotopy})]
Let $C$ and $D$ be partition-homotopic cobordisms between partitioned links $L_0$ and $L_1$ as in the previous paragraph. Let $w$ be a weighting on $C$, such that all differences of weights between components in different $C^j$ are invertible in $R$, and let $w_0$ and $w_1$ denote the induced weightings on $L_0$ and $L_1$, respectively. The identification between components of $C$ and components of $D$ induces a weighting on $D$, which we also denote by $w$. Applying Theorem \ref{thm: sep-tangle} and Lemma \ref{lemma: sep-split}, we deduce that the diagram
\[
\xymatrix@C=1in{
BS(L_0,w) \ar[r]^{BS(C)} \ar[d]^{CC} & BS(L_1,w) \ar[d]^{CC} \\
BS(\sep(L_0),w) \ar[r]^{BS(\sep(C))} \ar[d]^{\simeq} & BS(\sep(L_1),w) \ar[d]^{\simeq} \\
BS(\splitt(L_0),w) \ar[r]^{BS(\splitt(C))} \ar[d]^{\id} & BS(\splitt(L_1),w) \ar[d]^{\id} \\
BS(\splitt(L_0),w) \ar[r]^{BS(\splitt(D))} \ar[d]^{\simeq} & BS(\splitt(L_1),w) \ar[d]^{\simeq} \\
BS(\sep(L_0),w) \ar[r]^{BS(\sep(D))} \ar[d]^{CC} & BS(\sep(L_1),w) \ar[d]^{CC} \\
BS(L_0,w) \ar[r]^{BS(D)} & BS(L_1,w)
}
\]
commutes up to homotopy and multiplication by units. Applying the Khovanov TQFT followed by the homology functor, we deduce that the maps
\[
\KhBS(C), \KhBS(D) \co \KhBS(L_0,w_0) \to \KhBS(L_1,w_1)
\]
are equal up to multiplication by a unit in $R$, as required.
\end{proof}

Next, we turn to Theorem \ref{thm: main}, concerning partition-preserving cobordisms between split links. As in Definition \ref{def: splitting}, a decomposition $L = L^1 \cup \dots \cup L^k$ is called a \emph{splitting} if $L^1, \dots, L^k$ are contained in disjointly embedded $3$-balls in $\R^3$. Note that this is really a property of an isotopy class of links. Given any splitting, we may perform an ambient isotopy of $\R^3$ (which may be realized diagrammatically by Reidemeister moves) such that the $3$-balls are round and centered at points in the $(x,y)$ plane. The projection of $L$ onto this plane then yields a split diagram for $L$, in which the images of $L^1, \dots, L^k$ are contained in disjoint disks (and hence do not cross).

Let $L_0$ and $L_1$ be links, each equipped with a $k$-part splitting $L_0 = L_0^1 \cup \dots \cup L_0^k$ and $L_1 = L_1^1 \cup \dots \cup L_1^k$ and represented by split diagrams. Let $C = C^1 \cup \dots \cup C^k$ be a partition-preserving cobordism. Because the diagrams are split, we have $L_0 = \tilde L_0$ and $L_1 = \tilde L_1$. Furthermore, if we choose a weighting on $C$ such that for each $j = 1, \dots, k$, every component of $C^j$ (and hence every component of $L^0_j$ and $L^1_j$) has the same weight $w_j$, observe that the $d_-$ differentials on $BS(L_0,w)$ and $BS(L_1,w)$ both vanish identically, since the only crossings are between strands with the same weight.

Theorem \ref{thm: main} will follow from the following lemma:

\begin{lemma} \label{lemma: Kh-split}
Let $L = L^1 \cup \dots \cup L^k$ and $L' = L'{}^1 \cup \dots \cup L'{}^k$ be split link diagrams, and let $C = C^1 \cup \dots \cup C^k$ be a partition-preserving cobordism from $L$ to $L'$, represented by a movie $M$. Then $\Kh(C) \sim \pm \Kh(C^1) \otimes \dots \Kh(C^k)$ (in $\Kob(\emptyset;\Z)$).
\end{lemma}

\begin{proof}
Unlike in the preceding work, where the coefficient ring was arbitrary, we now wish to specialize to the case where $R = \Z$. Since this ring has very few invertible elements (namely $\pm 1$), our choices of weights are quite constrained, so we need to induct on $k$. (Note that $k$ is simply the number of parts in the partition; the cobordisms $C^j$ and the links $L^j$ and $L'{}^j$ need not be connected.)

First, consider the base case $k=2$. Assign the weight $1$ to all components of $C^1$, and assign the weight $0$ to all components of $C^2$. Consider the square \eqref{eq: sep(M)}, with $L$ and $L'$ in place of $T$ and $T'$. By assumption, in the diagram $L$ (resp.~$L'$), there are no crossings between components of $L^1$ and $L^2$ (resp.~$L'{}^1$ and $L'{}^2$). As a result, the vertical maps are simply the identity, and we deduce from Theorem \ref{thm: sep-tangle} that $BS(M) \sim \pm BS(\sep(M))$ in $\CKob(\emptyset; \Z)$. Furthermore, the $d_-$ differentials vanish on both $BS(L)$ and $BS(L')$. By Lemma \ref{lemma: dminus0}(2), we deduce that the homological grading $0$ parts of $BS(C)$ and $BS(\sep(C))$ are homotopic (up to sign) in $\Kob(\emptyset;\Z)$. By construction (see Section \ref{ssec: cobordism-construction}), these are $\Kh(C)$ and $\Kh(\sep(C))$, respectively, as required. And by the isotopy invariance of the Khovanov maps \cite[Theorem 4]{BarNatanTangles}, we have $\Kh(\sep(C)) \sim \pm \Kh(\splitt(C))$. This completes the $k=2$ case.

For the induction, let $C' = C^1 \cup \dots \cup C^{k-1}$. By applying the previous case to the decomposition $C = C' \cup C^k$, we deduce that $\Kh(C) \sim \pm \Kh(C' \sqcup C^k)$. Thus, $\Kh(C) \sim \pm \Kh(C') \otimes \Kh(C^k)$. By induction, we have $\Kh(C') \sim \pm \Kh(C_1) \otimes \Kh(C_{k-1})$, which give the necessary result.
\end{proof}

\begin{proof}[Proof of Theorem \ref{thm: main}]
Let $L_0$ and $L_1$ be link diagrams, and assume that we have \emph{splittings} $L_0 = L_0^1 \cup \dots \cup L_0^k$ and $L_1 = L_1^1 \cup \dots \cup L_1^k$. We do not require $L_0$ and $L_1$ to be split diagrams; we simply assume that there is a sequence $S_i$ of Reidemeister moves taking $L_i$ to a split diagram $L_i'$. Let $C$ and $D$ be partition-preserving cobordisms from $L_0$ to $L_1$, represented by movies $M_C$ and $M_D$ respectively, and assume that $C$ and $D$ are partition-homotopic. Let $M_C'$ (resp.~$M_D'$) be the movie from $L_0'$ to $L_1'$ given consisting of the reverse of $S_0$, followed by $M_C$ (resp.~$M_D$), followed by $S_1$. Lemma \ref{lemma: Kh-split} gives $\Kh(M_C') \sim \pm \Kh(M_D')$ (as morphisms in $\Kob(\emptyset;\Z)$). By pre- and post-composing with the homotopy equivalences induced by $S_0$ and $S_1$, we deduce that $\Kh(M_C) \sim \pm \Kh(M_D)$, again in the abstract sense. We then apply the Khovanov TQFT followed by the homology functor to deduce that the induced maps on Khovanov homology groups are equal up to an overall sign, as required.
\end{proof}

\begin{remark}
In principle, Theorem \ref{thm: main} gives an obstruction to a link $L$ being split: if there exist partition-homotopic cobordisms $C,D:L \to L$ such that the induced maps $\Kh(C)$ and $\Kh(D)$ are not equal up to sign, then $L$ cannot be split. However, there are much more direct ways of detecting splitness using Khovanov homology, namely the results of Lipshitz and Sarkar \cite{LipSar} discussed in the next section.
\end{remark}

We now address Corollary \ref{cor: closed-component}. To begin, suppose $S$ is a smoothly embedded, closed, orientable surface in $\R^3 \times [0,1]$. Viewing $S$ as a cobordism from the empty link to itself, and noting that $\Kh(\emptyset)$ is defined to be $\Z$, the induced map is multiplication by some integer, known as the \emph{Khovanov--Jacobsson number of $S$}, $\KJ(S)$. When $S$ is connected, Rasmussen \cite{RasmussenClosed} and Tanaka \cite{tanaka} showed that $\KJ(S)$ depends only on the genus of $S$: it is $2$ if $g(S)=1$, and $0$ otherwise. Crucially, it does not depend at all on the knotting of $S$. The methods of Tanaka's paper also easily show that for a dotted surface $S^\bullet$, $\KJ(S^\bullet)$ equals $1$ if $g(S)=0$, and $0$ otherwise.

\begin{proof}[Proof of Corollary \ref{cor: closed-component}]
Let $C$ be a cobordism from $L$ to $L'$, and let $S$ be a closed surface (possibly dotted) in the complement of $S$. Let $S'$ denote a copy of $S$, translated to lie in a $4$-ball disjoint from $C$. We may view $C \cup S$ and $C \cup S'$ as a partition-homotopic cobordisms between the split links $L \cup \emptyset$ and $L' \cup \emptyset$. By Theorem \ref{thm: main}, we have
\[
\Kh(C \cup S) \sim \pm \Kh(C \cup S')  = \pm \KJ(S) \Kh(C)
\]
as required.
\end{proof}

Corollary \ref{cor: closed-component} together with Tanaka's results completely determine the Khovanov--Jacobson number of a \emph{surface link}, i.e., a closed surface with multiple components:

\begin{corollary} \label{cor: KJ}
Let $S = S^1 \cup \dots \cup S^k$ be a surface link in $\R^3 \times [0,1]$. Then
\[
\KJ(S) = \pm \prod_j \KJ(S^j) =
\begin{cases}
  \pm 2^k & \text{if every component $S^j$ is a torus} \\
  0 &\text{otherwise}.
\end{cases}
\]
\end{corollary}

\begin{proof}
Induct using Corollary \ref{cor: closed-component}.
\end{proof}

\begin{remark}
We note that Theorem \ref{thm: main} generalises to split tangles and partition-preserving tangle cobordisms between them, where by a \emph{split tangle} we mean a tangle whose connected components can be partitioned in a way that distinct partitions can be split from each other by the use of Reidemeister moves. The proof is identical to the proof of Theorem \ref{thm: main}, and it works for the simple reason that the proof of Theorem \ref{thm: main} works at the level of abstract complexes in $\CKob_{/h}$, without passing to homology groups.
%
%
\end{remark}

\section{Applications to ribbon concordance} \label{sec: SHR}

We now turn to the original motivation for this project: applications of Khovanov homology to ribbon concordance and the more general notion of strongly homotopy-ribbon concordance.

When working with maps on Khovanov homology, we must \emph{a priori} take some care to distinguish between cobordisms in $\R^3 \times I$ and in $S^3 \times I$. We identify $S^3$ with $\R^3 \cup \{\infty\}$. While a generic link cobordism in $S^3 \times I$ can be arranged to avoid $\{\infty\} \times I$, a generic isotopy of cobordisms might not; that is, two cobordisms in $\R^3 \times I$ may be isotopic in $S^3 \times I$ but not in $\R^3 \times I$. The original proofs of isotopy invariance for the maps on Khovanov homology \cite{KhovanovCobordism, JacobssonCobordisms, BarNatanTangles} apply only to isotopy in $\R^3 \times I$. However, Morrison, Walker, and Wedrich \cite{MorrisonWalkerWedrich} have recently shown that isotopy invariance also holds in $S^3 \times I$.

The key topological fact for the proof of Theorem \ref{thm: SHR} is the following lemma, which is essentially due to Miller and Zemke \cite{MillerZemkeHomotopyRibbon}.

\begin{lemma} \label{lemma: MZ}
Let $C \subset S^3 \times I$ be a strongly homotopy-ribbon concordance from $L_0$ to $L_1$. Let $\overline C$ denote the mirror of $C$, viewed as a concordance from $L_1$ to $L_0$, and let $D$ denote the composite cobordism $C \cup _{L_1} \overline C$, viewed as a concordance from $L_0$ to itself. Then there exist disjointly embedded $2$-spheres $S_1, \dots, S_k \subset S^3 \times I \minus D$, and disjointly embedded $3$-dimensional $1$-handles $h_1, \dots, h_k$, where $h_i$ connects $D$ to $S_i$ and is otherwise disjoint from $D \cup S_1 \cup \dots \cup S_k$, such that the surface $D'$ obtained from $D \cup S_1 \cup \dots \cup S_k$ by surgery along $h_1, \dots, h_k$ is isotopic (rel boundary) to $L_0 \times I$.
\end{lemma}

\begin{proof}
By \cite[Lemma 3.1]{MillerZemkeHomotopyRibbon}, the complement $S^3 \times I \minus \nbd(D)$ is built from $(S^3 \minus L_0) \times I$ by attaching the following handles:
\begin{enumerate}
\item $n$ $1$-handles;
\item $n$ $2$-handles $H_1, \dots, H_n$;
\item $n$ more $2$-handles $H'_1, \dots, H'_n$, attached along the belt circles of $H_1, \dots, H_n$ respectively;
\item $n$ $3$-handles.
\end{enumerate}
(This is stated only for concordances of knots in \cite{MillerZemkeHomotopyRibbon}, but the proof is identical in the case of links.) For $i=1, \dots, n$, let $\Sigma_i$ be the sphere obtained as the union of the cocore of $H_i$ and the core of $H_i'$. These spheres have trivial self-intersection since they are contained in $S^3 \times I$.

As in the proof of \cite[Theorem 1.2]{MillerZemkeHomotopyRibbon}, $D$ can be transformed into the identity cobordism by isotopy and tubing with the spheres $\Sigma_i$ --- that is, performing surgery along an embedded $3$-dimensional $1$-handle joining $D$ to $\Sigma_i$. Note that $D$ may be tubed with $\Sigma_i$ more than once (say $k_i$ times) during this process. In other words, we take $k_i$ parallel pushoffs of $\Sigma_i$ (which are disjoint since $\Sigma_i$ has trivial self-intersection) and perform each successive tubing operation with a separate copy. These spheres, taken over all $i=1, \dots, n$ are the $S_1, \dots, S_k$ called for in the lemma.
\end{proof}

\begin{proof}[Proof of Theorem \ref{thm: SHR}]
Let $C \subset S^3 \times I$ be a strongly homotopy-ribbon concordance from $L_0$ to $L_1$. Let $\overline C$, $D$, $S_1, \dots, S_k$, $h_1, \dots, h_k$, and $D'$ be as in Lemma \ref{lemma: MZ}. Applying the neck-cutting relations to the tubes in $D'$, we have
\[
\Kh(D') =  \sum_{\vec e \in \{\varnothing, \bullet\}^k} \pm \Kh(D \cup S_1^{e_1} \cup \dots \cup S_k^{e_k} )
\]
for some choices of signs, where $S_i^{\varnothing}$ and $S_i^{\bullet}$ denote $S_i$ without or with a dot, respectively. Applying Corollary \ref{cor: closed-component}, we have:
\begin{align*}
\Kh(D') &=  \pm \Kh(D \cup S_1^\bullet \cup \dots \cup S_k^\bullet) \\
&=  \pm \Kh(D).
\end{align*}
At the same time, we have
\[
\Kh(D') = \pm \Kh( \id_{L_0}) = \pm \id_{\Kh(L_0)}
\]
by isotopy invariance. Thus,
\[
\Kh(\overline{C}) \circ \Kh(C) = \pm \id_{\Kh(L_0)},
\]
which shows that $\Kh(C)$ is injective and, indeed, left-invertible.
\end{proof}

We now turn to the proof of Theorem \ref{thm: ribbon-split}, which is a purely topological statement. To begin, let $\F$ denote the field of two elements, and let $R = \F[X,Y]/(X^2,Y^2)$. Given a link $L \subset S^3$ and points $p, q \in L$, $\Kh(L;\F)$ acquires the structure of an $R$-module, where the action of $X$ (resp.~$Y$) is given by the action of the product cobordism $L \times I$ with a dot on the component containing $p$ (resp.~$q$). Moreover, if $C$ is a link cobordism from $L_0$ to $L_1$, and $p_0, q_0 \in L_0$ and $p_1, q_1 \in L_1$ are chosen such that $p_0$ and $p_1$ (resp.~$q_0$ and $q_1$) are on the same component of $C$, then the induced map $\Kh(C) \co \Kh(L_0;\F) \to \Kh(L_1;\F)$ is $R$-linear.

A recent theorem of Lipshitz and Sarkar \cite[Corollary 1.3]{LipSar} shows that the $R$-module structure of $\Kh(L;\F)$ detects link splitting, in the following sense: there exists an embedded $2$-sphere in $S^3 \minus L$ separating $p$ from $q$ if and only if $\Kh(L;\F)$ is a free $R$-module, where $R$ acts as above.

\begin{proof}[Proof of Theorem \ref{thm: ribbon-split}]
Suppose $C$ is a strongly homotopy-ribbon concordance from $L_0$ to $L_1$, and suppose there is a $2$-sphere in $S^3 \minus L_1$ separating components $L_1^i$ and $L_1^j$. Let $L_0^i$ and $L_0^j$ be the corresponding components of $L_0$. Choose points $p_0 \in L_0^i$, $q_0 \in L_0^j$, $p_1 \in L_1^i$, and $q_1 \in L_1^j$, and equip $\Kh(L_0)$ and $\Kh(L_1)$ with the corresponding $R$-module structures; then $\Kh(L_1;\F)$ is free over $R$. The map $\Kh(C)$ is $R$-linear and has a left inverse (which is also $R$-linear) by Theorem \ref{thm: SHR}, so $\Kh(L_0;\F)$ injects as a direct summand of $\Kh(L_1;\F)$ and is therefore a projective $R$-module. Since $R$ is a local ring (with unique maximal ideal $(X,Y)$), every projective module is free, and thus $\Kh(L_0;\F)$ is free over $R$. Hence, by \cite[Corollary 1.3]{LipSar}, there is an embedded $2$-sphere in $S^3 \minus L_0$ separating $p_0$ from $q_0$, as required.
\end{proof}

\begin{remark} \label{rmk: ribbon-split-HF}
We may give an alternate proof of Theorem \ref{thm: ribbon-split} using Heegaard Floer homology. Given a link $L$ and points $p,q \in L$, the Heegaard Floer homology of the branched double cover of $L$, $\HF(\Sigma(L))$, acquires the structure of an $\F[X]/(X^2)$--module. Building on earlier work of Hedden--Ni \cite{HeddenNiUnlink} and Alishahi--Lipshitz \cite{AlishahiLipshitzSurfaces}, Lipshitz and Sarkar \cite{LipSar} proved that $\HF(\Sigma(L))$ is a free $\F[X]/(X^2)$--module if and only if $p$ and $q$ are separated by a $2$-sphere in $S^3 \minus L$. They then used this result along with a careful analysis of the Ozsv\'ath--Szab\'o spectral sequence from $\Khred(L)$ to $\HF(\Sigma(\overline L))$ \cite{OSzDouble} to deduce the analogous detection result for Khovanov homology (both reduced and unreduced), as mentioned above.

Now, if $L_0$ is ribbon concordant to $L_1$, then a result of Daemi, Lidman, Vela-Vick, and Wong \cite[Theorem 1.19]{DaemiLidmanVelaVickWong} shows that $\HF(\Sigma(L_0))$ injects into $\HF(\Sigma(L_1))$ as a summand, and it is not hard to see that this holds at the level of $\F[X]/(X^2)$--modules as well. The proof of Theorem \ref{thm: ribbon-split} then follows along the same lines as above.
\end{remark}

%
%
%
%
%
%

%

\bibliography{GL-bibliography}
\bibliographystyle{amsalpha}

\end{document}